\documentclass{amsart}
\usepackage{amsthm}
\usepackage{amsmath}
\usepackage{amssymb}
\usepackage{enumerate}
\usepackage{enumitem}
\usepackage{mathrsfs}
\usepackage{float}
\usepackage{graphicx}
\usepackage{mathtools}
\usepackage{hyperref}
\usepackage{caption}
\usepackage{subcaption}
\usepackage{faktor}
\usepackage[dvipsnames]{xcolor}
\usepackage{lineno}

\newtheorem{theorem}{Theorem}

\newtheorem{proposition}[theorem]{Proposition}
\newtheorem{lemma}[theorem]{Lemma}
\newtheorem{definition}[theorem]{Definition}
\newtheorem{example}[theorem]{Example}
\newtheorem{remark}[theorem]{Remark}
\newtheorem{question}{Question}

\numberwithin{theorem}{section}
\numberwithin{equation}{section}

\DeclareMathOperator{\sech}{sech}

\newcounter{relctr} 
\everydisplay\expandafter{\the\everydisplay\setcounter{relctr}{0}} 

\AtBeginDocument{} 

\makeatletter
\newcommand{\mylabel}[2]{#2\def\@currentlabel{#2}\label{#1}}
\makeatother

\allowdisplaybreaks

\begin{document}
	\title[Weighted $k$-plane transform on real Hyperbolic space]{On certain mapping properties of a weighted totally geodesic $k$-plane transform on the real hyperbolic space}
    \author[P Datta]{Pratyay Datta}
    \address{Department of Mathematics, Indian Institute of Technology Indore, Khandwa Road, Simrol, Indore-453552, Madhya Praesh, India.}
    \email{pratyaydatta2013@gmail.com; phd2501241013@iiti.ac.in}
	\author[A Deshmukh]{Aniruddha Deshmukh}
    \address{Department of Mathematics, Indian Institute of Technology Indore, Khandwa Road, Simrol, Indore-453552, Madhya Praesh, India.}
    \email{aniruddha480@gmail.com; phd2001241001@iiti.ac.in}
	\author[A Kumar]{Ashisha Kumar}
    \address{Department of Mathematics, Indian Institute of Technology Indore, Khandwa Road, Simrol, Indore-453552, Madhya Praesh, India.}
    \email{akumar@iiti.ac.in}
	\date{\today}
	\begin{abstract}
		In this article, we study a weighted analogue of the totally geodesic $k$-plane transform on the real hyperbolic space. Throughout the article, we are concerned with the mapping properties of this transform. Particularly, we first consider the existence of the operator on weighted Lebesgue spaces, where the weights under consideration are natural to the geometry of the hyperbolic space. Then, we proceed to see the $L^p$-improving mapping properties of the said transform on radial functions. Further, we also prove end-point boundedness results on Lorentz spaces of radial functions. Lastly, we look into the problem of weighted $L^p$-$L^p$ boundedness, and obtain partial results in this direction. In the way of obtaining these results, we also give a spherical integral representation of the weighted totally geodesic $k$-plane transform.
	\end{abstract}
	\maketitle
	\section{Introduction}
		\label{IntroductionSection}
		The Radon transform (more generally, the $k$-plane transform) on the Euclidean space is an important operator from tomography. Named after the Austrian Mathematician Johann Radon, the Radon transform integrates the restriction of a function on the Euclidean space over lower dimensional hyperplanes. The Radon transform forms the basis of computer assisted tomography (CAT), an important tool in medical imaging. For a brief survey on the importance of Euclidean Radon transforms, we refer the reader to Chapter 1 of \cite{Markoe} and Section 12.1 and its bibliographical notes on p.508 of \cite{RubinBookFIR}. Formally, given a function $f: \mathbb{R}^n \rightarrow \mathbb{C}$, its $k$-plane transform is defined as a function $R_kf: G \left( n, k \right) \rightarrow \mathbb{C}$, as
		$$R_kf \left( \xi \right) = \int\limits_{\xi} f \left( x \right) \mathrm{d}_{\xi}x,$$
		where, $G \left( n, k \right)$ is the affine Grassmann manifold for all (affine) $k$-dimensional planes in $\mathbb{R}^n$, and $d_{\xi}x$ is the $k$-dimensional Lebesgue measure on the plane $\xi$.
		
		While the Euclidean Radon transforms are important in themselves, many researchers have also focused on defining and studying these transforms on non-Euclidean spaces. To do so, however, one would first need to generalize the notion of ``planes" to non-Euclidean spaces. This may be achieved in two ways: one, by considering that the space has a Riemannian structure, and then considering the totally geodesic submanifolds as the ``planes"; another, by considering the horocycles to be the planes of the ambient space, in the setting of symmetric spaces.
		
		In this article, we are focused on the real hyperbolic space, famously denoted by $\mathbb{H}^n$, which is a Riemannian symmetric space. Consequently, one can talk about both the variants (totally geodesic and horocyclic) of Radon transforms on $\mathbb{H}^n$. For the study of the horocyclic (also referred to as ``horospherical") Radon transforms on $\mathbb{H}^n$, we refer the reader to Section 5.2 (p. 290--304) of \cite{GelfandBook}, Section 4C (p. 68--72) of Chapter 2 of \cite{HelgasonIGRTBook} and Section 6.3 (p. 416--443) of \cite{RubinBookRT}. Recently, in 2019, Bray and Rubin in \cite{BrayRubin} have also considered the lower dimensional version of the horospherical Radon transforms on $\mathbb{H}^n$. We would also like to refer the reader to the notes and comments on Section 12.3 (see p. 510) of \cite{RubinBookFIR}. However, in this article, we are not concerned with the horocyclic transforms.
		
		The totally geodesic Radon transforms on the real hyperbolic space are also studied for a long time. These transforms were introduced by Helgason in \cite{HelgasonDO}. For a brief history (and theory) of the development of such transform (at least for the codimension $1$ case), we refer the reader to Section 6.2 of \cite{RubinBookRT} and the bibliographical notes for these section (see p. 443 of the book \cite{RubinBookRT}). Formally, we consider the set $\Xi_k \left( \mathbb{H}^n \right)$ of all $k$-dimensional totally geodesic submanifolds in $\mathbb{H}^n$ as a substitute for $G \left( n, k \right)$. Now, given a function $f: \mathbb{H}^n \rightarrow \mathbb{C}$, its totally geodesic $k$-plane transform is defined as a function $R_kf: \Xi_k \left( \mathbb{H}^n \right) \rightarrow \mathbb{C}$, given by
		\begin{equation}
			\label{KPlaneTransformFormalDefinition}
			R_kf \left( \xi \right) := \int\limits_{\xi} f \left( x \right) \mathrm{d}_{\xi}x,
		\end{equation}
		where, $\mathrm{d}_{\xi}x$ is the induced Riemannian measure on the submanifold $\xi$.
		
		The motivation for the study conducted in this article lies in the work \cite{BerensteinRubin} of Berenstein and Rubin. In their article, the authors have defined a Riesz-type potential operator (see Equation (4.50) on p. 49 of \cite{BerensteinRubin}) for the purposes of inverting the totally geodesic Radon transform. One may use this definition to define a Riesz-type potential operator related to the totally geodesic Radon transform. Particularly, one may consider the operator
		\begin{equation}
			\label{RieszPotentialRTDefinition}
			I_{\alpha}f \left( x, \xi \right) = \int\limits_{\xi} f \left( y \right) \sinh^{\alpha - k} d \left( x, y \right) \mathrm{d}y,
		\end{equation}
		where, $x \in \mathbb{H}^n$ and $\xi \in \Xi_k \left( \mathbb{H}^n \right)$ is chosen so that $x \in \xi$. A Euclidean analogue of the operator given in Equation \eqref{RieszPotentialRTDefinition} was studied by Duoandikoetxea et al. in \cite{Duoandikoetxea}. While the description of the Euclidean analogue is easy due to the presence of a group structure on $\mathbb{R}^n$, the same seems to be difficult in non-Euclidean spaces. Therefore, instead of studying a Riesz-type potential operator related to the Radon transforms, we focus on a non-divergent version of the same. That is, we consider for $\xi \in \Xi_k \left( \mathbb{H}^n \right)$,
		\begin{equation}
			\label{OurOperator}
			R_{\alpha}f \left( \xi \right) := \int\limits_{\xi} f \left( x \right) \sinh^{\alpha - k} d \left( x, u_{\xi} \right) \mathrm{d}x,
		\end{equation}
		where, $u_{\xi} \in \xi$ is the point closest to the ``origin" of $\mathbb{H}^n$. These terminologies shall be made clear in Section \ref{HyperbolicSpacePreliminariesSection}. While the operator mentioned in Equation \eqref{OurOperator} can be termed as ``non-divergent Riesz type potential", we call it a ``weighted analogue of the totally geodesic $k$-plane transform". This is because the term ``potential type operator" could give an impression of the operator being of convolution-type on the $k$-dimensional totally geodesic submanifolds. However, this is not the case. The operator defined in Equation \eqref{OurOperator} introduces a decay (or amplification) in the measure (hence the mass) of the submanifolds as one move away from (or close to) the point $u_{\xi}$, which may be treated as a point of singularity.
		
		Our goal in this article is to study the mapping properties of the operator $R_{\alpha}$ on weighted Lebesgue spaces. First, we notice that for $\alpha = k$, the operator reduces to the ordinary totally geodesic $k$-plane transform on $\mathbb{H}^n$, and a similar study for this operator was conducted by Deshmukh and Kumar in \cite{SelfLPImproving}. The results of this article, therefore, generalize their results.
		
		This article is organised as follows: in Section \ref{HyperbolicSpacePreliminariesSection}, we give certain preliminaries about the real hyperbolic space $\mathbb{H}^n$ that are needed throughout the study. Particularly, we emphasise on various models of the space $\mathbb{H}^n$, descriptions of the measure on $\mathbb{H}^n$, the set $\Xi_k \left( \mathbb{H}^n \right)$ and its measure. Towards the end of this section, we formally define weighted Lebesgue and Lorentz spaces which shall be the center of the study of mapping properties. The weights that we consider come naturally from the geometry of the hyperbolic space. Next, in Section \ref{TransfosmPreliminariesSection}, we give the formal definition of the transform $R_{\alpha}$ and its dual. The main results of this section are the one-dimensional (integral) representations of these transforms when they act on radial functions. In Section \ref{ExistenceSection}, we get optimal conditions on the values of $p$ and the weights for which the transform $R_{\alpha}$ is well-defined on weighted Lebesgue spaces. Next, in Sections \ref{LPImprovingSection} and \ref{EndPointEstimateSection}, we discuss the mapping properties of the transform $R_{\alpha}$ when it acts on radial functions. First, we consider the problem of $L^p$-improving mapping properties on Lebesgue spaces, and then we move on to the end-point behaviour. Lastly, in Section \ref{LPLPBoundednessSection}, we give certain sufficient conditions for the operator $R_{\alpha}$ to be bounded on weighted $L^p$ spaces. In this process, we also obtain a representation for the operator $R_{\alpha}$ when it acts on generic (not necessarily radial) functions. Finally, in Section \ref{ConclusionSection}, we conclude this article and give a few future directions one can tread on.
		
		Throughout this article, we are focused upon obtaining estimates for the operator $R_{\alpha}$. Consequently, at most places we do not care about constants appearing in the estimates. We shall write $C$ to mean a constant that might change at every step and depends on the parameters, but not on the functions. Also, we use the phrase ``$k$-plane transform" to mean the totally geodesic $k$-plane transform.
	\section{The real hyperbolic space \texorpdfstring{$\mathbb{H}^n$}{}}
		\label{HyperbolicSpacePreliminariesSection}
		In this section, we first recall various definitions and results related to the real hyperbolic space. The real hyperbolic space\footnote{Throughout this article, we use the phrase ``hyperbolic space" to mean real hyperbolic space.}, denoted by $\mathbb{H}^n$, can be described by various geometric models (see, for instance, p. 69--72 of \cite{Cannon}). Out of the various model available to describe the hyperbolic space, in this article, we consider the Poincar\'{e} ball model or the upper sheet of hyperboloid model, whenever we wish to look at the geometry of the space. We quickly give certain insights into these models.
		
		In the Poincar\'{e} Ball model, we consider $\mathbb{H}^n = B \left( 0, 1 \right)$, the open unit ball in $\mathbb{R}^n$, with a Riemannian structure given by
		$$\mathrm{d}x^2 = \frac{4\mathrm{d}u^2}{\left( 1 - \| u \|^2 \right)^2},$$
		where, $\mathrm{d}u^2$ is the Euclidean Riemannian structure of $B \left( 0, 1 \right)$. For more details on this Riemannian structure, we refer the reader to Theorem 3.7 of \cite{LeeRM}. With this Riemannian structure, the measure on $\mathbb{H}^n$ is given by
		$$\mathrm{d}x = \frac{2 \mathrm{d}u}{\left( 1 - \| u \|^2 \right)}.$$

		Similarly, the upper sheet of hyperboloid model of $\mathbb{H}^n$ is described by the following set.
		$$\left\lbrace x \in \mathbb{R}^{n + 1} \Bigg| \sum\limits_{i = 1}^{n} x_i^2 - x_{n + 1}^2 = - 1 \text{ and } x_{n + 1} > 0 \right\rbrace.$$
		The Riemannian structure on $\mathbb{H}^n$, in this model, is the one induced from the pseudo-Riemannian structure of $\mathbb{R}^{n + 1}$ equipped with the Lorentz inner product $\left[ x, y \right] = \sum\limits_{i = 1}^{n} x_i^2 - x_{n + 1}^2$. The particular description of measure in this model is not of use to us, and therefore, we do not describe it here.
		
		From a viewpoint of Harmonic Analysis, we see the real hyperbolic space as a quotient of certain Lie groups. For the purpose of this article, we use $G = SO_0 \left( n, 1 \right)$ as the connected component of the group of all special Lorentz transformations on $\mathbb{R}^{n + 1}$ with signature $1$ (see p. 58 of \cite{Ratcliffe} for more details). It is known that the group $G$ acts transitively on $\mathbb{H}^n$ (when considered as the upper sheet of hyperboloid) through isometries. We also use $K = SO \left( n \right)$ to denote the subgroup of $G$ that contains all rotations of $\mathbb{R}^n$ (by fixing $e_{n + 1}$), and $H = SO \left( n - k \right) \times SO_0 \left( k, 1 \right)$ as the isotropy subgroup of $\mathbb{H}^n \cap \mathbb{R}^{k + 1}$. Here, we consider $\mathbb{R}^{k + 1} \hookrightarrow \mathbb{R}^n$ by fixing the first $\left( n - k \right)$-coordinates zero. Then, it is known that $K$ is the isotropy group for $e_{n + 1} \in \mathbb{H}^n$, and $H$ is the isotropy group of $\mathbb{H}^n \cap \mathbb{R}^{k + 1}$. Further, we have $\mathbb{H}^n = \faktor{G}{K}$.
		
		In the quotient space description, we normalise the invariant measure on $K$ so that we have (see Theorem 2.51 of \cite{FollandAHA}),
		$$\int\limits_{\mathbb{H}^n} f \left( x \right) \mathrm{d}x = \int\limits_{G} \tilde{f} \left( gK \right) \mathrm{d}g = \int\limits_{G} f \left( g \cdot e_{n + 1} \right) \mathrm{d}g,$$
		where, $\mathrm{d}g$ is the left-invariant measure on $G$, and $\tilde{f}$ is the lift of $f$ to $G$.
		
		Also, since the real hyperbolic space is a manifold with constant sectional curvature (for details, we refer the reader to Theorem 8.34 of \cite{LeeRM}), the measure on $\mathbb{H}^n$ may be decomposed into a radial part and an angular part. This is famously known as the polar decomposition, and is an essential tool in our analysis. The result we now state is found in \cite{LeeRM} on p.295 (Corollary 10.17). We state it in a manner convenient to our use.
		\begin{proposition}[\cite{LeeRM}]
			\label{PolarDecompositionTheorem}
			Let $f: \mathbb{H}^n \rightarrow \mathbb{C}$ be a function that is integrable. Then, for any point $u \in \mathbb{H}^n$, we have,
			\begin{equation}
				\label{PolarDecompositionEquation}
				\int\limits_{\mathbb{H}^n} f \left( x \right) \mathrm{d}x = \int\limits_{\mathbb{S}^{n - 1}} \int\limits_{0}^{\infty} f \left( \omega, t \right) \sinh^{n - 1} t \ \mathrm{d}t \mathrm{d}\omega,
			\end{equation}
			where, $\left( \omega, t \right)$ are the (polar) coordinates on the tangent space with respect to the exponential map $\exp_u: T_u\mathbb{H}^n \rightarrow \mathbb{H}^n$. That is, the symbol $f \left( \omega, t \right)$ is to be understood as $f \left( \exp_u \left( t \omega \right) \right)$.
		\end{proposition}
		
		In this article, we deal with the totally geodesic submanifolds of $\mathbb{H}^n$. We begin by recalling the definition in the setting of general Riemannian manifolds (see Proposition 8.12 of \cite{LeeRM}).
		\begin{definition}[totally geodesic submanifolds]
			\label{TGSDefinition}
			Let $M$ be a Riemannian manifold. A subset $S \subseteq M$ is totally geodesic if every geodesic in $S$ is a geodesic in $M$.
		\end{definition}
		Throughout this article, we denote the collection of all $k$-dimensional totally geodesic subanifolds of $\mathbb{H}^n$ by $\Xi_k \left( \mathbb{H^n} \right)$. The following fact about totally geodesic submanifolds is crucial for our considerations, and can be found in \cite{LeeRM}.
		\begin{theorem}
			\label{FootOfPerpendicularTGS}
			Let $M$ be a complete Riemannian manifold and $S \subseteq M$ be a totally geodesic submanifold. For any $p \in M$, there is a point $q \in S$ such that $d \left( p, q \right) = d \left( p, S \right)$. Moreover, the geodesic joining $p$ and $q$ is orthogonal to $S$ at $q$.
		\end{theorem}
		\begin{remark}
			\normalfont
			Theorem \ref{FootOfPerpendicularTGS} gives the existence of a ``foot of the perpendicular" from a given point $p$ to a totally geodesic submanifold $S$. It can be proved that in the real hyperbolic space $\mathbb{H}^n$, the foot of the perpendicular is unique. To see this, assume that there are two distinct points $q_1$ and $q_2$ in $S$ that realise the distance with $p$. Then, the geodesic triangle formed by the vertices $p$, $q_1$ and $q_2$ would have two right-angles, which is impossible in the hyperbolic space (see Corollary 1 on p. 86 of \cite{Ratcliffe}).
			
			For a given submanifold $\xi \in \Xi_k \left( \mathbb{H}^n \right)$, we denote by $u_{\xi}$ the point closest to $0 \in \mathbb{H}^{n}$. Here, by the symbol `$0$', we mean the $0$ of $\mathbb{R}^n$ when the real hyperbolic space is described by the Poincar\'{e} disc model, and $e_{n + 1} \in \mathbb{R}^{n + 1}$, when $\mathbb{H}^n$ is described by the upper sheet of hyperboloid. We must remark here that since $\mathbb{H}^n$ is a homogeneous and isotropic space, any point may be considered as its ``origin". Our particular choices are solely for the purposes of computational convenience.
		\end{remark}
		An explicit description of the totally geodesic submanifolds of $\mathbb{H}^n$ can be obtained by fixing one of the models of the real hyperbolic space. The next result gives a summary of such descriptions, and the details can be found in \cite{LeeRM} (Exercise 8--31 on p. 262).
		\begin{proposition}
			\label{TGSHnDescription}
			~
			\begin{enumerate}
				\item When $\mathbb{H}^n$ is described by the upper sheet of hyperboloid, any $k$-dimensional totally geodesic submanifold is obtained by intersecting $\left( k + 1 \right)$-dimensional subspace of $\mathbb{R}^{n + 1}$ with $\mathbb{H}^n$.
				\item When $\mathbb{H}^n$ is described by the Poincar\'{e} ball model, any $k$-dimensional totally geodesic submanifold is realised as a $k$-dimensional spherical cap inside the open unit ball that intersects the boundary of the unit ball in $\mathbb{R}^n$ orthogonally.
			\end{enumerate}
		\end{proposition}
		\begin{remark}
			\label{SmallSubmanifoldsSame}
			\normalfont
			In any description of the real hyperbolic space, a $k$-dimensional totally geodesic submanifold is again a real hyperbolic space. Consequently, the tools of polar decomposition can be utilised easily on such submanifolds by fixing any point on the submanifold as the ``origin".
		\end{remark}
		
		The measure on the set $\Xi_k \left( \mathbb{H}^n \right)$ can be introduced by considering it as the quotient space $\faktor{G}{H}$. The measure can then be represented in various ways, the details of which can be found in Lemma 1 (Equation (2.22)) of \cite{BerensteinRubin}. We mention two of the representations that are used in this article. First, we fix a positive function $f_1$ on $G$ such that for every $g \in G$, we have,
		$$\bar{f}_1 \left( g \right) := \int\limits_{H} f_1 \left( g \gamma \right) \mathrm{d}\gamma < +\infty.$$
		Let $\pi: G \rightarrow \faktor{G}{H}$ be the canonical projection and $\varphi: \Xi_k \left( \mathbb{H}^n \right) \rightarrow \mathbb{C}$ be a ``nice\footnote{By a ``nice" function, we mean a function for which the integrals that appear exist in the Lebesgue sense. We use this seemingly vague terminology throughout our article, especially when we give definitions. In the main results we shall give explicit classes of functions that are ``nice".}" function. We define $\Phi: G \rightarrow \mathbb{C}$ as
		$$\Phi \left( g \right) = \frac{\left( \varphi \circ \pi \right) \left( g \right) f_1 \left( g \right)}{\bar{f}_1 \left( g \right)}.$$
		Then, the measure on $\Xi_k \left( \mathbb{H}^n \right)$ is defined so that
		$$\int\limits_{\Xi_k \left( \mathbb{H}^n \right)} \varphi \left( \xi \right) \mathrm{d}\xi = \int\limits_{G} \Phi \left( g \right) \mathrm{d}g.$$
		It is also known (see, for instance, \cite{BerensteinRubin}) that for functions $\varphi: \Xi_k \left( \mathbb{H}^n \right) \rightarrow \mathbb{C}$ for which the integral makes sense, the following representation is true.
		\begin{equation}
			\label{XiKMeasure}
			\int\limits_{\Xi_k \left( \mathbb{H}^n \right)} \varphi \left( \xi \right) \mathrm{d}\xi = \left| \mathbb{S}^{n - k - 1} \right| \int\limits_{K} \int\limits_{0}^{\infty} \varphi \left( kg_{\theta}^{-1} \cdot \mathbb{H}^k \right) \sinh^{n - k - 1} \theta \ \cosh^{k} \theta \mathrm{d}\theta,
		\end{equation}
		where, $\mathbb{H}^{k}$ is the $k$-dimensional totally geodesic submanifold through $0 \in \mathbb{H}^n$ that has tangent space $\mathbb{R}^k \hookrightarrow \mathbb{R}^n$, where, the embedding is obtained by fixing the first $\left( n - k \right)$-coordinates to be zero, and $g_{\theta} = \begin{bmatrix}
					I_{n - 1} & 0 & 0 \\
					0 & \sinh \theta & \cosh \theta \\
					0 & \cosh \theta & \sinh \theta
				\end{bmatrix}$ is a hyperbolic translation that takes $0 \in \mathbb{H}^n$ to a point at a distance $\theta$ from $0$. The idea behind the representation given in Equation \eqref{XiKMeasure} is that a totally geodesic submanifold in $\mathbb{H}^n$ can be obtained by appropriately translating and rotating a fixed submanifold through $0$. Essentially, Equation \eqref{XiKMeasure} gives a ``polar decomposition" on the set $\Xi_k \left( \mathbb{H}^n \right)$ and is essential to our study.
				
		The following description of $\Xi_k \left( \mathbb{H}^n \right)$ is crucial for the considerations of Section \ref{LPLPBoundednessSection}. The result we now state is given in Proposition 2.8 of \cite{SelfLPImproving}.
		\begin{proposition}
			\label{DesriptionXik}
			The collection $\Xi_k \left( \mathbb{H}^n \right)$, of all $k$-dimensional totally geodesic subamanifolds of $\mathbb{H}^n$, can be described by the set $\left\lbrace \left( \xi, u \right) | \xi \in G_{n, k} \text{ and } u \in \xi^{\perp} \right\rbrace$, where $G_{n, k}$ is the Grassmannian manifold of all $k$-dimensional subspaces of $\mathbb{R}^n$. Here, $u$ may also be identified with the point in the submanifold closest to $0$ of $\mathbb{H}^n$.
		\end{proposition}
		With the understanding of the hyperbolic space and its totally geodesic submanifolds made above, we now give the definition of weighted Lebesgue spaces. The weights we consider are radial and arise naturally through the geometry of $\mathbb{H}^n$. The weighted Lebesgue space on $\mathbb{H}^n$ is given by,
		\begin{equation}
			\label{WeightedLPSPace}
            \begin{aligned}
                L^p_{\mu_1, \mu_2} \left( \mathbb{H}^n \right) := \Bigg\lbrace f: \mathbb{H}^n \rightarrow \mathbb{C} \Bigg| &f \text{ is measurable and } \\
                &\int\limits_{\mathbb{H}^n} \left| f \left( x \right) \right|^p \sinh^{\mu_1} d \left( 0, x \right) \cosh^{\mu_2} d \left( 0, x \right) \mathrm{d}x < + \infty \Bigg\rbrace.
            \end{aligned}
		\end{equation}
		The norm on this space is given by
		\begin{equation}
			\label{WeightedLPNorm}
			\| f \|_{L^p_{\mu_1, \mu_2} \left( \mathbb{H}^n \right)} := \left( \int\limits_{\mathbb{H}^n} \left| f \left( x \right) \right|^p \sinh^{\mu_1} d \left( 0, x \right) \cosh^{\mu_2} d \left( 0, x \right) \mathrm{d}x \right)^{\frac{1}{p}}.
		\end{equation}
		Similarly, the weighted Lebesgue space on $\Xi_k \left( \mathbb{H}^n \right)$ is given by
		\begin{equation}
			\label{WeightedLPSpaceXik}
            \begin{aligned}
                L^p_{\nu_1, \nu_2} \left( \Xi_k \left( \mathbb{H}^n \right) \right) := &\Bigg\lbrace \varphi: \Xi_k \left( \mathbb{H}^n \right) \rightarrow \mathbb{C} \Bigg| \varphi \text{ is measurable and } \\
                &\int\limits_{\Xi_k \left( \mathbb{H}^n \right)} \left| \varphi \left( \xi \right) \right|^p \sinh^{\nu_1} d \left( 0, \xi \right) \cosh^{\nu_2} d \left( 0, \xi \right) \mathrm{d}\xi < + \infty \Bigg\rbrace,
            \end{aligned}
		\end{equation}
		where, the measure $\mathrm{d}\xi$ on $\Xi_k \left( \mathbb{H}^n \right)$ is understood in the sense of Equation \eqref{XiKMeasure}. The norm on this space is given by,
		\begin{equation}
			\label{WeightedLPNormXik}
			\| \varphi \|_{L^p_{\nu_1, \nu_2} \left( \Xi_k \left( \mathbb{H}^n \right) \right)} := \left( \int\limits_{\Xi_k \left( \mathbb{H}^n \right)} \left| \varphi \left( \xi \right) \right|^p \sinh^{\nu_1} d \left( 0, \xi \right) \cosh^{\nu_2} d \left( 0, \xi \right) \mathrm{d}\xi \right)^{\frac{1}{p}}.
		\end{equation}
		We also use the notion of weighted Lorentz space $L^{p, 1}_{\mu_1, \mu_2} \left( \mathbb{H}^n \right)$ to study the end-point behaviour. For the definition of Lorentz space in the general measure theoretic setting, we refer the reader to Section 1.4.2 of  \cite{GrafakosCFA}.
	\section{The weighted analogue of the \texorpdfstring{$k$}{k}-plane transform and its dual}
		\label{TransfosmPreliminariesSection}	
		In this section, we define the operator of our interest formally, and its dual. With these definitions made concrete, we give a formula for these operators when they act on radial functions. Lastly, using these formulae, we make some explicit computations that are helpful further in the article.
		With these notations made precise, we now give the definition of a weighted totally geodesic $k$-plane transform on $\mathbb{H}^n$.
		\begin{definition}
			\label{WeightedKPlaneTransform}
			For a ``nice" function $f: \mathbb{H}^n \rightarrow \mathbb{C}$ and $\alpha > 0$, we define $R_{\alpha}f: \Xi_k \left( \mathbb{H}^n \right) \rightarrow \mathbb{C}$ as
			\begin{equation}
				\label{WeightedKPlaneTransformEquation}
				R_{\alpha}f \left( \xi \right) := \int\limits_{\xi} f \left( x \right) \sinh^{\alpha - k} d \left( x, u_{\xi} \right) \mathrm{d}x,
			\end{equation}
			where, $\mathrm{d}x$ is understood to be the $k$-dimensional Riemannian measure induced on $\xi$.
		\end{definition}
		\begin{remark}
			\normalfont
			For induced measures on (embedded) Riemannian submanifolds, we refer the reader to p. 15 and p. 30 of \cite{LeeRM}.
			
			We now give another representation of Equation \eqref{WeightedKPlaneTransformEquation}. We begin by noticing that since $\Xi_k \left( \mathbb{H}^n \right)$ is identified with $\faktor{G}{H}$, any element $\xi \in \Xi_k \left( \mathbb{H}^n \right)$ may be represented by a coset $gH$, for some $g \in G$. Thus, any element $x \in \xi$ may be represented by $gh$, for appropriate $h \in H$. That is, we have,
			\begin{equation}
				\label{WeightedKPlaneTransformGroupEquation}
				R_{\alpha}f \left( \xi \right) = \int\limits_{H} f \left( gh \right) \sinh^{\alpha - k} d \left( gh, u_{gH} \right) \mathrm{d}h.
			\end{equation}
		\end{remark}
		\begin{remark}
			\label{SingularIntegralsRemark}
			\normalfont
			The condition $\alpha > 0$ in Definition \ref{WeightedKPlaneTransform} is required for assuring the convergence of the integral near $u_{\xi}$. For $\alpha \leq 0$, one would land in the realm of singular integrals along totally geodesic submanifolds. In this article, we do not wish to work with singular integrals and therefore consider only the case when $\alpha > 0$.
		\end{remark}
		We now give the definition of the dual transform to Equation \eqref{WeightedKPlaneTransformEquation}.
		\begin{definition}[Dual weighted $k$-plane transform]
			\label{DualTransformDefinition}
			For a ``nice" function $\varphi: \Xi_k \left( \mathbb{H}^n \right) \rightarrow \mathbb{C}$, and $\alpha > 0$, we define the dual transform to $R_{\alpha}$ as $R_{\alpha}^*\varphi: \mathbb{H}^n \rightarrow \mathbb{C}$, as
			\begin{equation}
				\label{DualTransformEquation}
				R_{\alpha}^*\varphi \left( x \right) := \int\limits_{\xi \ni x} \varphi \left( \xi \right) \sinh^{\alpha - k} d \left( x, u_{\xi} \right) \mathrm{d}\xi,
			\end{equation}
			where, the measure $\mathrm{d}\xi$ on the set $\left\lbrace \xi \in \Xi_k \left( X \right) | x \in \xi \right\rbrace$ is the induced measure by the rotation group $K = SO \left( n \right)$.
		\end{definition}
		\begin{remark}
			\normalfont
			The idea behind inducing a measure on the set $\left\lbrace \xi \in \Xi_k \left( \mathbb{H}^n \right) | x \in \xi \right\rbrace$ is that any $k$-dimensional totally geodesic submanifold through $x \in \mathbb{H}^n$ can be realised as a rotation of a fixed submanifold (through $x$). For further details, we refer the reader to p. 44--45 of \cite{BerensteinRubin}.
		\end{remark}
		\begin{remark}
			\normalfont
			We notice that since $\mathbb{H}^n$ is identified with $\faktor{G}{K}$, every element $x \in \mathbb{H}^n$ is a coset $gK$, for some $g \in G$. Particularly, if $gK = x$, then we have $g \cdot 0 = x$. Therefore, any element $\xi \in \Xi_k \left( \mathbb{H}^n \right)$ with $x \in \xi$ can be obtained by $gk \cdot \mathbb{H}^k$, for an appropriate rotation $k \in K$. Here, by the symbol $\mathbb{H}^k$, we mean the submanifold of the upper sheet of hyperboloid obtained by intersecting it with $\mathbb{R}^{k + 1}$ (having the embedding given in Section \ref{HyperbolicSpacePreliminariesSection}). With this understanding, Equation \eqref{DualTransformEquation} may also be written as
			\begin{equation}
				\label{DualTransformGroupEquation}
				R_{\alpha}^*f \left( x \right) = \int\limits_{K} \varphi \left( gk \cdot \mathbb{H}^k \right) \sinh^{\alpha - k} d \left( g \cdot 0, u_{gk \cdot \mathbb{H}^k} \right) \mathrm{d}k.
			\end{equation}
		\end{remark}
		We now prove that the transform given in Equation \eqref{DualTransformEquation} is a formal dual of the weighted $k$-plane transform given in Equation \eqref{WeightedKPlaneTransformEquation}.
		\begin{proposition}
			\label{FormalDualProposition}
			For ``nice" functions $f: \mathbb{H}^n \rightarrow \mathbb{C}$ and $\varphi: \Xi_k \left( \mathbb{H}^n \right) \rightarrow \mathbb{C}$, we have,
			\begin{equation}
				\label{DualityEquation}
				\int\limits_{\Xi_k \left( \mathbb{H}^n \right)} R_{\alpha}f \left( \xi \right) \varphi \left( \xi \right) \mathrm{d}\xi = \int\limits_{\mathbb{H}^n} f \left( x \right) R_{\alpha}^*\varphi \left( x \right) \mathrm{d}x.
			\end{equation}
		\end{proposition}
		\begin{proof}
			Throughout this proof, we fix $f_1: G \rightarrow \left( 0, \infty \right)$ as the function used in defining the measure on $\Xi_k \left( \mathbb{H}^n \right)$. Also, due to the quotient structure of $\Xi_k \left( \mathbb{H}^n \right)$, any submanifold $\xi \in \Xi_k \left( \mathbb{H}^n \right)$ can be realised as a coset $gH$, for some $g \in G$. Consequently, integration of a function $f$ over the submanifold $\xi$ can be realised as integrating the translated function $f \left( g \cdot \right)$ over $H$ in the following manner:
			\begin{equation}
				\label{WeighteKPlaneTransformOverH}
				R_{\alpha}f \left( \xi \right) = \int\limits_{H} f \left( gh \cdot 0 \right) \sinh^{\alpha - k} d \left( gh \cdot 0, u_{g \cdot \mathbb{H}^k} \right) \mathrm{d}h.
			\end{equation}
			Thus, we have, for ``nice" functions $f: \mathbb{H}^n \rightarrow \mathbb{C}$ and $\varphi: \Xi_k \left( \mathbb{H}^n \right) \rightarrow \mathbb{C}$,
            \begin{align*}
                &\int\limits_{\Xi_k \left( \mathbb{H}^n \right)} R_{\alpha}f \left( \xi \right) \varphi \left( \xi \right) \mathrm{d}\xi \\
                &= \int\limits_{G} \frac{\left( R_{\alpha}f \circ \pi \right) \left( g \right) \left( \varphi \circ \pi \right) \left( g \right) f_1 \left( g \right)}{\bar{f}_1 \left( g \right)} \mathrm{d}g \\
                &= \int\limits_{G} \frac{\left( \varphi \circ \pi \right) \left( g \right) f_1 \left( g \right)}{\bar{f}_1 \left( g \right)} \int\limits_{H} f \left( gh \cdot 0 \right) \sinh^{\alpha - k} d \left( gh \cdot 0, u_{g \cdot \mathbb{H}^k} \right) \mathrm{d}h \mathrm{d}g.
            \end{align*}
			Since $K$ is a compact group with normalized measure, we have,
            \begin{align*}
                &\int\limits_{\Xi_k \left( \mathbb{H}^n \right)} R_{\alpha}f \left( \xi \right) \varphi \left( \xi \right) \mathrm{d}\xi \\
                &= \int\limits_{K} \int\limits_{G} \frac{\left( \varphi \circ \pi \right) \left( g \right) f_1 \left( g \right)}{\bar{f}_1 \left( g \right)} \int\limits_{H} f \left( ghk^{-1} \cdot 0 \right) \sinh^{\alpha - k} d \left( g h k^{-1} \cdot 0, u_{g \cdot \mathbb{H}^k} \right).
            \end{align*}
			We now use $ghk^{-1} = \lambda$ and use the left-invariance of the measure on $G$ to obtain
            \begin{align*}
                &\int\limits_{\Xi_k \left( \mathbb{H}^n \right)} R_{\alpha}f \left( \xi \right) \varphi \left( \xi \right) \mathrm{d}\xi \\
                &= \int\limits_{K} \int\limits_{H} \int\limits_{G} f \left( \lambda \cdot 0 \right) \sinh^{\alpha - k} d \left( \lambda \cdot 0, u_{\lambda k h^{-1} \cdot \mathbb{H}^k} \right) \frac{\left( \varphi \circ \pi \right) \left( \lambda k h^{-1} \right) f_1 \left( \lambda k h^{-1} \right)}{\bar{f}_1 \left( \lambda k h^{-1} \right)} \mathrm{d}\lambda \mathrm{d}h \mathrm{d}k.
            \end{align*}
			Using the fact that $H$ fixes $\mathbb{H}^k$ and the definition of $f_1$, we obtain
			$$\int\limits_{\Xi_k \left( \mathbb{H}^n \right)} R_{\alpha}f \left( \xi \right) \varphi \left( \xi \right) \mathrm{d}\xi = \int\limits_{G} f \left( \lambda \cdot 0 \right) \left( \int\limits_{K} \varphi \left( \lambda k H \right) \sinh^{\alpha - k} d \left( \lambda \cdot 0,  u_{\lambda k \cdot \mathbb{H}^k} \right) \mathrm{d}k \right) \mathrm{d}\lambda.$$
			As $k \in K$ varies, the collection $\left\lbrace \lambda k H \right\rbrace$ becomes the set of those submanifolds $\xi \in \Xi_k \left( \mathbb{H}^n \right)$ that pass through $\lambda \cdot 0$. That is, after passing to $\mathbb{H}^n$, we have,
			$$\int\limits_{\Xi_k \left( \mathbb{H}^n \right)} R_{\alpha}f \left( \xi \right) \varphi \left( \xi \right) \mathrm{d}\xi = \int\limits_{\mathbb{H}^n} f \left( x \right) \left( \int\limits_{\xi \ni x} \varphi \left( \xi \right) \sinh^{\alpha - k} d \left( x, u_{\xi} \right) \mathrm{d}\xi \right) \mathrm{d}x.$$
			This proves the duality of $R_{\alpha}$ and $R_{\alpha}^*$.
		\end{proof}
		In this article, we deal with radial functions to obtain $L^p$-improving boundedness of the operator $R_{\alpha}$. To achieve this end, we first describe radial functions on $\mathbb{H}^n$ and then give a representation of $R_{\alpha}f$ as a one-dimensional integral whenever $f$ is radial.
		\begin{definition}[Radial functions]
			\label{RadialFunctionDefinition}
			A function $f: \mathbb{H}^n \rightarrow \mathbb{C}$ is radial if there is a function $\tilde{f}: \left[ 1, \infty \right) \rightarrow \mathbb{C}$ such that $f \left( x \right) = \tilde{f} \left( \cosh d \left( 0, x \right) \right)$.
		\end{definition}
		With the definition of radial functions at hand, we now provide a ``formula" for $R_{\alpha}f$, when $f$ is radial. For the same, we require the cosine rule on the real hyperbolic space, the exact statement of which can be found in Theorem 3.5.3 of \cite{Ratcliffe}.
		\begin{theorem}[Cosine Rule]
			\label{CosineRuleHn}
			Let $p, q$ and $r$ be points in $\mathbb{H}^n$. Consider the geodesic triangle formed by these points with angles $\alpha, \beta$ and $\gamma$, and length of sides opposite to these angles given by $a, b$ and $c$. Then, we have,
			\begin{equation}
				\label{CosineRuleHnEquation}
				\cosh c = \cosh a \cosh b - \sinh a \sinh b \cos \gamma.
			\end{equation}
			Particularly, when $\gamma = \frac{\pi}{2}$, so that the geodesic triangle is a right-angled triangle, we have,
			\begin{equation}
				\label{PythagorasTheoremHnEquation}
				\cosh c = \cosh a \cosh b,
			\end{equation}
			where, $c$ is the length of the ``hypotenuse" (side opposite to $\frac{\pi}{2}$ angle).
		\end{theorem}
		We now provide the formula for $R_{\alpha}f$ when $f$ is radial.
		\begin{proposition}
			\label{FormulaForRadialFunctionsRAlpha}
			Let $f: \mathbb{H}^n \rightarrow \mathbb{C}$ be a ``nice" radial function. Then, we have for any $\xi \in \Xi_k \left( \mathbb{H}^n \right)$,
			\begin{equation}
				\label{FormualRadialFunctionsEquation}
				R_{\alpha}f \left( \xi \right) = \frac{\left| \mathbb{S}^{k - 1} \right|}{\cosh^{\alpha - 1} d \left( 0, \xi \right)} \int\limits_{\cosh d \left( 0, \xi \right)}^{\infty} \tilde{f} \left( t \right) \left( t^2 - \cosh^2 d \left( 0, \xi \right) \right)^{\frac{\alpha}{2} - 1} \mathrm{d}t.
			\end{equation}
		\end{proposition}
		\begin{proof}
			Given a $k$-dimensional submanifold $\xi \in \Xi_k \left( \mathbb{H}^n \right)$, Theorem \ref{FootOfPerpendicularTGS} guarantees the existence of a point $u_{\xi} \in \xi$ that is closest to $0 \in \mathbb{H}^n$. Further, we know that the geodesic joining $0$ and $u_{\xi}$ is orthogonal to $\xi$ at $u_{\xi}$. Therefore, for any $x \in \xi$, the geodesic triangle formed by $0$, $x$, and $u_{\xi}$ is a right-angled triangle (see Figure \ref{GeodesicTriangleHnFigure} for a diagrammatic representation in the Poincar\'{e} ball model) with the length of hypotenuse given by $d \left( 0, x \right)$. Therefore, using Equation \eqref{PythagorasTheoremHnEquation}, we get
			\begin{equation}
				\label{PythagorasSpecialHnXik}
				\cosh d \left( 0, x \right) = \cosh d \left( 0, \xi \right) \cosh d \left( x, u_{\xi} \right).
			\end{equation}
			\begin{figure}[ht!]
				\centering
				\includegraphics[scale=1.0]{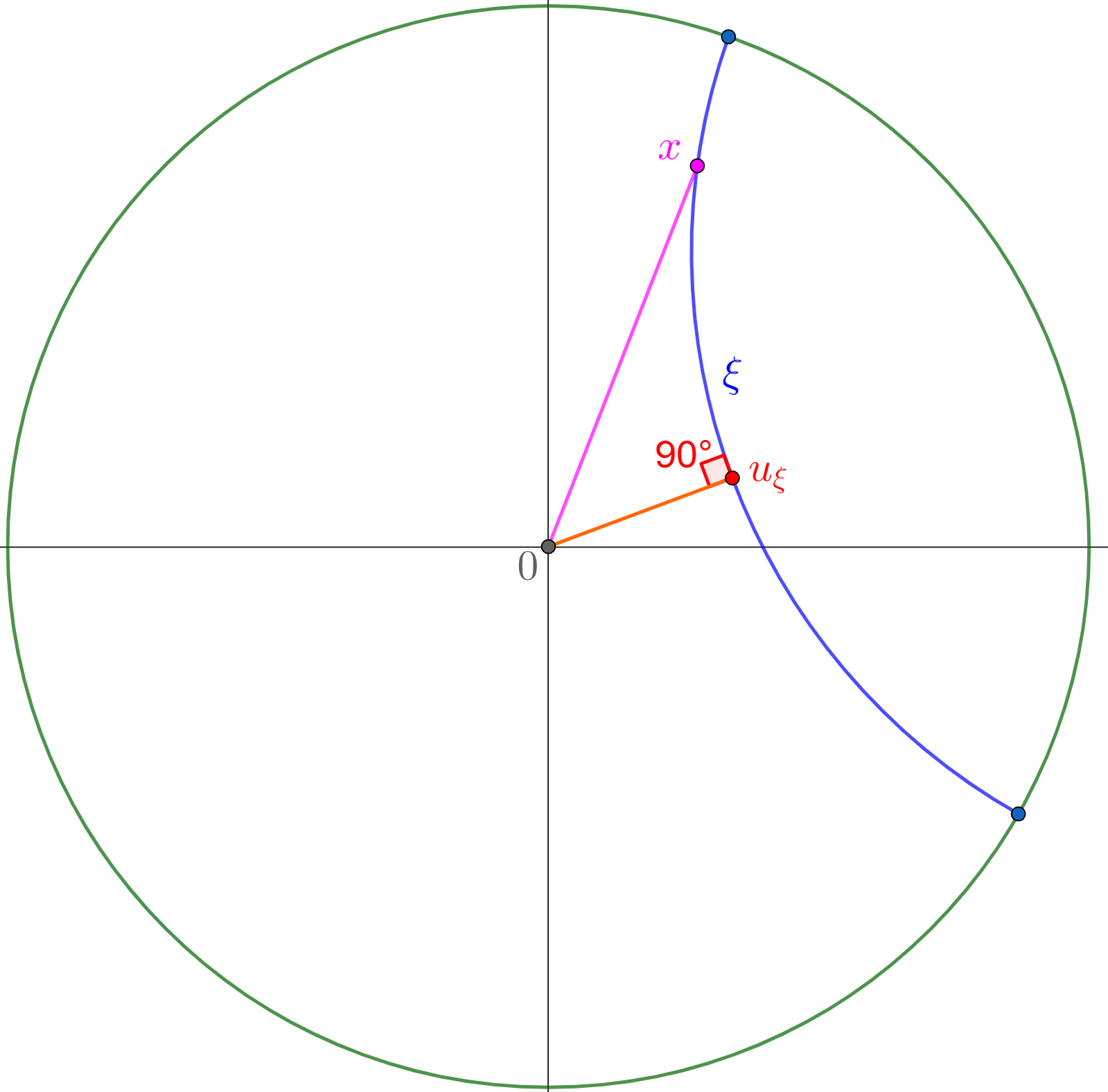}
				\caption{The geodesic triangle formed by $0$, $x \in \xi$ and $u_{\xi}$, shown in the Poincar\'{e} ball model.}
				\label{GeodesicTriangleHnFigure}
			\end{figure}
			Combining this observation with the definition of radial functions, we get
			$$R_{\alpha}f \left( \xi \right) = \int\limits_{\xi} f \left( \cosh d \left( 0, \xi \right) \cosh d \left( x, u_{\xi} \right) \right) \sinh^{\alpha - k} d \left( x, u_{\xi} \right) \mathrm{d}x.$$
			Now, using the polar decomposition on $\xi$ keeping $u_{\xi}$ as the center, and observing that the integrand depends only on the distance from $u_{\xi}$, we get
			$$R_{\alpha}f \left( \xi \right) = \left| \mathbb{S}^{k - 1} \right| \int\limits_{0}^{\infty} \tilde{f} \left( \cosh d \left( 0, \xi \right) \cosh r \right) \sinh^{\alpha - 1} r \ \mathrm{d}r.$$
			Next, we substitute $\cosh d \left( 0, \xi \right) \cosh r = t$ and simplify to obtain
			$$R_{\alpha}f \left( \xi \right) = \left| \mathbb{S}^{k - 1} \right| \int\limits_{\cosh d \left( 0, \xi \right)}^{\infty} \tilde{f} \left( t \right) \left( \frac{t^2}{\cosh^2 d \left( 0, \xi \right)} - 1 \right)^{\frac{\alpha}{2} - 1} \frac{\mathrm{d}t}{\cosh d \left( 0, \xi \right)}.$$
			This is same as Equation \eqref{FormualRadialFunctionsEquation}.
		\end{proof}
		Next, We get a formula for the dual transform of radial functions on $\Xi_k \left( \mathbb{H}^n \right)$. we begin by recalling the definition of radial functions.
		\begin{definition}
			\label{RadialFunctionsXik}
			A function $\varphi: \Xi_k \left( \mathbb{H}^n \right) \rightarrow \mathbb{C}$ is radial if there is a function $\tilde{\varphi}: \left[ 0, \infty \right) \rightarrow \mathbb{C}$ such that $\varphi \left( \xi \right) = \tilde{\varphi} \left( \sinh d \left( 0, \xi \right) \right)$.
		\end{definition}
		\begin{proposition}
			\label{RadialFunctionsDualFormula}
			Let $\varphi: \Xi_k \left( \mathbb{H}^n \right) \rightarrow \mathbb{C}$ be a ``nice" radial function. Then, for any $x \in \mathbb{H}^n$, we have,
			\begin{equation}
				\label{RadialFunctionsDualFormuaEquation}
                \begin{aligned}
                    R_{\alpha}^*\varphi \left( x \right) = &\frac{\left| \mathbb{S}^{k - 1} \right| \left| \mathbb{S}^{n - k - 1} \right|}{\left| \mathbb{S}^{n - 1} \right|} \frac{1}{\sinh^{n - 2} d \left( 0, x \right)} \times \\
                    &\int\limits_{0}^{\sinh d \left( 0, x \right)} \frac{\tilde{\varphi} \left( s \right)}{\left( 1 + s^2 \right)^{\frac{\alpha - k}{2}}} s^{n - k - 1} \left( \sinh^2 d \left( 0, x \right) - s^2 \right)^{\frac{\alpha}{2} - 1} \mathrm{d}s.
                \end{aligned}
			\end{equation}
		\end{proposition}
		\begin{proof}
			Given $x \in \mathbb{H}^n$, let $g \in G$ be such that $g \cdot 0 = x$. Then, it is clear that any $\xi \in \Xi_k \left( \mathbb{H}^n \right)$ with $x \in \xi$ can be given by $gk \cdot \mathbb{H}^k$, for some $k \in SO \left( n \right)$. That is, we have,
			$$R_{\alpha}^*\varphi \left( x \right) = \int\limits_{SO \left( n \right)} \tilde{\varphi} \left( \sinh d \left( 0, gk \cdot \mathbb{H}^k \right) \right) \sinh^{\alpha - k} d \left( g \cdot 0, u_{gk \cdot \mathbb{H}^k} \right) \mathrm{d}k.$$
			As shown in Figure \ref{GeodesicTriangleHnFigure}, the geodesic triangle formed by $0$, $x$ and $u_{gk \cdot \mathbb{H}^k}$ is a right-angled triangle. Therefore, by using Equation \eqref{PythagorasTheoremHnEquation}, we get
			$$\cosh d \left( 0, x \right) = \cosh d \left( 0, gk \cdot \mathbb{H}^k \right) \cosh d \left( x, u_{gk \cdot \mathbb{H}^k} \right).$$
			This gives
			$$\sinh d \left( x, u_{gk \cdot \mathbb{H}^k} \right) = \left( \frac{\cosh^2 d \left( 0, x \right)}{\cosh^2 d \left( 0, gk \cdot \mathbb{H}^k \right)} - 1 \right)^{\frac{1}{2}}.$$
			Now, using the fact the $G$ and $K = SO \left( n \right)$ act on $\mathbb{H}^n$ through isometries and that $K$ is a unimodular group (see Corollary 2.28 of \cite{FollandAHA}), we get
			$$R_{\alpha}^*\varphi \left( x \right) = \int\limits_{SO \left( n \right)} \tilde{\varphi} \left( \sinh d \left( kg^{-1} \cdot 0, \mathbb{H}^k \right) \right) \left( \frac{\cosh^2 d \left( 0, x \right)}{\cosh^2 d \left( kg^{-1} \cdot 0, \mathbb{H}^k \right)} - 1 \right)^{\frac{\alpha - k}{2}} \mathrm{d}k.$$
			For the sake of convenience involving the computations, we shall restrict ourselves to the upper sheet of hyperboloid model of $\mathbb{H}^n$ throughout this proof. Consequently, by $0$ of $\mathbb{H}^n$, we mean $e_{n + 1} \in \mathbb{R}^{n + 1}$. We write $g^{-1} \cdot 0 = \cosh \left( d \left( 0, x \right) \right) e_{n + 1} + \sinh \left( d \left( 0, x \right) \right) \sigma$, for some $\sigma \in \mathbb{S}^{n - 1}$ (see Figure \ref{DecompositionHnFigure}). Then, as $k \in SO \left( n \right)$ varies, $k \cdot \sigma$ varies over $\mathbb{S}^{n - 1}$. Therefore, we get, after converting the integral over $SO \left( n \right)$ to an integral over $\mathbb{S}^{n - 1}$,
			\begin{align*}
				R_{\alpha}^*\varphi \left( x \right) = &\frac{1}{\left| \mathbb{S}^{n - 1} \right|} \int\limits_{\mathbb{S}^{n - 1}} \tilde{\varphi} \left( \sinh \left( d \left( \cosh \left( d \left( 0, x \right) \right) e_{n + 1} + \sinh \left( d \left( 0, x \right) \right) \omega, \mathbb{H}^k \right) \right) \right) \times \\
				&\left( \frac{\cosh^2 d \left( 0, x \right)}{\cosh^2 d \left( \cosh \left( d \left( 0, x \right) \right) e_{n + 1} + \sinh \left( d \left( 0, x \right) \right) \omega, \mathbb{H}^k \right)} - 1 \right)^{\frac{\alpha - k}{2}} \mathrm{d}\omega.
			\end{align*}
            
			\begin{figure}[ht!]
				\centering
				\includegraphics[scale=0.4]{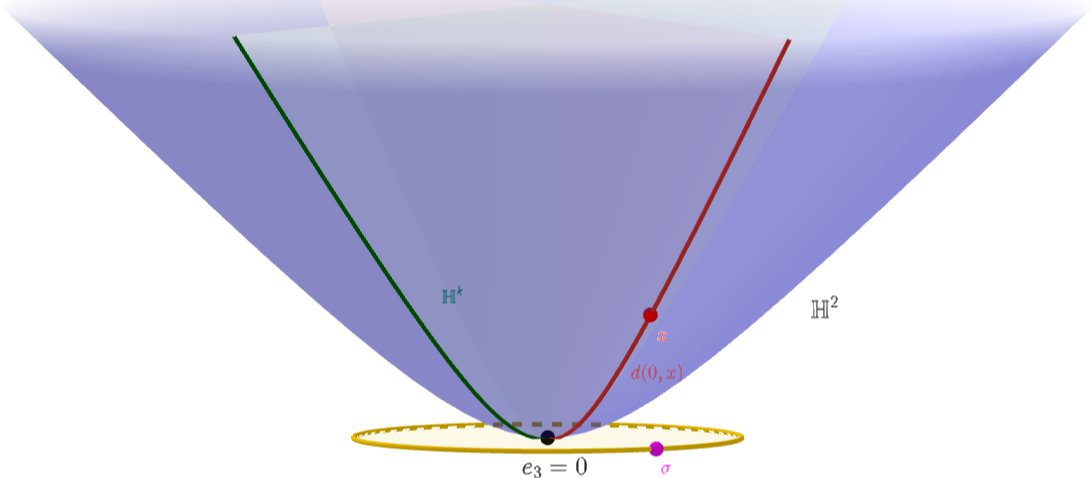}
				\caption{Writing a point $x \in \mathbb{H}^n$ in terms of $e_{n + 1}$ and a direction $\sigma \in \mathbb{S}^{n - 1}$.}
				\label{DecompositionHnFigure}
			\end{figure}
            
			We use the bispherical coordinates on $\mathbb{S}^{n - 1}$ and write $\omega = \cos \theta \zeta + \sin \theta \eta$, for $\zeta \in \mathbb{S}^{k - 1}$, $\eta \in \mathbb{S}^{n - k - 1}$ and $\theta \in \left[ 0, \frac{\pi}{2} \right]$. In this decomposition, we have $\mathrm{d}\omega = \sin^{n - k - 1} \theta \cos^{k - 1} \theta \ \mathrm{d} \theta \mathrm{d}\zeta \mathrm{d}\eta$. For more details on bispherical coordinates on $\mathbb{S}^{n - 1}$, we refer the reader to p. 22 of \cite{VilenkinBook}. Then, it is immediate that
			$$\sinh \left( d \left( \cosh \left( d \left( 0, x \right) \right) e_{n + 1} + \sinh \left( d \left( 0, x \right) \right) \omega, \mathbb{H}^k \right) \right) = \sinh d \left( 0, x \right) \sin \theta,$$
			and
			$$\cosh d \left( \cosh \left( d \left( 0, x \right) \right) e_{n + 1} + \sinh \left( d \left( 0, x \right) \right) \omega, \mathbb{H}^k \right) = \sinh^2 d \left( 0, x \right) \sin^2 \theta + 1.$$
			We observe from the above equations that the integral in the representation of $R_{\alpha}^*\varphi$ is independent of $\zeta$ and $\eta$ from the bispherical decomposition of $\mathbb{S}^{n - 1}$. Thus, upon simplifying the terms, we get
			\begin{align*}
				R_{\alpha}^*\varphi \left( x \right) &= \frac{\left| \mathbb{S}^{k - 1} \right| \left| \mathbb{S}^{n - k - 1} \right|}{\left| \mathbb{S}^{n - 1} \right|} \sinh^{\alpha - k} d \left( 0, x \right) \times \\
                &\int\limits_{0}^{\frac{\pi}{2}} \tilde{\varphi} \left( \sinh d \left( 0, x \right) \sin \theta \right) \frac{\left( \sinh^2 d \left( 0, x \right) - \sinh^2 d \left( 0, x \right) \sin^2 \theta \right)^{\frac{\alpha - k}{2}}}{\left( \sinh^2 d \left( 0, x \right) \sin^2 \theta + 1 \right)^{\frac{\alpha - k}{2}}} \times \\
                &\sin^{n - k - 1} \theta \cos^{k - 1} \theta \ \mathrm{d}\theta.
			\end{align*}
			Equation \eqref{RadialFunctionsDualFormuaEquation} is obtained at once by substituting $\sinh d \left( 0, x \right) \sin \theta = s$.
		\end{proof}
		Next, we provide an explicit computation of the dual transform of powers of hyperbolic sine and cosine that is of use in the sequel.
		\begin{proposition}
			\label{DualSinhCoshPowers}
			Let $\varphi: \Xi_k \left( \mathbb{H}^n \right) \rightarrow \mathbb{C}$ be defined by 
            $$\varphi \left( \xi \right) = \sinh^{\gamma_1} d \left( 0, \xi \right) \cosh^{\gamma_2} d \left( 0, \xi \right),$$
            with $\gamma_1 > k - n$. Then, we have,
			\begin{equation}
				\label{DualSinhCpshPowersEquation}
				\begin{aligned}
					R_{\alpha}^*\varphi \left( x \right) &= \frac{\left| \mathbb{S}^{k - 1} \right| \left| \mathbb{S}^{n - k - 1} \right|}{\left| \mathbb{S}^{n - 1} \right|} \frac{\Gamma \left( \frac{\gamma_1 + n - k}{2} \right) \Gamma \left( \frac{\alpha}{2} \right)}{\Gamma \left( \frac{\gamma_1 + \alpha + n - k}{2} \right)} \sinh^{\gamma_1 + \alpha - k} d \left( 0, x \right) \times \\
					&\cosh^{\gamma_2 - \alpha + k} d \left( 0, x \right) {}_2F_1 \left( \frac{\alpha - k - \gamma_2}{2}, \frac{\alpha}{2}; \frac{\gamma_1 + \alpha + n - k}{2}; \tanh^2 d \left( 0, x \right) \right),
				\end{aligned}
			\end{equation}
			where, $\Gamma \left( \cdot \right)$ is the Gamma function and ${}_2F_1$ is the Gauss' hypergeometric function.
		\end{proposition}
		\begin{proof}
			It is clear that $\varphi$ is a radial function on $\Xi_k \left( \mathbb{H}^n \right)$, and in the notation of Definition \ref{RadialFunctionsXik}, we have $\varphi \left( \xi \right) = \tilde{\varphi} \left( \sinh d \left( 0, \xi \right) \right)$ with $\tilde{\varphi} \left( s \right) = s^{\gamma_1} \left( 1 + s^2 \right)^{\frac{\gamma_2}{2}}$. Using this in Equation \eqref{RadialFunctionsDualFormuaEquation}, we get
            \begin{align*}
                R_{\alpha}^*\varphi \left( x \right) = &\frac{\left| \mathbb{S}^{k - 1} \right| \left| \mathbb{S}^{n - k - 1} \right|}{\left| \mathbb{S}^{n - 1} \right|} \frac{1}{\sinh^{n - 2} d \left( 0, x \right)} \times \\
            &\int\limits_{0}^{\sinh d \left( 0, x \right)} \frac{s^{\gamma_1 + n - k - 1}}{\left( 1 + s^2 \right)^{\frac{\alpha - k - \gamma_2}{2}}} \left( \sinh^2 d \left( 0, x \right) - s^2 \right)^{\frac{\alpha}{2} - 1} \mathrm{d}s.
            \end{align*}
			Substituting $s^2 = \sinh^2 d \left( 0, x \right) t$ and simplifying, we get
			$$R_{\alpha}^*\varphi \left( x \right) = \frac{\left| \mathbb{S}^{k - 1} \right| \left| \mathbb{S}^{n - k - 1} \right|}{\left| \mathbb{S}^{n - 1} \right|} \sinh^{\gamma_1 - k + \alpha} d \left( 0, x \right) \int\limits_{0}^{1} \frac{t^{\frac{\gamma_1 + n - k}{2} - 1} \left( 1 - t \right)^{\frac{\alpha}{2} - 1}}{\left( 1 + \sinh^2 d \left( 0, x \right) t \right)^{\frac{\alpha - k - \gamma_2}{2}}} \mathrm{d}t.$$
			Since $\gamma_1 > k - n$, we may use Equation 15.6.1 of \cite{Olver} to obtain
			\begin{align*}
				R_{\alpha}^*\varphi \left( x \right) &= \frac{\left| \mathbb{S}^{k - 1} \right| \left| \mathbb{S}^{n - k - 1} \right|}{\left| \mathbb{S}^{n - 1} \right|} \frac{\Gamma \left( \frac{\gamma_1 + n - k}{2} \right) \Gamma \left( \frac{\alpha}{2} \right)}{\Gamma \left( \frac{\gamma_1 + \alpha + n - k}{2} \right)} \sinh^{\gamma_1 + n - k} d \left( 0, x \right) \times \\
				&{}_2F_1 \left( \frac{\alpha - k - \gamma_2}{2}, \frac{\gamma_1 + n - k}{2}; \frac{\gamma_1 + \alpha + n - k}{2}; - \sinh^2 d \left( 0, x \right) \right).
			\end{align*}
			Equation \eqref{DualSinhCpshPowersEquation} is obtained at once by an application of Equation 15.8.1 of \cite{Olver}.
		\end{proof}
	\section{Existence on Weighted Lebesgue Spaces}
		\label{ExistenceSection}
		In this section, we obtain the existence of the transform $R_{\alpha}$ on weighted Lebesgue spaces $L^p_{\mu_1, \mu_2} \left( \mathbb{H}^n \right)$. Particularly, we answer the following question.
		\begin{question}
			\label{ExistenceQuestion}
			What are the optimal values of $p \geq 1$ and $\mu_1, \mu_2 \in \mathbb{R}$ so that for every $f \in L^p_{\mu_1, \mu_2} \left( \mathbb{H}^n \right)$, the transform $R_{\alpha}f$ is well-defined?
		\end{question}
		We answer Question \ref{ExistenceQuestion} in the following theorem.
		\begin{theorem}
			\label{ExistenceTheorem}
			Let $\alpha \geq 1$, $\mu_1 + \mu_2 > \alpha - n$ and $1 \leq p < \frac{\mu_1 + \mu_2 + n - 1}{\alpha - 1}$. Then, $R_{\alpha}$ is well-defined on $L^p_{\mu_1, \mu_2} \left( \mathbb{H}^n \right)$. Further, when $\mu_1 + \mu_2 = \alpha - n$, the transform $R_{\alpha}$ is well-defined on $L^1_{\mu_1, \mu_2} \left( \mathbb{H}^n \right)$. Moreover, the restrictions on the values of $\alpha$, $\mu_1$, $\mu_2$ and $p$ mentioned in the statement of the result are optimal.
		\end{theorem}
		\begin{proof}
			We prove the result in three parts. First, let us assume $\alpha \geq 1$, $\mu_1 + \mu_2 > \alpha - n$ and $1 < p < \frac{\mu_1 + \mu_2 + n - 1}{\alpha - 1}$. We notice that the existence of such a $p$ is guaranteed by the condition on $\mu_1$ and $\mu_2$. Let us fix $\gamma_1 > \max \left\lbrace k - n - \mu_1, k - \alpha - \frac{\mu + n}{p'} \right\rbrace$ and $\gamma_2 \in \mathbb{R}$ such that $\alpha - n - \mu_1 - \mu_2 < \gamma_1 + \gamma_2 < \frac{1 - \mu_1 - \mu_2 - n}{p'}$. Again, we emphasize that such a choice of $\gamma_1$ and $\gamma_2$ is possible due to the conditions imposed on $\alpha$, $\mu_1$, $\mu_2$ and $p$. Then, using the duality of the transforms $R_{\alpha}$ and $R_{\alpha}^*$, as given in Equation \eqref{DualityEquation}, and Equation \eqref{DualSinhCpshPowersEquation}, we get for $f \in L^p_{\mu_1, \mu_2} \left( \mathbb{H}^n \right)$,
			\begin{align*}
				&\left| \int\limits_{\Xi_k \left( \mathbb{H}^n \right)} R_{\alpha}f \left( \xi \right) \sinh^{\gamma_1 + \mu_1} d \left( 0, \xi \right) \cosh^{\gamma_2 + \mu_2} d \left( 0, \xi \right) \mathrm{d}\xi \right| \\
                &\leq C \int\limits_{\mathbb{H}^n} \left| f \left( x \right) \right| \sinh^{\gamma_1 + \mu_1 + \alpha - k} d \left( 0, x \right) \cosh^{\gamma_2 + \mu_2 - \alpha + k} d \left( 0, x \right) \times \\
				&\left| {}_2F_1 \left( \frac{\alpha - k - \mu_2 - \gamma_2}{2}, \frac{\alpha}{2}; \frac{\gamma_1 + \mu_1 + \alpha - n - k}{2}; \tanh^2 d \left( 0, x \right) \right) \right| \mathrm{d}x.
			\end{align*}
			By an application of H\"{o}lder's inequality with exponents $p$ and $p'$ and using the polar decomposition on $\mathbb{H}^n$, we get
			\begin{align*}
				&\left| \int\limits_{\Xi_k \left( \mathbb{H}^n \right)} R_{\alpha}f \left( \xi \right) \sinh^{\gamma_1 + \mu_1} d \left( 0, \xi \right) \cosh^{\gamma_2 + \mu_2} d \left( 0, \xi \right) \mathrm{d}\xi \right| \leq C \| f \|_{L^p_{\mu_1, \mu_2} \left( \mathbb{H}^n \right)} \times \\
				&\left[ \int\limits_{0}^{\infty} \sinh^{p' \left( \gamma_1 + \alpha - k \right) - \mu_1 + n - 1} r \cosh^{p' \left( \gamma_2 - \alpha + k \right) + \mu_2} r \times \right. \\
                &\left.\left| {}_2F_1 \left( \frac{\alpha - k - \mu_2 - \gamma_2}{2}, \frac{\alpha}{2}; \frac{\gamma_1 + \mu_1 + \alpha + n - k}{2}; \tanh^2 r \right) \right|^{p'} \right]^{\frac{1}{p'}}.
			\end{align*}
			We notice that due to the choice of $\gamma_1$ and $\gamma_2$, and Equation 15.4.20 of \cite{Olver}, the hypergeometric function on the right side of the above inequality is bounded. Further, the integral is finite, again due to the choice of $\gamma_1$ and $\gamma_2$. Since $f \in L^p_{\mu_1, \mu_2} \left( \mathbb{H}^n \right)$, we conclude that the left side of the above inequality must be finite. Hence, $R_{\alpha}f$ is defined for almost every $\xi \in \Xi_k \left( \mathbb{H}^n \right)$.
			
			Next, we consider the case $p = 1$ with the above conditions on $\alpha$, $\mu_1$ and $\mu_2$. Here, we choose $\gamma_1 > \max \left\lbrace k - \alpha, k - n - \mu_1 \right\rbrace$ and set $\gamma_2 = - \gamma_1$. Then, using duality of Equation \eqref{DualityEquation} and Equation \eqref{DualSinhCpshPowersEquation}, we get
			\begin{align*}
				&\left| \int\limits_{\Xi_k \left( \mathbb{H}^n \right)} R_{\alpha}f \left( \xi \right) \sinh^{\gamma_1 + \mu_2} d \left( 0, \xi \right) \cosh^{\gamma_2 + \mu_2} d \left( 0, \xi \right) \mathrm{d}\xi \right| \\
                &\leq C \int\limits_{\mathbb{H}^n} \left| f \left( x \right) \right| \sinh^{\mu_1} d \left( 0, x \right) \cosh^{\mu_2} d \left( 0, x \right) \tanh^{\gamma_1 + \alpha - k} d \left( 0, x \right) \times \\
				&\left| {}_2F_1 \left( \frac{\alpha - k - \mu_1 + \gamma_1}{2}, \frac{\alpha}{2}; \frac{\gamma_1 + \mu_1 + \alpha + n - k}{2}; \tanh^2 d \left( 0, x \right) \right) \right| \mathrm{d}x.
			\end{align*}
			Again, using Equation 15.4.20 of \cite{Olver}, we observe that the hypergeometric function in the above inequality is bounded. Moreover, since $\gamma_1 > k - \alpha$, we have $\tanh^{\gamma_1 + \alpha - k} d \left( 0, x \right) \leq 1$. That is, we get
            \begin{align*}
                &\left| \int\limits_{\Xi_k \left( \mathbb{H}^n \right)} R_{\alpha}f \left( \xi \right) \sinh^{\gamma_1 + \mu_2} d \left( 0, \xi \right) \cosh^{\gamma_2 + \mu_2} d \left( 0, \xi \right) \mathrm{d}\xi \right| \\
                &\leq C \int\limits_{\mathbb{H}^n} \left| f \left( x \right) \right| \sinh^{\mu_1} d \left( 0, x \right) \cosh^{\mu_2} d \left( 0, x \right) \mathrm{d}x = C \| f \|_{L^1_{\mu_1, \mu_2} \left( \mathbb{H}^n \right)}.
            \end{align*}
			This gives the existence of $R_{\alpha}$ on $L^1_{\mu_1, \mu_2} \left( \mathbb{H}^n \right)$.
			
			Finally, we consider the case $\alpha > 0$, $\mu_1 + \mu_2 = \alpha - n$ and $p = 1$. Here, by choosing $\gamma_1 > \max \left\lbrace k - \alpha, k - n - \mu_1 \right\rbrace$ and $\gamma_2 = - \gamma_1 - 1$, we get by proceeding as above,
			\begin{align*}
				&\left| \int\limits_{\Xi_k \left( \mathbb{H}^n \right)} R_{\alpha}f \left( \xi \right) \sinh^{\gamma_1 + \mu_1} d \left( 0, \xi \right) \cosh^{\gamma_2 + \mu_2} d \left( 0, \xi \right) \mathrm{d}\xi \right| \\
                &\leq C \int\limits_{\mathbb{H}^n} \left| f \left( x \right) \right| \sinh^{\gamma_1 + \mu_1} d \left( 0, x \right) \cosh^{\gamma_2 + \mu_2} d \left( 0, x \right) \tanh^{\gamma_1 + \alpha - k} d \left( 0, x \right) \times \\
				&\left| \frac{{}_2F_1 \left( \frac{\alpha - k - \mu_2 + \gamma_1 + 1}{2}, \frac{\alpha}{2}; \frac{\gamma_1 + \mu_1 + \alpha + n - k}{2}; \tanh^2 d \left( 0, x \right) \right)}{\cosh d \left( 0, x \right)} \right| \mathrm{d}x.
			\end{align*}
			By using Equation 15.4.23 of \cite{Olver}, we see that the term
            $$\left| \frac{{}_2F_1 \left( \frac{\alpha - k - \mu_2 + \gamma_1 + 1}{2}, \frac{\alpha}{2}; \frac{\gamma_1 + \mu_1 + \alpha + n - k}{2}; \tanh^2 d \left( 0, x \right) \right)}{\cosh d \left( 0, x \right)} \right|,$$
            is bounded on $\mathbb{H}^n$. Hence, we obtain,
			$$\left| \int\limits_{\Xi_k \left( \mathbb{H}^n \right)} R_{\alpha}f \left( \xi \right) \sinh^{\gamma_1 + \mu_1} d \left( 0, \xi \right) \cosh^{\gamma_2 + \mu_2} d \left( 0, \xi \right) \mathrm{d}\xi \right| \leq C \| f \|_{L^1_{\mu_1, \mu_2} \left( \mathbb{H}^n \right)}.$$
			This completes the proof! The optimality of the conditions imposed on $\mu_1$, $\mu_2$ and $p$ is proved in Examples \ref{CounterExampleExistence1} and \ref{CounterExampleExistence2}.
		\end{proof}
		We now show that for $\alpha \geq 1$ and $\mu_1 + \mu_2 > \alpha - n$, the condition $1 \leq p < \frac{\mu_1 + \mu_2 + n - 1}{\alpha - 1}$ is optimal.
		\begin{example}
			\label{CounterExampleExistence1}
			\normalfont
			For $\alpha \geq 1$ and $\mu_1 + \mu_2 > \alpha - n$, we fix $p \geq \frac{\mu_1 + \mu_2 + n - 1}{\alpha - 1}$. Consider the function $f: \mathbb{H}^n \rightarrow \mathbb{C}$ defined as
			$$f \left( x \right) = \frac{\sinh^{\frac{2 - n - \mu_1}{p}} d \left( 0, x \right) \cosh^{- \frac{\mu_2}{p}} d \left( 0, x \right)}{\left( 1 + \cosh d \left( 0, x \right) \right)^{\frac{1}{p}} \ln \left( 1 + \cosh d \left( 0, x \right) \right)}.$$
			It was shown in Example 3.2 of \cite{SelfLPImproving} that $f \in L^p_{\mu_1, \mu_2} \left( \mathbb{H}^n \right)$. On the other hand, using Equation \eqref{FormualRadialFunctionsEquation}, we get
            \begin{align*}
                &R_{\alpha}f \left( \xi \right) \\
                &= \frac{C}{\cosh^{\alpha - 1} d \left( 0, \xi \right)} \int\limits_{\cosh d \left( 0, \xi \right)}^{\infty} \frac{\left( t^2 - 1 \right)^{\frac{2 - n - \mu_1}{2p}} t^{- \frac{\mu_2}{p}}}{\left( 1 + t \right)^{\frac{1}{p}} \ln \left( 1 + t \right)} \left( t^2 - \cosh^2 d \left( 0, x \right) \right)^{\frac{\alpha}{2} - 1} \mathrm{d}t.
            \end{align*}
			Upon substituting $\ln \left( 1 + t \right) = u$ and simplifying, we get
            \begin{align*}
                R_{\alpha}f \left( \xi \right) = \frac{C}{\cosh^{\alpha - 1} d \left( 0, \xi \right)} \int\limits_{\ln \left( 1 + \cosh d \left( 0, \xi \right) \right)}^{\infty} &\left( \left( e^u - 1 \right)^2 - 1 \right)^{\frac{2 - n - \mu_1}{2p}} \left( e^u - 1 \right)^{- \frac{\mu_2}{p}} \times \\
                &e^{\frac{u}{p'}} \left( \left( e^u - 1 \right)^2 - \cosh^2 d \left( 0, \xi \right) \right)^{\frac{\alpha}{2} - 1} \frac{\mathrm{d}u}{u}.
            \end{align*}
			It is clear that the total power of $e^u$ towards $+\infty$ is $\alpha - 1 - \frac{\mu_1 + \mu_2 + n - 1}{p} \geq 0$. Hence, $R_{\alpha}f \left( \xi \right) = + \infty$, for every $\xi \in \Xi_k \left( \mathbb{H}^n \right)$.
		\end{example}
		Now, we consider the case $\mu_1 + \mu_2 \leq \alpha - n$ for $\alpha > 0$.
		\begin{example}
			\label{CounterExampleExistence2}
			\normalfont
			Let $\alpha > 0$ and $\mu_1 + \mu_2 \leq \alpha - n$. For $p \geq 1$, consider the function $f: \mathbb{H}^n \rightarrow \mathbb{C}$ defined by
			$$f \left( x \right) = \frac{\sinh^{- \frac{\mu_1}{p}} d \left( 0, x \right) \cosh^{\frac{1 - \mu_2}{p}} d \left( 0, x \right) \left( 2 + \sinh d \left( 0, x \right) \right)^{- \frac{n}{p}}}{\ln^{1 + \frac{1}{p}} \left( 2 + \sinh d \left( 0, x \right) \right)}.$$
			It was shown in Example 3.3 of \cite{SelfLPImproving} that $f \in L^p_{\mu_1, \mu_2} \left( \mathbb{H}^n \right)$. As for the weighted $k$-plane transform of $f$, we have, using Equation \eqref{FormualRadialFunctionsEquation},
            \begin{align*}
                R_{\alpha}f \left( \xi \right) = &\frac{C}{\cosh^{\alpha - 1} d \left( 0, \xi \right)} \\
                &\int\limits_{\cosh d \left( 0, \xi \right)} \frac{\left( t^2 - 1 \right)^{- \frac{\mu_1}{2p}} t^{\frac{1 - \mu_2}{p}} \left( 2 + \sqrt{t^2 - 1} \right)^{- \frac{n}{p}}}{\ln^{1 + \frac{1}{p}} \left( 2 + \sqrt{t^2 - 1} \right)} \left( t^2 - \cosh^2 d \left( 0, \xi \right) \right)^{\frac{\alpha}{2} - 1} \mathrm{d}t.
            \end{align*}
			Next, we substitute $\ln \left( 2 + \sqrt{t^2 - 1} \right) = u$ and simplify to obtain
            \begin{align*}
                R_{\alpha}f \left( \xi \right) = \frac{C}{\cosh d \left( 0, \xi \right)} \int\limits_{\ln \left( 2 + \sinh d \left( 0, \xi \right) \right)}^{\infty} &\left( e^u - 2 \right)^{1 - \frac{\mu_1}{p}} \left( 1 + \left( e^u - 2 \right)^2 \right)^{\frac{1}{ - \mu_2	2p} - \frac{1}{2}} \times \\
                &e^{\left( 1 - \frac{n}{p} \right)u} \left( \left( e^u - 2 \right)^2 - \sinh^2 d \left( 0, x \right) \right)^{\frac{\alpha}{2} - 1} \frac{\mathrm{d}u}{u^{1 + \frac{1}{p}}}.
            \end{align*}
			It is clear at once that the total power of $e^u$ towards $+\infty$ is $\alpha - 1 - \frac{\mu_1 + \mu_2 + n - 1}{p} > 0$ whenever $\mu_1 + \mu_2 < \alpha - n$ and $p \geq 1$ or $\mu_1 + \mu_2 = \alpha - n$ and $p > 1$. In these cases, we have $R_{\alpha}f \left( \xi \right) = + \infty$, for every $\xi \in \Xi_k \left( \mathbb{H}^n \right)$.
		\end{example}

        \begin{remark}
            \normalfont
            At the moment, we do not have the optimality of the condition $\alpha \geq 1$ mentioned in Theorem \ref{ExistenceTheorem}. However, we conjecture that the condition is indeed optimal. This is due to the fact that the condition $\alpha \geq 1$ appears naturally in the choice of $\gamma_1$ and $\gamma_2$ in the proof of Theorem \ref{ExistenceTheorem} and we believe that some modifications to the counterexamples given above should give the optimality. In the sequel, we shall keep this condition with us for all the boundedness results.
        \end{remark}
	\section{\texorpdfstring{$L^p$}{Lp}-improving boundedness for radial functions}
		\label{LPImprovingSection}
		With the knowledge of existence of the weighted $k$-plane transform on the spaces $L^p_{\mu_1, \mu_2} \left( \mathbb{H}^n \right)$, we now move on to understand the mapping properties of the operator $R_{\alpha}$. It is to be noted that the complete mapping properties of the $k$-plane transform (the case $\alpha = k$) are not well understood for generic functions in $L^p_{\mu_1, \mu_2} \left( \mathbb{H}^n \right)$. Therefore, as a first step, we only focus on radial functions. Particularly, in this section, we wish to answer the following two questions.
		\begin{question}
			\label{LPImprovingQuestion}
			What are the admissible values of $\alpha \geq 1$, $\mu_1, \mu_2, \nu_1, \nu_2 \in \mathbb{R}$ and $r \geq p \geq 1$ such that the following inequality holds for all radial functions $f \in L^p_{\mu_1, \mu_2} \left( \mathbb{H}^n \right)$?
			\begin{equation}
				\label{LPImprovngInequality}
				\| R_{\alpha}f \|_{L^r_{\nu_1, \nu_2} \left( \Xi_k \left( \mathbb{H}^n \right) \right)} \leq C \| f \|_{L^p_{\mu_1, \mu_2} \left( \mathbb{H}^n \right)}.
			\end{equation}
			Here, the constant $C$ depends only on $\alpha, \mu_1, \mu_2, \nu_1, \nu_2, n, p$ and $r$, and not on the function $f$.
		\end{question}
		Similarly, we also ask the question when $r = +\infty$.
		\begin{question}
			\label{LPLInfinityQuestion}
			What are the admissible values of $\alpha \geq 1$, $\mu_1, \mu_2, \nu_1, \nu_2 \in \mathbb{R}$ and $p \geq 1$ such that the following inequality holds for all radial functions $f \in L^p_{\mu_1, \mu_2} \left( \mathbb{H}^n \right)$?
			\begin{equation}
				\label{LPLInfinityInequality}
				\| \sinh^{\nu_1} d \left( 0, \cdot \right) \cosh^{\nu_2} d \left( 0, \cdot \right) R_{\alpha}f \left( \cdot \right) \|_{L^{\infty} \left( \Xi_k \left( \mathbb{H}^n \right) \right)} \leq C \| f \|_{L^p_{\mu_1, \mu_2} \left( \mathbb{H}^n \right)}.
			\end{equation}
			Here, the constant $C$ depends only on $\alpha, \mu_1, \mu_2, \nu_1, \nu_2$ and $p$, and not on the choice of the function $f$.
		\end{question}
		\begin{remark}
			\normalfont
			Due to the conditions of existence of the transform $R_{\alpha}$ on the weighted Lebesgue space $L^p_{\mu_1, \mu_2} \left( \mathbb{H}^n \right)$ as described in Theorem \ref{ExistenceTheorem}, it is imperative that $\alpha \geq 1$, $\mu_1 + \mu_2 > k - n$ and $1 \leq p < \frac{\mu_1 + \mu_2 + n - 1}{\alpha - 1}$ hold to answer Questions \ref{LPImprovingQuestion} and \ref{LPLInfinityQuestion}. These assumptions are taken throughout the sequel and are considered to be necessary for the results of this section to hold.
		\end{remark}
		We begin by obtaining necessary conditions for Inequality \eqref{LPImprovngInequality} to hold.
		\begin{theorem}
			\label{LPImprovingNecessary}
			Let $\alpha \geq 1$, $\mu_1 + \mu_2 > \alpha - n$, and $1 \leq p < \frac{\mu_1 + \mu_2 + n - 1}{\alpha - 1}$. Then, $R_{\alpha}$ is bounded from $L^p_{\mu_1, \mu_2} \left( \mathbb{H}^n \right)$ to $L^r_{\nu_1, \nu_2} \left( \Xi_k \left( \mathbb{H}^n \right) \right)$ only if
			\begin{equation}
				\label{LPImprovingNecessaryEquation}
				\nu_1 > k - n \text{ and } \frac{\nu_2 + k - 1}{r} - \frac{\mu_2 - 1}{p} \leq \frac{\mu_1 + n}{p} - \frac{\nu_1 + n - k}{r} \leq \alpha,
			\end{equation}
			with strict inequality on the left side above when $\nu_1 + \nu_1 \leq r \left( \alpha - 1 \right) - \left( n - 1 \right)$.
		\end{theorem}
		\begin{proof}
			We prove the necessary conditions stated in Equation \eqref{LPImprovingNecessaryEquation} in two cases.
			
			\textbf{Case I:} First, we consider $\mu_1 > - n$. Then, for $\lambda > 0$, we look at the function $f_{\lambda}: \mathbb{H}^n \rightarrow \mathbb{C}$ given by $f_{\lambda} \left( x \right) = \chi_{B \left( 0, \lambda \right)} \left( x \right) \cosh d \left( 0, x \right)$. It is clear that $f_{\lambda} \in L^p_{\mu_1, \mu_2} \left( \mathbb{H}^n \right)$, whenever $\mu_1 > -n$. Moreover, we have, by using polar decomposition on $\mathbb{H}^n$,
			\begin{align*}
				\| f_{\lambda} \|_{L^p_{\mu_1, \mu_2} \left( \mathbb{H}^n \right)} &= C \left[ \int\limits_{0}^{\lambda} \sinh^{\mu_1 + n - 1} t \cosh^{\mu_2 + p} t \ \mathrm{d}t \right]^{p}.
			\end{align*}
			By substituting $\sinh^2 t = v \sinh^2 \lambda$ and simplifying, we obtain,
			\begin{align*}
				\| f_{\lambda} \|_{L^p_{\mu_1, \mu_2} \left( \mathbb{H}^n \right)} &= C \sinh^{\frac{\mu_1 + n}{p}} \lambda \left[ \int\limits_{0}^{1} v^{\frac{\mu_1 + n}{2} - 1} \left( 1 + v \sinh^2 \lambda \right)^{\frac{\mu_2 - 1 + p}{2}} \mathrm{d}v \right]^p.
			\end{align*}
			Using Equation 15.6.1 of \cite{Olver} followed by an application of Equation 15.8.1 in \cite{Olver}, we obtain
			\begin{align*}
				\| f_{\alpha} \|_{L^p_{\mu_1, \mu_2} \left( \mathbb{H}^n \right)} &= C \sinh^{\frac{\mu_1 + n}{p}} \lambda \cosh^{\frac{\mu_2 - 1}{p} + 1} \lambda \times \\
                &\left[ {}_2F_1 \left( \frac{1 - \mu_2 - p}{2}, 1; 1 + \frac{\mu_1 + n}{2}; \tanh^2 \lambda \right) \right]^{\frac{1}{p}}.
			\end{align*}
			Due to Equation 15.4.1 of \cite{Olver}, we have the following estimate.
			\begin{equation}
				\label{LPNormFBallEstimate}
				\| f_{\lambda} \|_{L^p_{\mu_1, \mu_2} \left( \mathbb{H}^n \right)} \leq \begin{cases}
	\lambda^{\frac{\mu_1 + n}{p}}, & \lambda \text{ is small}. \\
	e^{\left( \frac{\mu_1 + \mu_2 + n - 1}{p} + 1 \right) \lambda}, & \lambda \text{ is large}.	
\end{cases}
			\end{equation}
			On the other hand, we have from Equations \eqref{XiKMeasure} and \eqref{FormualRadialFunctionsEquation},
			\begin{align*}
				&\| R_{\alpha}f_{\lambda} \|_{L^r_{\nu_1, \nu_2} \left( \Xi_k \left( \mathbb{H}^n \right) \right)} \\
                &= C \left[ \int\limits_{0}^{\infty} \left| \frac{1}{\cosh^{\alpha - 1} r} \int\limits_{\cosh r}^{\infty} \tilde{f}_{\lambda} \left( t \right) \left( t^2 - 1 \right)^{\frac{\alpha}{2} - 1} \mathrm{d}t \right|^r \sinh^{n - k - 1 - \nu_1} r \cosh^{k + \nu_2} r \ \mathrm{d}r \right]^{\frac{1}{r}},
			\end{align*}
			where, $\tilde{f}_{\lambda} \left( t \right) = \chi_{\left( 1, \cosh \lambda \right)} \left( t \right) t$. Substituting $\cosh r = s$ and simplifying, we get
			\begin{align*}
				&\| R_{\alpha}f_{\lambda} \|_{L^r_{\nu_1, \nu_2} \left( \Xi_k \left( \mathbb{H}^n \right) \right)} \\
                &= C \left[ \int\limits_{1}^{\cosh \lambda} \left( \cosh^2 \lambda - s^2 \right)^{\frac{r \alpha}{2}} \left( s^2 - 1 \right)^{\frac{\nu_1 + n - k}{2}  - 1} s^{k + \nu_2 - r \left( \alpha - 1 \right)} \mathrm{d}s \right]^{\frac{1}{r}}.
			\end{align*}
			Next, we substitute $s^2 = 1 + u \sinh^2 \lambda$ and simplify to obtain,
			\begin{align*}
				\| R_{\alpha}f_{\lambda} \|_{L^r_{\nu_1, \nu_2} \left( \Xi_k \left( \mathbb{H}^n \right) \right)} &= C \sinh^{\alpha + \frac{\nu_1 + n - k}{r}} \lambda \times \\
                &\left[ \int\limits_{0}^{1} u^{\frac{\nu_1 + n - k}{2} - 1} \left( 1 - u \right)^{\frac{r \alpha}{2}} \left( 1 + u \sinh^2 \lambda \right)^{\frac{k - 1 - r \left( \alpha - 1 \right) + \nu_2}{2}} \mathrm{d}u \right]^{\frac{1}{r}}.
			\end{align*}
			It is clear that if $v_1 \leq k - n$, the integral in the above equation is divergent and therefore $R_{\alpha}f_{\lambda} \notin L^r_{\nu_1, \nu_2} \left( \Xi_k \left( \mathbb{H}^n \right) \right)$. That is, $v_1 > k - n$ is a necessary condition for Inequality \eqref{LPImprovngInequality} to hold. Under this assumption on $\nu_1$, using Equation 15.6.1 and Equation 15.8.1 of \cite{Olver}, we get
			\begin{align*}
				&\| R_{\alpha}f_{\lambda} \|_{L^r_{\nu_1, \nu_2} \left( \Xi_k \left( \mathbb{H}^n \right) \right)} \\
                &= C \sinh^{\alpha + \frac{\nu_1 + n - k}{2}} \lambda \cosh^{\frac{\nu_1 + k - 1 - r \left( \alpha - 1 \right)}{r}} \lambda \times \\
				&\left[ {}_2F_1 \left( \frac{r \left( \alpha - 1 \right) - k + 1 - \alpha_2}{2}, 1 + \frac{r \alpha}{2}; 1 + \frac{\nu_1 + r \alpha + n - k}{2}; \tanh^2 r \right) \right]^{\frac{1}{r}}.
			\end{align*}
			Now, using Equations 15.4.20--15.4.23 of \cite{Olver}, we get the following estimate.
			\begin{equation}
				\label{LRNormRAlphsEstimate}
                \begin{aligned}
				&\| R_{\alpha}f_{\lambda} \|_{L^r_{\nu_1, \nu_2} \left( \Xi_k \left( \mathbb{H}^n \right) \right)} \\
                &\geq C \begin{cases}
											\lambda^{\alpha + \frac{\nu_1 + n - k}{r}}, & \lambda \text{ is small}. \\
											e^{\left( \frac{\nu_1 + \nu_2 + n - 1}{r} + 1 \right) \lambda}, & \lambda \text{ is large and } \nu_1 + \nu_2 > r \left( \alpha - 1 \right) - n + 1. \\
											e^{\left( \frac{\nu_1 + \nu_2 + n - 1}{r} + 1 \right) \lambda} \ln \cosh \lambda, & \lambda \text{ is large and } \nu_1 + \nu_2 = r \left( \alpha - 1 \right) - n + 1. \\
											e^{\alpha \lambda}, & \lambda \text{ is large and } \nu_1 + \nu_2 < r \left( \alpha - 1 \right) - n + 1.
										\end{cases}
                \end{aligned}
			\end{equation}
			We now test Inequality \eqref{LPImprovngInequality} that we desire against the function $f_{\lambda}$, for varying $\lambda > 0$. Using the estimates given in Equations \eqref{LPNormFBallEstimate} and \eqref{LRNormRAlphsEstimate}, we get for small $\lambda > 0$,
			$$\lambda^{\alpha + \frac{\nu_1 + n - k}{r}} \leq C \lambda^{\frac{\mu_1 + n}{p}}.$$
			As $\lambda \rightarrow 0$, the above inequality forces
			$$\frac{\mu_1 + n}{p} - \frac{\nu_1 + n - k}{r} \leq \alpha.$$
			Next, for large $\lambda > 0$, we have the following three cases. First, if $\nu_1 + \nu_2 > r \left( \alpha - 1 \right) - n + 1$, then, using Equations \eqref{LPNormFBallEstimate} and \eqref{LRNormRAlphsEstimate} in Inequality \eqref{LPImprovngInequality}, we get
			$$e^{\left( \frac{\nu_1 + \nu_2 + n - 1}{r} + 1 \right) \lambda} \leq C e^{\left( \frac{\mu_1 + \mu_2 + n - 1}{p} + 1 \right) \lambda}.$$
			As $\lambda \rightarrow \infty$, the above inequality forces
			$$\frac{\nu_1 + \nu_2 + n - 1}{r} \leq \frac{\mu_1 + \mu_2 + n - 1}{p},$$
			or, what is the same,
			$$\frac{\nu_2 + n - k}{r} - \frac{\mu_2 - 1}{p} \leq \frac{\mu_1 + n}{p} - \frac{\nu_1 + n - k}{r}.$$
			Next, we consider the case when $\nu_1 + \nu_2 = r \left( \alpha - 1 \right) - n + 1$. Also, we observe that for $\lambda > 0$ large, $\cosh \lambda$ behaves as $e^{\lambda}$ and therefore, $\ln \cosh \lambda$ behaves as $\lambda$. Hence, using Equations \eqref{LPNormFBallEstimate} and \eqref{LRNormRAlphsEstimate} in Inequality \eqref{LPImprovngInequality}, we get
			$$\lambda e^{\left( \frac{\nu_1 + \nu_2 + n - 1}{r} + 1 \right) \lambda} \leq C e^{\left( \frac{\mu_1 + \mu_2 + n - 1}{p} + 1 \right) \lambda}.$$
			As $\lambda \rightarrow \infty$, the above inequality forces
			$$\frac{\nu_1 + \nu_2 + n - 1}{r} < \frac{\mu_1 + \mu_2 + n - 1}{p}.$$
			The equality is not expected in this case due to the presence of a linear factor on the left hand side of the desired estimate.
			
			Finally, for $\nu_1 + \nu_2 < r \left( \alpha - 1 \right) - n + 1$, we get
			$$e^{\alpha \lambda} \leq C e^{\left( \frac{\mu_1 + \mu_2 + n - 1}{p} + 1 \right) \lambda}.$$
			As $\lambda \rightarrow \infty$, the above inequality forces
			$$\alpha \leq \frac{\mu_1 + \mu_2 + n - 1}{p} + 1,$$
			which is redundant due to the fact that $p < \frac{\mu_1 + \mu_2 + n - 1}{\alpha - 1}$. This gives the necessary conditions mentioned in Equation \eqref{LPImprovingNecessaryEquation} when $\mu_1 > -n$.
			
			\textbf{Case II:} When $\mu_2 \leq -n$, we consider the function $f_{\lambda}: \mathbb{H}^n \rightarrow \mathbb{C}$ defined by $f_{\lambda} \left( x \right) = \chi_{B \left( 0, \lambda \right)} \left( x \right) \sinh^{- \frac{\mu_1}{p}} d \left( 0, x \right) \cosh d \left( 0, x \right)$, and proceed as in Case I to get conditions of Equation \eqref{LPImprovingNecessaryEquation}.
		\end{proof}
		Next, we see that the conditions mentioned in Theorem \ref{LPImprovingNecessary} are also sufficient.
		\begin{theorem}
			\label{LPImprovingSufficient}
			Let $\alpha \geq 1$, $\mu_1 + \mu_2 > \alpha - n$ and $1 \leq p < \frac{\mu_1 + \mu_2 + n - 1}{\alpha - 1}$ and $r \geq p$.
			\begin{enumerate}
				\item[\mylabel{LPLRAlphaMoreThan2}{(A)}] For $\alpha \geq 2$, Inequality \eqref{LPImprovngInequality} holds if and only if $\nu_1$ and $\nu_2$ satisfy the conditions given in Equation \eqref{LPImprovingNecessaryEquation}.
				\item[\mylabel{LPLRAlphaLessThan2}{(B)}] For $1 \leq \alpha < 2$, Inequality \eqref{LPImprovngInequality} holds if and only if $\nu_1$ and $\nu_2$ satisfy the conditions given in Equation \eqref{LPImprovingNecessaryEquation} and $\frac{1}{p} - \frac{1}{r} \leq \frac{\alpha}{2}$.
			\end{enumerate}
		\end{theorem}
		\begin{proof}
			The necessity of the conditions of Equation \eqref{LPImprovingNecessaryEquation} is shown in Theorem \ref{LPImprovingNecessary}. We only focus on their sufficiency in this proof.
			\begin{enumerate}
				\item[\mylabel{LPLRAlphaMoreThan2Proof}{(A)}] First, we take up the case when $\alpha \geq 2$. For radial functions $f \in L^p_{\mu_1, \mu_2} \left( \mathbb{H}^n \right)$, we have from Equation \eqref{FormualRadialFunctionsEquation} and writing $\cosh r = s$, where $r = d \left( 0, \xi \right)$,
				\begin{align*}
					&\| R_{\alpha}f \|_{L^r_{\nu_1, \nu_2} \left( \Xi_k \left( \mathbb{H}^n \right) \right)} \\
                    &=  C \left[ \int\limits_{1}^{\infty} \left| \int\limits_{s}^{\infty} \tilde{f} \left( t \right) \left( t^2 - s^2 \right)^{\frac{\alpha}{2} - 1} \mathrm{d}t \right|^r s^{\nu_2 + k - r \left( \alpha - 1 \right)} \left( s^2 - 1 \right)^{\frac{\nu_1 + n - k}{2} - 1} \mathrm{d}s \right]^{\frac{1}{r}}.
				\end{align*}
				Substituting $t^2 - 1 = v$ and $s^2 - 1 = u$ and simplifying, we get
				\begin{equation}
					\label{LRNormRepresentationRAlpha}
					\begin{aligned}
						&\| R_{\alpha}f \|_{L^r_{\nu_1, \nu_2} \left( \Xi_k \left( \mathbb{H}^n \right) \right)} \\
                        &= C \left[ \int\limits_{0}^{\infty} \left| \int\limits_{u}^{\infty} \tilde{f} \left( \sqrt{1 + v} \right) v^{\frac{\mu_1 + n}{2p}} \left( 1 + v \right)^{\frac{\mu_2 - 1}{2p}} \left( 1 - \frac{u}{v} \right)^{\frac{\alpha}{2} - 1} \left( \frac{u}{v} \right)^{\frac{\nu_1 + n - k}{2r}} \times \right. \right. \\
						&\left. \left. \frac{v^{\frac{\alpha}{2} + \frac{\nu_1 + n - k}{2r} - \frac{\mu_1 + n}{2p}} \left( 1 + u \right)^{\frac{\nu_2 + k - 1 - r \left( \alpha - 1 \right)}{2r}}}{\left( 1 + v \right)^{\frac{1}{2} + \frac{\mu_2 - 1}{2p}}} \frac{\mathrm{d}v}{v} \right|^r \frac{\mathrm{d}u}{u} \right]^{\frac{1}{r}}.
					\end{aligned}
				\end{equation}
				Now, by using triangle inequality of the $L^r$-norm, we get
				\begin{align*}
					&\| R_{\alpha}f \|_{L^r_{\nu_1, \nu_2} \left( \Xi_k \left( \mathbb{H}^n \right) \right)} \leq C \left[ \int\limits_{0}^{\infty} \left| \int\limits_{u}^{\infty} \tilde{f} \left( \sqrt{1 + v} \right) \chi_{\left( 0, 1 \right)} \left( v \right) v^{\frac{\mu_1 + n}{2p}} \left( 1 + v \right)^{\frac{\mu_2 - 1}{2p}} \times \right. \right. \\
					&\left. \left. \left( 1 - \frac{u}{v} \right)^{\frac{\alpha}{2} - 1} \left( \frac{u}{v} \right)^{\frac{\nu_1 + n - k}{2r}} \frac{v^{\frac{\alpha}{2} + \frac{\nu_1 + n - k}{2r} - \frac{\mu_1 + n}{2p}} \left( 1 + u \right)^{\frac{\nu_2 + k - 1 - r \left( \alpha - 1 \right)}{2r}}}{\left( 1 + v \right)^{\frac{1}{2} + \frac{\mu_2 - 1}{2p}}} \frac{\mathrm{d}v}{v} \right|^r \frac{\mathrm{d}u}{u} \right]^{\frac{1}{r}} \\
					&+ C \left[ \int\limits_{0}^{\infty} \left| \int\limits_{u}^{\infty} \tilde{f} \left( \sqrt{1 + v} \right) \chi_{\left( 1, \infty \right)} \left( v \right) v^{\frac{\mu_1 + n}{2p}} \left( 1 + v \right)^{\frac{\mu_2 - 1}{2p}} \times \right. \right. \\
					&\left. \left. \left( 1 - \frac{u}{v} \right)^{\frac{\alpha}{2} - 1} \left( \frac{u}{v} \right)^{\frac{\nu_1 + n - k}{2r}} \frac{v^{\frac{\alpha}{2} + \frac{\nu_1 + n - k}{2r} - \frac{\mu_1 + n}{2p}} \left( 1 + u \right)^{\frac{\nu_2 + k - 1 - r \left( \alpha - 1 \right)}{2r}}}{\left( 1 + v \right)^{\frac{1}{2} + \frac{\mu_2 - 1}{2p}}} \frac{\mathrm{d}v}{v} \right|^r \frac{\mathrm{d}u}{u} \right]^{\frac{1}{r}}.
				\end{align*}
				Now, we notice that when $v \leq 1$, we also require $u \leq 1$, for otherwise the integral in the first term of the above inequality is zero. With these assumptions and Equation \eqref{LPImprovingNecessaryEquation}, it is easy to see that
				$$\frac{v^{\frac{\alpha}{2} + \frac{\nu_1 + n - k}{2r} - \frac{\mu_1 + n}{2p}} \left( 1 + u \right)^{\frac{\nu_2 + k - 1 - r \left( \alpha - 1 \right)}{2r}}}{\left( 1 + v \right)^{\frac{1}{2} + \frac{\mu_2 - 1}{2p}}} \leq C,$$
				whenever $v \in \left( 0, 1 \right)$. Thus, we have,
				\begin{align*}
					&\| R_{\alpha}f \|_{L^r_{\nu_1, \nu_2} \left( \Xi_k \left( \mathbb{H}^n \right) \right)} \\
                    &\leq C \left[ \int\limits_{0}^{\infty} \left| \int\limits_{u}^{\infty} \tilde{f} \left( \sqrt{1 + v} \right) \chi_{\left( 0, 1 \right)} \left( v \right) v^{\frac{\mu_1 + n}{2p}} \left( 1 + v \right)^{\frac{\mu_2 - 1}{2p}} \times \right. \right. \\
                    &\left. \left. \left( 1 - \frac{u}{v} \right)^{\frac{\alpha}{2} - 1} \left( \frac{u}{v} \right)^{\frac{\nu_1 + n - k}{2r}}  \frac{\mathrm{d}v}{v} \right|^r \frac{\mathrm{d}u}{u} \right]^{\frac{1}{r}} \\
					&+ C \left[ \int\limits_{0}^{\infty} \left| \int\limits_{u}^{\infty} \tilde{f} \left( \sqrt{1 + v} \right) \chi_{\left( 1, \infty \right)} \left( v \right) v^{\frac{\mu_1 + n}{2p}} \left( 1 + v \right)^{\frac{\mu_2 - 1}{2p}} \left( 1 - \frac{u}{v} \right)^{\frac{\alpha}{2} - 1} \left( \frac{u}{v} \right)^{\frac{\nu_1 + n - k}{2r}} \times \right. \right. \\
					&\left. \left. \left( \frac{1 + u}{1 + v} \right)^{\frac{\nu_2 + k - 1}{2r} - \frac{\alpha - 1}{2}} \left( 1 + v \right)^{\frac{\nu_1 + \nu_2 + n - 1}{2r} - \frac{\mu_1 + \mu_2 + n - 1}{2p}} \frac{\mathrm{d}v}{v} \right|^r \frac{\mathrm{d}u}{u} \right]^{\frac{1}{r}}.
				\end{align*}
				Now, since $\nu_2$ is a power of $\cosh$, we may assume that $\frac{\nu_2 + k - 1}{2r} \geq \frac{\alpha - 1}{2}$. This is because if $\nu_2' \leq \nu_2$, then we have,
				$$\| R_{\alpha}f \|_{L^r_{\nu_1, \nu_2'} \left( \Xi_k \left( \mathbb{H}^n \right) \right)} \leq \| R_{\alpha}f \|_{L^r_{\nu_1, \nu_2} \left( \Xi_k \left( \mathbb{H}^n \right) \right)}.$$
				With the assumption on $\nu_2$, we see that $\left( \frac{1 + u}{1 + v} \right)^{\frac{\nu_2 + k - 1}{2r} - \frac{\alpha - 1}{2}} \leq 1$, since $v \geq u$. Also, due to Equation \eqref{LPImprovingNecessaryEquation}, we have $\left( 1 + v \right)^{\frac{\nu_1 + \nu_2 + n - 1}{2r} - \frac{\mu_1 + \mu_2 + n - 1}{2p}} \leq 1$. Therefore, we get,
				\begin{align*}
					&\| R_{\alpha}f \|_{L^r_{\nu_1, \nu_2} \left( \Xi_k \left( \mathbb{H}^n \right) \right)} \\
                    &\leq C \left[ \int\limits_{0}^{\infty} \left| \int\limits_{u}^{\infty} \tilde{f} \left( \sqrt{1 + v} \right) \chi_{\left( 0, 1 \right)} \left( v \right) v^{\frac{\mu_1 + n}{2p}} \left( 1 + v \right)^{\frac{\mu_2 - 1}{2p}} \times \right. \right. \\
                    &\left. \left. \left( 1 - \frac{u}{v} \right)^{\frac{\alpha}{2} - 1} \left( \frac{u}{v} \right)^{\frac{\nu_1 + n - k}{2r}}  \frac{\mathrm{d}v}{v} \right|^r \frac{\mathrm{d}u}{u} \right]^{\frac{1}{r}} \\
					&+ C \left[ \int\limits_{0}^{\infty} \left| \int\limits_{u}^{\infty} \tilde{f} \left( \sqrt{1 + v} \right) \chi_{\left( 1, \infty \right)} \left( v \right) v^{\frac{\mu_1 + n}{2p}} \left( 1 + v \right)^{\frac{\mu_2 - 1}{2p}} \times \right. \right. \\
                    &\left. \left. \left( 1 - \frac{u}{v} \right)^{\frac{\alpha}{2} - 1} \left( \frac{u}{v} \right)^{\frac{\nu_1 + n - k}{2r}} \frac{\mathrm{d}v}{v} \right|^r \frac{\mathrm{d}u}{u} \right]^{\frac{1}{r}} \\
					&\leq C \| F * G \|_{L^r \left( \left( 0, \infty \right), \frac{\mathrm{d}t}{t} \right)},
				\end{align*}
				where,
				\begin{equation}
					\label{ConvolutionFLPLR}
					F \left( t \right) = \tilde{f} \left( \sqrt{1 + t} \right) t^{\frac{\mu_1 + n}{2p}} \left( 1 + v \right)^{\frac{\mu_2 - 1}{2p}},
				\end{equation}
				and
				\begin{equation}
					\label{ConvolutionGLPLR}
					G \left( t \right) = t^{\frac{\nu_1 + n - k}{2r}} \left( 1 - t \right)^{\frac{\alpha}{2} - 1}.
				\end{equation}
				Using Young's convolution inequality, we obtain
				\begin{align*}
					\| R_{\alpha}f \|_{L^r_{\nu_1, \nu_2} \left( \Xi_k \left( \mathbb{H}^n \right) \right)} \leq C \| F \|_{L^p \left( \left( 0, \infty \right), \frac{\mathrm{d}t}{t} \right)} \| G \|_{L^q \left( \left( 0, \infty \right), \frac{\mathrm{d}t}{t} \right)},
				\end{align*}
				where, $\frac{1}{p} + \frac{1}{q} = 1 + \frac{1}{r}$. It is clear that $\| F \|_{L^p \left( \left( 0, \infty \right), \frac{\mathrm{d}t}{t} \right)} = C \| f \|_{L^p_{\mu_1, \mu_2} \left( \mathbb{H}^n \right)}$. Also, since $\nu_1 > k - n$ (as seen in Equation \eqref{LPImprovingNecessaryEquation}) and $\alpha \geq 2$, we observe that $\| G \|_{L^q \left( \left( 0, \infty \right), \frac{\mathrm{d}t}{t} \right)} < + \infty$. This gives Inequality \eqref{LPImprovngInequality} for radial functions in $L^p_{\mu_1, \mu_2} \left( \mathbb{H}^n \right)$.
				\item[\mylabel{LPLRAlphaLessThan2Proof}{(B)}] Now, we take up the case when $1 \leq \alpha < 2$. First, we observe that when $\frac{1}{p} - \frac{1}{r} < \frac{\alpha}{2}$, the proof of Inequality \eqref{LPImprovngInequality} is verbatim to that of Part \ref{LPLRAlphaMoreThan2Proof} above. We only consider the case when $\frac{1}{p} - \frac{1}{r} = \frac{1}{2}$. In this case, we employ the boundedness of Riemann-Liouville fractional integrals to get the desired result. For a detailed survey on fractional integrals and derivatives, we refer the reader to Chapters 2--4 of \cite{RubinBookFIR}. We begin by rewriting Equation \eqref{LRNormRepresentationRAlpha} as
				\begin{align*}
					&\| R_{\alpha}f \|_{L^r_{\nu_1, \nu_2} \left( \Xi_k \left( \mathbb{H}^n \right) \right)} \\
                    &= C \left[ \int\limits_{0}^{\infty} \left| \int\limits_{u}^{\infty} \frac{\tilde{f} \left( \sqrt{1 + v} \right)}{\sqrt{1 + v}} \left( v - u \right)^{\frac{\alpha}{2} - 1} u^{\frac{\nu_1 + n - k}{2r} - \frac{1}{r}} \left( 1 + u \right)^{\frac{\nu_2 + k - 1}{2r} - \frac{\alpha - 1}{2}} \mathrm{d}v \right| \mathrm{d}u \right]^{\frac{1}{r}}.
				\end{align*}
				Let $\varphi \left( t \right) = \frac{\tilde{f} \left( \sqrt{1 + t} \right)}{\sqrt{1 + t}}$ and $\rho \left( t \right) = t^{\frac{\nu_1 + n - k}{2r} - \frac{1}{r}} \left( 1 + t \right)^{\frac{\nu_2 + k - 1}{2r} - \frac{\alpha - 1}{2}}$. Then, the above equation can be written as
				$$\| R_{\alpha}f \|_{L^r_{\nu_1, \nu_2} \left( \Xi_k \left( \mathbb{H}^n \right) \right)} = C \| \rho I^{\frac{\alpha}{2}}_- \varphi \|_{L^r \left( \left( 0, \infty \right), \mathrm{d}x \right)}.$$
				Using Theorem 3.64 in \cite{RubinBookFIR} about the weighted boundedness of Riemann-Liouville fractional integrals, we get
				\begin{align*}
					\| R_{\alpha}f \|_{L^r_{\nu_1, \nu_2} \left( \Xi_k \left( \mathbb{H}^n \right) \right)} \leq C &\left[ \int\limits_{0}^{\infty} \left| \tilde{f} \left( \sqrt{1 + v} \right) \right|^p v^{\frac{\mu_1 + n}{2} - 1} \left( 1 + v \right)^{\frac{\mu_2 - 1}{2}} \times \right. \\
                    &\left.v^{p \left( \frac{\nu_1 + n - k}{2r} + \frac{\alpha}{2} - \frac{\mu_1 + n}{2p} \right)} \left( 1 + v \right)^{p \left( \frac{\nu_2 + k - 1}{2r} - \frac{\alpha}{2} - \frac{\mu_2 - 1}{2p} \right)} \mathrm{d}v \right]^{\frac{1}{p}}.
				\end{align*}
				Here, we have utilised the fact that $\frac{1}{p} - \frac{1}{r} = \frac{\alpha}{2}$. Now, we observe that Equation \eqref{LPImprovingNecessaryEquation} gives that
				$$\frac{\nu_1 + n - k}{2r} + \frac{\alpha}{2} - \frac{\mu_1 + n}{2p} \geq 0.$$
				Hence, we get
				\begin{align*}
					\| R_{\alpha}f \|_{L^r_{\nu_1, \nu_2} \left( \Xi_k \left( \mathbb{H}^n \right) \right)} \leq C &\left[ \int\limits_{0}^{\infty} \left| \tilde{f} \left( \sqrt{1 + v} \right) \right|^p v^{\frac{\mu_1 + n}{2} - 1} \left( 1 + v \right)^{\frac{\mu_2 - 1}{2}} \times \right. \\
                    &\left. \left( 1 + v \right)^{p \left( \frac{\nu_1 + \nu_2 + n - 1}{2r} - \frac{\mu_2 + \mu_2 + n - 1}{2p} \right)} \mathrm{d}v \right]^{\frac{1}{p}}.
				\end{align*}
				Again, Equation \eqref{LPImprovingNecessaryEquation} suggests that $\frac{\nu_1 + \nu_2 + n - 1}{2r} - \frac{\mu_1 + \mu_2 + n - 1}{2p} \leq 0$. Hence, we get,
				\begin{align*}
					\| R_{\alpha}f \|_{L^r_{\nu_1, \nu_2} \left( \Xi_k \left( \mathbb{H}^n \right) \right)} &\leq C \left[ \int\limits_{0}^{\infty} \left| \tilde{f} \left( \sqrt{1 + v} \right) \right|^p v^{\frac{\mu_1 + n}{2} - 1} \left( 1 + v \right)^{\frac{\mu_2 - 1}{2}} \mathrm{d}v \right]^{\frac{1}{p}} \\
                    &= C \| f \|_{L^p_{\mu_1, \mu_2} \left( \mathbb{H}^n \right)}.
				\end{align*}
				This completes the proof of the sufficiency. The necessity of the condition $\frac{1}{p} - \frac{1}{r} \leq \frac{\alpha}{2}$ is given in Remark \ref{NecessityNewConditionLPLR}.
			\end{enumerate}
		\end{proof}
		\begin{remark}
			\label{NecessityNewConditionLPLR}
			\normalfont
			We now see that when $1 \leq \alpha < 2$, the condition $\frac{1}{p} - \frac{1}{r} \leq \frac{\alpha}{2}$ is necessary. The proof that we give does not explicitly use the fact that $\alpha < 2$. However, let us observe that for $\alpha \geq 2$, the condition is redundant since $\frac{1}{p} - \frac{1}{r} \leq 1 \leq \frac{\alpha}{2}$. Since the operator $R_{\alpha}$ preserves non-negativity of a function, we assume that the functions $f$ considered in this proof are non-negative. To prove the necessity, let us begin by substituting $u = e^x$ and $v = e^y$ in Equation \eqref{LRNormRepresentationRAlpha}. Then, we get,
			\begin{align*}
				&\| R_{\alpha}f \|_{L^r_{\nu_1, \nu_2} \left( \Xi_k \left( \mathbb{H}^n \right) \right)} = C \left[ \int\limits_{\mathbb{R}} \left| \int\limits_{\mathbb{R}} \tilde{f} \left( \sqrt{1 + e^y} \right) e^{\frac{\mu_1 + n}{2p} y} \left( 1 + e^y \right)^{\frac{\mu_2 - 1}{2p}} \left( 1 - e^{x - y} \right)^{\frac{\alpha}{2} - 1} \times \right. \right. \\
				&\left. \left. e^{\left( \frac{\nu_1 + n - k}{2r} \right) \left( x - y \right)} \chi_{\left( - \infty, 0 \right)} \left( x - y \right) \frac{e^{\left( \frac{\alpha}{2} + \frac{\nu_1 + n - k}{2r} - \frac{\mu_1 + n}{2p} \right) y} \left( 1 + e^x \right)^{\frac{\nu_2 + k - 1 - r \left( \alpha - 1 \right)}{2r}}}{\left( 1 + e^y \right)^{\frac{1}{2} + \frac{\mu_2 - 1}{2p}}} \mathrm{d}y \right|^r \mathrm{d}x \right]^{\frac{1}{r}} \\
				&\geq C \left[ \int\limits_{\mathbb{R}} \left| \int\limits_{\mathbb{R}} \tilde{f} \left( \sqrt{1 + e^y} \right) e^{\frac{\mu_1 + n}{2p} y} \left( 1 + e^y \right)^{\frac{\mu_2 - 1}{2p}} \left( 1 - e^{x - y} \right)^{\frac{\alpha}{2} - 1} \times \right. \right. \\
				&\left. \left. e^{\left( \frac{\nu_1 + n - k}{2r} \right) \left( x - y \right)} \chi_{\left( -1, 0 \right)} \left( x - y \right) \frac{e^{\left( \frac{\alpha}{2} + \frac{\nu_1 + n - k}{2r} - \frac{\mu_1 + n}{2p} \right) y} \left( 1 + e^x \right)^{\frac{\nu_2 + k - 1 - r \left( \alpha - 1 \right)}{2r}}}{\left( 1 + e^y \right)^{\frac{1}{2} + \frac{\mu_2 - 1}{2p}}} \mathrm{d}y \right|^r \mathrm{d}x \right]^{\frac{1}{r}}.
			\end{align*}
			Next, let us observe that due to the conditions of Equation \eqref{NecessityNewConditionLPLR} and the restriction that $x - y \in \left( -1, 0 \right)$, we have,
			$$\frac{e^{\left( \frac{\alpha}{2} + \frac{\nu_1 + n - k}{2r} - \frac{\mu_1 + n}{2p} \right) y} \left( 1 + e^x \right)^{\frac{\nu_2 + k - 1 - r \left( \alpha - 1 \right)}{2r}}}{\left( 1 + e^y \right)^{\frac{1}{2} + \frac{\mu_2 - 1}{2p}}} \geq C \left( 1 + e^y \right)^{\frac{\nu_2 + k - 1}{2r} - \frac{\alpha}{2} - \frac{\mu_2 - 1}{2p}}.$$
			Hence, we get
			\begin{align*}
				&\| R_{\alpha}f \|_{L^r_{\nu_1, \nu_2} \left( \Xi_k \left( \mathbb{H}^n \right) \right)} \\
				&\geq C \left[ \int\limits_{\mathbb{R}} \left| \int\limits_{\mathbb{R}} \tilde{f} \left( \sqrt{1 + e^y} \right) e^{\left( \frac{\mu_1 + n}{2p} \right) y} \left( 1 + e^y \right)^{\frac{\nu_2 + k - 1}{2r} - \frac{\alpha}{2}} \left( 1 - e^{x - y} \right)^{\frac{\alpha}{2} - 1} \times \right. \right. \\
                &\left. \left. e^{\left( \frac{\nu_1 + n - k}{2r} \right) \left( x - y \right)} \chi_{\left( -1, 0 \right)} \left( x - y \right) \mathrm{d}y \right|^r \mathrm{d}x \right]^{\frac{1}{r}} \\
				&= C \| \varphi * \psi \|_{L^r \left( \mathbb{R} \right)},
			\end{align*}
			where, 
            $$\varphi \left( x \right) = \tilde{f} \left( \sqrt{1 + e^x} \right) e^{\left( \frac{\nu_1 + n}{2p} \right)x} \left( 1 + e^x \right)^{\frac{\nu_2 + k - 1}{2r} - \frac{\alpha}{2}},$$ and 
            $$\psi \left( x \right) = \left( 1 - e^x \right)^{\frac{\alpha}{2} - 1} e^{\left( \frac{\nu_1 + n - k}{2r} \right)x} \chi_{\left( -1, 0 \right)} \left( x \right).$$
            Due to Equation \eqref{LPImprovingNecessaryEquation}, it is clear that for $f \in L^p_{\mu_1, \mu_2} \left( \mathbb{H}^n \right)$, we have $\varphi \in L^p \left( \mathbb{R} \right)$. Moreover, we may write
			$$\| f \|_{L^p_{\mu_1, \mu_2} \left( \mathbb{H}^n \right)} = \| \varphi \rho \|_{L^p \left( \mathbb{R} \right)},$$
			where, $\rho \left( x \right) = \left( 1 + e^x \right)^{\frac{\mu_2 - 1}{2p} + \frac{\alpha}{2} - \frac{\nu_1 + k - 1}{2r}}$. Hence, Inequality \eqref{LPImprovngInequality} implies the following boundedness inequality of convolutions.
			\begin{equation}
				\label{NeessaryConvolutionInequality}
				\| \varphi * \psi \|_{L^r \left( \mathbb{R} \right)} \leq C \| \varphi \rho \|_{L^p \left( \mathbb{R} \right)}.
			\end{equation}
			Now, we consider $\lambda > 1$ and define $\varphi_{\lambda} \left( x \right) = \varphi \left( \lambda x \right)$. Then, we have,
			$$\| \varphi_{\lambda} \rho \|_{L^p \left( \mathbb{R} \right)} = \frac{1}{\lambda^{\frac{1}{p}}} \left[ \int\limits_{\mathbb{R}} \left| \varphi \left( x \right) \left( 1 + e^{\frac{x}{\lambda}} \right)^{\frac{\mu_2 - 1}{2p} + \frac{\alpha}{2} - \frac{\nu_1 + k - 1}{2r}} \right|^p \mathrm{d}x \right]^{\frac{1}{p}} = \frac{1}{\lambda^{\frac{1}{p}}} \| \varphi \rho_{\frac{1}{\lambda}} \|_{L^p \left( \mathbb{R}^n \right)}.$$
			By an application of dominated convergence theorem, we have that $\| \varphi \rho_{\frac{1}{\lambda}} \|_{L^p \left( \mathbb{R} \right)} \rightarrow C \| \varphi \|_{L^p \left( \mathbb{R} \right)}$, as $\lambda \rightarrow \infty$. Also, by an argument similar to that given in Remark 3.3(i) of \cite{KumarRayWtEst}, we get
			$$\| \varphi_{\lambda} * \psi \|_{L^r \left( \mathbb{R} \right)} \geq \frac{1}{\lambda^{\frac{\alpha}{2} + \frac{1}{r}}} \| \varphi * \psi \|_{L^r \left( \mathbb{R} \right)}.$$
			Using these observations in Inequality \eqref{NecessityNewConditionLPLR}, we get
			$$\frac{1}{\lambda^{\frac{\alpha}{2} + \frac{1}{r} - \frac{1}{p}}} \| \varphi * \psi \|_{L^r \left( \mathbb{R} \right)} \leq C \| \varphi \rho_{\frac{1}{\lambda}} \|_{L^p \left( \mathbb{R} \right)}.$$
			Hence, as $\lambda \rightarrow \infty$, we require $\frac{1}{p} - \frac{1}{r} \leq \frac{\alpha}{2}$.
		\end{remark}
		Theorems \ref{LPImprovingNecessary} and \ref{LPImprovingSufficient} together with Remark \ref{NecessityNewConditionLPLR} completely answer Question \ref{LPImprovingQuestion}. We now move to answer Question \ref{LPLInfinityQuestion}. As before, we first obtain necessary conditions for Inequality \eqref{LPLInfinityInequality} to hold.
		\begin{theorem}
			\label{LPLInfinityNecessary}
			Let $\alpha \geq 1$, $\mu_1 + \mu_2 > \alpha - n$ and $1 \leq p < \frac{\mu_1 + \mu_2 + n - 1}{\alpha - 1}$. Then, Inequality \eqref{LPLInfinityInequality} holds only if
			\begin{equation}
				\label{LPLInfinityNecessaryEquation}
				\nu_1 \geq \max \left\lbrace 0, \frac{\mu_1 + n}{p} - \alpha \right\rbrace \text{ and } \nu_1 + \nu_2 \leq \frac{\mu_1 + \mu_2 + n - 1}{p}.
			\end{equation}
		\end{theorem}
		\begin{proof}
			To obtain the necessary conditions, we consider two cases.
			
			\textbf{Case I:} First, let us assume that $\mu_1 > -n$. In this case, for $\lambda > 0$, we consider the function $f_{\lambda}: \mathbb{H}^n \rightarrow \mathbb{C}$ given by $f_{\lambda} \left( x \right) = \chi_{B \left( 0, \lambda \right)} \left( x \right) \cosh d \left( 0, x \right)$. The estimate for $\| f_{\lambda} \|_{L^p_{\mu_1, \mu_2} \left( \mathbb{H}^n \right)}$ was obtained in Equation \eqref{LPNormFBallEstimate}. On the other hand, using Equation \eqref{FormualRadialFunctionsEquation} and the definition of $f$, it is easy to see that
			\begin{equation}
				\label{RAlphaBallComputation}
				R_{\alpha}f_{\lambda} \left( \xi \right) = C \frac{\sinh^{\alpha} \lambda}{\cosh^{\alpha - 1} d \left( 0, \xi \right)} \left( 1 - \frac{\sinh^2 d \left( 0, \xi \right)}{\sinh^2 \lambda} \right)^{\frac{\alpha}{2}}.
			\end{equation}
			Hence, by choosing $\xi \in \Xi_k \left( \mathbb{H}^n \right)$ such that $\sinh^2 d \left( 0, \xi \right) = \frac{1}{2} \sinh^2 \lambda$, we observe that
            \begin{align*}
                &\| \sinh^{\nu_1} d \left( 0, \cdot \right) \cosh^{\nu_2} d \left( 0, \cdot \right) R_{\alpha}f_{\lambda} \left( \cdot \right) \|_{L^{\infty} \left( \Xi_k \left( \mathbb{H}^n \right) \right)} \\
                &\geq C \sinh^{\nu_1 + \alpha} \lambda \left( 1 + \cosh^2 \lambda \right)^{\nu_2 - \alpha + 1}.
            \end{align*}
			Now, using the fact that for small $\lambda > 0$, $\sinh \lambda$ behaves as $\lambda$ and $\cosh \lambda$ behaves as a constant, and for large $\lambda > 0$, both $\sinh \lambda$ and $\cosh \lambda$ behave as $e^{\lambda}$, we get the following estimate.
			\begin{equation}
				\label{LInfinityNormRAlphaEstimate}
				\| \sinh^{\nu_1} d \left( 0, \cdot \right) \cosh^{\nu_2} d \left( 0, \cdot \right) R_{\alpha}f_{\lambda} \left( \cdot \right) \|_{L^{\infty} \left( \Xi_k \left( \mathbb{H}^n \right) \right)} \geq C \begin{cases}
							\lambda^{\nu_1 + \alpha}, & \lambda \text{ is small}. \\
							e^{\left( \nu_1 + \nu_2 + 1 \right) \lambda}, & \lambda \text{ is large}.
						\end{cases}
			\end{equation}
			Now, we test Inequality \eqref{LPLInfinityInequality} against the function $f_{\lambda}$. Using Equations \eqref{LPNormFBallEstimate} and \eqref{LInfinityNormRAlphaEstimate} and taking $\lambda \rightarrow 0$ and $\lambda \rightarrow \infty$, we get the conditions mentioned in Equation \eqref{LPLInfinityNecessaryEquation}. The condition $\nu_1 \geq 0$ is observed from the fact that for any $\xi \in \Xi_k \left( \mathbb{H}^n \right)$ with $0 \in \xi$, we have from Equation \eqref{RAlphaBallComputation}, $R_{\alpha}f_{\lambda} \left( \xi \right) = C \sinh^{\alpha} \lambda > 0$. Hence, if $\nu_1 < 0$, we would have $\| \sinh^{\nu_1} d \left( 0, \cdot \right) \cosh^{\nu_2} d \left( 0, \cdot \right) R_{\alpha}f_{\lambda} \left( \cdot \right) \|_{L^{\infty} \left( \Xi_k \left( \mathbb{H}^n \right) \right)} = + \infty$ and Inequality \eqref{LPLInfinityInequality} cannot hold.
			
			\textbf{Case II:} When $\mu_1 \leq -n$, we consider the function $f_{\lambda}: \mathbb{H}^n \rightarrow \mathbb{C}$ defined by $f_{\lambda} \left( x \right) = \chi_{B \left( 0, \lambda \right)} \left( x \right) \sinh^{- \frac{\mu_1}{p}} d \left( 0, x \right) \cosh d \left( 0, x \right)$ and proceed as in Case I to obtain the conditions of Equation \eqref{LPLInfinityNecessaryEquation}.
		\end{proof}
		We now see the sufficiency of the conditions on $\nu_1$ and $\nu_2$ in obtaining Inequality \eqref{LPLInfinityInequality}.
		\begin{theorem}
			\label{LPLInfinitySufficient}
			Let $\alpha \geq 1$, $\mu_1 + \mu_2 > \alpha - n$ and $1 \leq p < \frac{\mu_1 + \mu_2 + n - 1}{\alpha - 1}$. Let $\nu_1, \nu_2 \in \mathbb{R}$ satisfy Equation \eqref{LPLInfinityNecessaryEquation}. Further, we assume that when $p = \frac{\mu_1 + n}{\alpha}$. Then, we have the following.
			\begin{enumerate}
				\item[\mylabel{LPLInfinityAlphaMoreThan2}{(A)}] For $\alpha \geq 2$, we have Inequality \eqref{LPLInfinityInequality} for all radial functions $f \in L^p_{\mu_1, \mu_2} \left( \mathbb{H}^n \right)$.
				\item[\mylabel{LPLInfinityAlphaLessThan2}{(B)}] For $1 \leq \alpha < 2$, we have Inequality \eqref{LPLInfinityInequality} for all radial functions $f \in L^p_{\mu_1, \mu_2} \left( \mathbb{H}^n \right)$ if and only if $p > \frac{2}{\alpha}$.
			\end{enumerate}
		\end{theorem}
		\begin{proof}
			We first observe that using Equation \eqref{FormualRadialFunctionsEquation}, we get for any $\xi \in \Xi_k \left( \mathbb{H}^n \right)$,
			\begin{equation}
				\label{RepresentationLInfinityNormRAlpha}
                \begin{aligned}
                    &\left| \sinh^{\nu_1} d \left( 0, \xi \right) \cosh^{\nu_2} d \left( 0, \xi \right) R_{\alpha}f \left( \xi \right) \right| \\
                    &\leq C \sinh^{\nu_1} d \left( 0, \xi \right) \cosh^{\nu_2 - \alpha + 1} d \left( 0, \xi \right) \\
                    &\int\limits_{\cosh d \left( 0, \xi \right)}^{\infty} \left| \tilde{f} \left( t \right) \right| t^{\alpha - 2} \left( 1 - \frac{\cosh^2 d \left( 0, \xi \right)}{t^2} \right)^{\frac{\alpha}{2} - 1} \mathrm{d}t.
                \end{aligned}
			\end{equation}
			\begin{enumerate}
				\item[\mylabel{LPLInfinityAlphaMoreThan2Proof}{(A)}] We now consider the case when $\alpha \geq 2$. We prove the result in two parts.
				
				\textbf{Case I:} First, let us assume that $p > 1$. By an application of H\"{o}lder's inequality, we get
				\begin{align*}
					&\left| \sinh^{\nu_1} d \left( 0, \xi \right) \cosh^{\nu_2} d \left( 0, \xi \right) R_{\alpha}f \left( \xi \right) \right| \\
                    &\leq C \| f \|_{L^p_{\mu_1, \mu_2} \left( \mathbb{H}^n \right)} \sinh^{\nu_1} d \left( 0, \xi \right) \cosh^{\nu_2 - \alpha + 1} d \left( 0, \xi \right) \times \\
					&\left[ \int\limits_{\cosh d \left( 0, \xi \right)}^{\infty} \frac{t^{p' \left( \alpha - \frac{\mu_2}{p} - 2 \right)}}{\left( t^2 - 1 \right)^{p' \left( \frac{\mu_1 + n}{2p} - \frac{1}{p} \right)}} \left( 1 - \frac{\cosh^2 d \left( 0, \xi \right)}{t^2} \right)^{p' \left( \frac{\alpha}{2} - 1 \right)} \mathrm{d}t \right]^{\frac{1}{p'}}.
				\end{align*}
				Let us consider
				$$I = \int\limits_{\cosh d \left( 0, \xi \right)}^{\infty} \frac{t^{p' \left( \alpha - \frac{\mu_2}{p} - 2 \right)}}{\left( t^2 - 1 \right)^{p' \left( \frac{\mu_1 + n}{2p} - \frac{1}{p} \right)}} \left( 1 - \frac{\cosh^2 d \left( 0, \xi \right)}{t^2} \right)^{p' \left( \frac{\alpha}{2} - 1 \right)} \mathrm{d}t.$$
				By substituting $t^2 = \frac{\cosh^2 d \left( 0, \xi \right)}{u}$ and simplifying, we get
				\begin{equation}
					\label{HypergeometricIntegralLPLInfinityProof}
					I = \sech^{p' \left( \frac{\mu_1 + \mu_2 + n}{p} - \alpha \right) + 1} d \left( 0, \xi \right) \int\limits_{0}^{1} \frac{u^{\frac{p'}{2} \left( \frac{\mu_1 + \mu_2 + n}{p} - \alpha \right) - \frac{1}{2}} \left( 1 - u \right)^{p' \left( \frac{\alpha}{2} - 1 \right)}}{\left( 1 - u \sech^2 d \left( 0, \xi \right) \right)^{p' \left( \frac{\mu_1 + n}{2p} - \frac{1}{p} \right)}} \mathrm{d}u.
				\end{equation}
				Using Equation 15.6.1 of \cite{Olver}, we get
				\begin{align*}
					I &= \sech^{p' \left( \frac{\mu_1 + \mu_2 + n}{p} - \alpha \right) + 1} d \left( 0, \xi \right) {}_2F_1 \left( a, b; c; \sech^2 d \left( 0, \xi \right) \right),
				\end{align*}
                where, we have $a = p' \left( \frac{\mu_1 + n}{2p} - \frac{1}{p} \right)$, $b = \frac{p'}{2} \left( \frac{\mu_1 + \mu_2 + n}{p} - \alpha \right) + \frac{1}{2}$ and $c = \frac{p'}{2} \left( \frac{\mu_1 + \mu_2 + n}{p} - 2 \right) + \frac{3}{2}$. Thus, we have,
				\begin{align*}
					&\left| \sinh^{\nu_1} d \left( 0, \xi \right) \cosh^{\nu_2} d \left( 0, \xi \right) R_{\alpha}f \left( \xi \right) \right| \\
                    &\leq C \| f \|_{L^p_{\mu_1, \mu_2} \left( \mathbb{H}^n \right)} \sinh^{\nu_1} d \left( 0, \xi \right) \cosh^{\nu_2 - \frac{\mu_1 + \mu_2 + n - 1}{p}} d \left( 0, \xi \right) \times \\
                    &\left[ {}_2F_1 \left( a, b; c; \sech^2 d \left( 0, \xi \right) \right) \right]^{\frac{1}{p}}.
				\end{align*}
				For large values of $d \left( 0, \xi \right)$, we know that the hypergeometric function is near $1$ and hence behaves as a constant. On the other hand, since both hyperbolic sine and cosine behave as exponentials, using Equation \eqref{LPLInfinityNecessaryEquation}, we see that
				$$\left| \sinh^{\nu_1} d \left( 0, \xi \right) \cosh^{\nu_2} d \left( 0, \xi \right) R_{\alpha}f \left( \xi \right) \right| \leq C \| f \|_{L^p_{\mu_1, \mu_2} \left( \mathbb{H}^n \right)},$$
				whenever $d \left( 0, \xi \right)$ is large. On the other hand, when $d \left( 0, \xi \right)$ is near $0$, the behaviour of the hypergeometric function depends on the values of $\mu_1$ and $p$. Also, in this case, the hyperbolic cosine behaves as a constant and gives no input to the estimates. We have the following cases.
				
				\textbf{Case I(A):} If $p > \frac{\mu_1 + n}{\alpha}$, we see from Equation 15.4.1 of \cite{Olver}, that the hypergeometric function is bounded. Then, Inequality \eqref{LPLInfinityInequality} is obtained by using the fact that $\nu_1 \geq 0$.
				
				\textbf{Case I(B):} If $p = \frac{\mu_1 + n}{\alpha}$, we see from Equation 15.4.2 of \cite{Olver} that the hypergeometric function behaves as $\left| \ln \tanh d \left( 0, \xi \right) \right|$. However, in this case, we have assumed further that $\nu_1 > 0$ and therefore we get Inequality \eqref{LPLInfinityInequality}.
				
				\textbf{Case I(C):} If $p < \frac{\mu_1 + n}{\alpha}$, from Equation 15.4.4 of \cite{Olver} we conclude that the hypergeometric function behaves as $\tanh^{p' \left( \alpha - \frac{\mu_1 + n}{p} \right)} \lambda$. Here, we use the fact that $\nu_1 \geq \frac{\mu_1 + n}{p} - \alpha$ to get the desired result.
				
				This completes the proof when $p > 1$.
				
				\textbf{Case II:} Now, we take up the case when $p = 1$. Here, we write Equation \eqref{RepresentationLInfinityNormRAlpha} as
				\begin{align*}
					&\left| \sinh^{\nu_1} d \left( 0, \xi \right) \cosh^{\nu_2} d \left( 0, \xi \right) R_{\alpha}f \left( \xi \right) \right| \\
                    &\leq C \sinh^{\nu_1} d \left( 0, \xi \right) \cosh^{\nu_2 - \alpha + 1} d \left( 0, \xi \right) \times \\
					&\int\limits_{\cosh d \left( 0, \xi \right)}^{\infty} \left| \tilde{f} \left( t \right) \right| \left( t^2 - 1 \right)^{\frac{\mu_1 + n}{2} - 1} t^{\mu_2} \frac{t^{\alpha - \mu_1 - \mu_2 - n}}{\left( 1 - \frac{1}{t^2} \right)^{\frac{\mu_1 + n}{2} - 1}} \mathrm{d}t. 
				\end{align*}
				Since $\mu_1 + \mu_2 > \alpha - n$, we get,
				\begin{align*}
					&\left| \sinh^{\nu_1} d \left( 0, \xi \right) \cosh^{\nu_2} d \left( 0, \xi \right) R_{\alpha}f \left( \xi \right) \right| \\
                    &\leq C \sinh^{\nu_1} d \left( 0, \xi \right) \cosh^{\nu_2 - \mu_1 - \mu_2 - n + 1} d \left( 0, \xi \right)  \times \\
					&\int\limits_{\cosh d \left( 0, \xi \right)}^{\infty} \left| \tilde{f} \left( t \right) \right| \left( t^2 - 1 \right)^{\frac{\mu_1 + n}{2} - 1} t^{\mu_2} \left( 1 - \frac{1}{t^2} \right)^{1 - \frac{\mu_1 + n}{2}} \mathrm{d}t.
				\end{align*}
				We again split the proof into two parts.
				
				\textbf{Case II(A):} First, let us consider $d \left( 0, \xi \right) \geq 1$. Then, we observe that $\tanh^2 1 \leq 1 - \frac{1}{t^2} \leq 1$, for $t > \cosh d \left( 0, \xi \right)$. Also, both hyperbolic sine and cosine behave as exponentials. Therefore, by using the fact that $\nu_1 + \nu_2 \leq \mu_1 + \mu_2 + n - 1$, as given in Equation \eqref{LPLInfinityNecessaryEquation}, we get
				\begin{align*}
					&\left| \sinh^{\nu_1} d \left( 0, \xi \right) \cosh^{\nu_2} d \left( 0, \xi \right) R_{\alpha}f \left( \xi \right) \right| \\
                    &\leq C \int\limits_{\cosh d \left( 0, \xi \right)}^{\infty} \left| \tilde{f} \left( t \right) \right| \left( t^2 - 1 \right)^{\frac{\mu_1 + n}{2} - 1} t^{\mu_2} \mathrm{d}t = C \| f \|_{L^p_{\mu_1, \mu_2} \left( \mathbb{H}^n \right)}.
				\end{align*}
				
				\textbf{Case II(B):} Next, let us consider $0 < d \left( 0, \xi \right) < 1$. In this case, we have,
				\begin{align*}
					&\left| \sinh^{\nu_1} d \left( 0, \xi \right) \cosh^{\nu_2} d \left( 0, \xi \right) R_{\alpha}f \left( \xi \right) \right| \\
                    &\leq C \sinh^{\nu_1} d \left( 0, \xi \right) \cosh^{\nu_2 - \alpha + 1} d \left( 0, \xi \right) \times \\
					&\int\limits_{\cosh d \left( 0, \xi \right)}^{\cosh 1} \left| \tilde{f} \left( t \right) \right| \left( t^2 - 1 \right)^{\frac{\mu_1 + n}{2} - 1} t^{\mu_2} \frac{t^{\alpha - \mu_1 - \mu_2 - n}}{\left( 1 - \frac{1}{t^2} \right)^{\frac{\mu_1 + n}{2} - 1}} \mathrm{d}t \\
					&+ C \sinh^{\nu_1} d \left( 0, \xi \right) \cosh^{\nu_2 - \alpha + 1} d \left( 0, \xi \right) \times \\
					&\int\limits_{\cosh 1}^{\infty} \left| \tilde{f} \left( t \right) \right| \left( t^2 - 1 \right)^{\frac{\mu_1 + n}{2} - 1} t^{\mu_2} \frac{t^{\alpha - \mu_1 - \mu_2 - n}}{\left( 1 - \frac{1}{t^2} \right)^{\frac{\mu_1 + n}{2} - 1}} \mathrm{d}t.
				\end{align*}
				In the second term above, we observe that since $d \left( 0, \xi \right)$ is small and $\nu_1 \geq 0$ due to Equation \eqref{LPLInfinityNecessaryEquation}, the term can be bounded above by $\| f \|_{L^p_{\mu_1, \mu_2} \left( \mathbb{H}^n \right)}$ by using the technique mentioned in Case II(A). We now show that the first term can also be estimated by $\| f \|_{L^p_{\mu_1, \mu_2} \left( \mathbb{H}^n \right)}$. We rewrite the integral as
                \begin{align*}
                    &\int\limits_{\cosh d \left( 0, \xi \right)}^{\infty} \left| \tilde{f} \left( t \right) \right| \left( t^2 - \cosh^2 d \left( 0, \xi \right) \right)^{\frac{\alpha}{2} - 1} \mathrm{d}t \\
                    &= \int\limits_{\cosh d \left( 0, \xi \right)}^{\infty} \left| \tilde{f} \left( t \right) \right| \left( t^2 - 1 \right)^{\frac{\alpha}{2} - 1} \left( 1 - \frac{\sinh^2 d \left( 0, \xi \right)}{t^2 - 1} \right)^{\frac{\alpha}{2} - 1} \mathrm{d}t.
                \end{align*}
				Since we have assumed that $\alpha \geq 2$, we see that $\left( 1 - \frac{\sinh^2 d \left( 0, \xi \right)}{t^2 - 1} \right)^{\frac{\alpha}{2} - 1} \leq 1$, whenever $\cosh d \left( 0, \xi \right) < t < + \infty$. Therefore, we get
                \begin{align*}
                    &\int\limits_{\cosh d \left( 0, \xi \right)}^{1} \left| \tilde{f} \left( t \right) \right| \left( t^2 - 1 \right)^{\frac{\mu_1 + n}{2} - 1} t^{\mu_2} \frac{t^{\alpha - \mu_1 - \mu_2 - n}}{\left( 1 - \frac{1}{t^2} \right)^{\frac{\mu_1 + n}{2} - 1}} \mathrm{d}t \\
                    &\leq C \int\limits_{\cosh d \left( 0, \xi \right)}^{\cosh 1} \left| \tilde{f} \left( t \right) \right| \left( t^2 - 1 \right)^{\frac{\mu_1 + n}{2} - 1} t^{\mu_2} t^{- \mu_2} \left( t^2 - 1 \right)^{\frac{\alpha}{2} - \frac{\mu_1 + n}{2}} \mathrm{d}t.
                \end{align*}
				Since $1 \leq \cosh d \left( 0, \xi \right) \leq t \leq \cosh 1$, we have $t^{-\mu_2} \leq \max \left\lbrace 1, \cosh^{- \mu_2} 1 \right\rbrace$. Now, we have the following estimates
				\begin{align*}
					&\sinh^{\nu_1} d \left( 0, \xi \right) \cosh^{\nu_2 - \alpha + 1} d \left( 0, \xi \right) \times \\
                    &\int\limits_{\cosh d \left( 0, \xi \right)}^{\cosh 1} \left| \tilde{f} \left( t \right) \right| \left( t^2 - 1 \right)^{\frac{\mu_1 + n}{2} - 1} t^{\mu_2} \frac{t^{\alpha - \mu_1 - \mu_2 - n}}{\left( 1 - \frac{1}{t^2} \right)^{\frac{\mu_1 + n}{2} - 1}} \mathrm{d}t \\
					&\leq C \| f \|_{L^1_{\mu_1, \mu_2} \left( \mathbb{H}^n \right)} \begin{cases}
	\sinh^{\nu_1 - \mu_1 - n + \alpha} d \left( 0, \xi \right), & \mu_1 + n < \alpha. \\
	\sinh^{\nu_1} d \left( 0, \xi \right), & \mu_1 + n \geq \alpha.
\end{cases}
				\end{align*}
				In either case, the power of hyperbolic sine can be bounded above by a constant owing to Equation \eqref{LPLInfinityNecessaryEquation}. This completes the proof of Part \ref{LPLInfinityAlphaMoreThan2Proof}!
				\item[\mylabel{LPLInfinityAlphaLessThan2Proof}{(B)}] The proof of Part \ref{LPLInfinityAlphaLessThan2Proof} is verbatim to that of Case I of Part \ref{LPLInfinityAlphaMoreThan2Proof}, where the fact that $p > \frac{2}{\alpha}$ is used to get the finiteness of the hypergeometric integral in Equation \eqref{HypergeometricIntegralLPLInfinityProof}. The necessity of the condition $p > \frac{2}{\alpha}$ is shown in Example \ref{CoutnerExampleLPLInfinityAnuulus} and Remark \ref{NonboundednessRemarkRLFI}.
			\end{enumerate}
		\end{proof}
		We now see that the condition $p > \frac{2}{\alpha}$ mentioned in Part \ref{LPLInfinityAlphaLessThan2} of Theorem \ref{LPLInfinitySufficient} is necessary. First, we see through an example that $p \geq \frac{2}{\alpha}$ is required.
		\begin{example}
			\label{CoutnerExampleLPLInfinityAnuulus}
			\normalfont
			Fix $a > 1$ and define for $b > a$, the function $f_b: \mathbb{H}^n \rightarrow \mathbb{C}$ as $f_b \left( x \right) \chi_{\left( a, b \right)} \left( \cosh d \left( 0, x \right) \right) \cosh d \left( 0, x \right)$. Then, using Equation \eqref{FormualRadialFunctionsEquation}, it is easy to see that
			\begin{equation}
				\label{WeightedAnnulusRAlphaComputation}
                \begin{aligned}
                    R_{\alpha}f \left( \xi \right) = &\frac{C}{\cosh^{\alpha - 1} d \left( 0, \xi \right)} \times \\
                    &\left[ \left( b^2 - \cosh^2 d \left( 0, \xi \right) \right)^{\frac{\alpha}{2}} - \left( \max \left\lbrace a^2, \cosh^2 d \left( 0, \xi \right) \right\rbrace - \cosh^2 d \left( 0, \xi \right) \right)^{\frac{\alpha}{2}} \right].
                \end{aligned}
			\end{equation}
			By choosing $\xi \in \Xi_k \left( \mathbb{H}^n \right)$ with $\cosh d \left( 0, \xi \right) = a$, we observe that
			$$\| \sinh^{\nu_1} d \left( 0, \cdot \right) \cosh^{\nu_2} d \left( 0, \cdot \right) R_{\alpha}f_b \|_{L^{\infty} \left( \Xi_k \left( \mathbb{H}^n \right) \right)} \geq C \left( a^2 - 1 \right)^{\nu_1} a^{\nu_2 - \alpha + 1} \left( b^2 - a^2 \right)^{\frac{\alpha}{2}}.$$
			Since for every $b > a$, we have $f_b \in L^p \left( \mathbb{H}^n \right)$ for any $p \geq 1$, Inequality \eqref{LPLInfinityInequality} implies that the following holds for all $b > a$.
			$$\left( b^2 - a^2 \right)^{\frac{p \alpha}{2}} \leq C \int\limits_{a}^{b} \left( t^2 - 1 \right)^{\frac{\mu_1 + n}{2} - 1} t^{\mu_2} \mathrm{d}t.$$
			However, using the L'H\^{o}spital rule, it is easy to see that the inequality fails when $p < \frac{2}{\alpha}$. That is, for Inequality \eqref{LPLInfinityInequality} to hold, we require $p \geq \frac{2}{\alpha}$.
		\end{example}
		\begin{remark}
			\label{NonboundednessRemarkRLFI}
			\normalfont
			To observe that Inequality \eqref{LPLInfinityInequality} does not hold for $p = \frac{2}{\alpha}$, we resort to the theory of fractional integrals. The idea of the proof lies in the work of Hardy and Littlewood (see p. 578 \cite{HardyLittlewoodFI}). Assuming Inequality \eqref{LPLInfinityInequality} for radial functions, one can deduce a $L^{\frac{2}{\alpha}}$-$L^{\infty}$ estimate for the Riemann-Liouville fractional integral $I^{\frac{\alpha}{2}}_-$. However, it is easily seen using the arguments of Example C.3 of \cite{SelfLPImproving} that such an inequality is not possible. Particularly, the function of Example C.3 in \cite{SelfLPImproving} can be used in this case as well by considering $1 < \gamma \leq 2 - \frac{\alpha}{2}$.
		\end{remark}
		Theorems \ref{LPLInfinityNecessary} and \ref{LPLInfinitySufficient} together with Example \ref{CoutnerExampleLPLInfinityAnuulus} and Remark \ref{NonboundednessRemarkRLFI} almost completely answer Question \ref{LPLInfinityQuestion}. Let us observe that in Theorem \ref{LPLInfinitySufficient}, we have added another assumption that $\nu_1 > 0$ when $p = \frac{\mu_1 + n}{\alpha}$, which may not be necessary. The techniques to tackle the case $\nu_1 = 0$ and $p = \frac{\mu_1 + n}{\alpha}$ elude us. Nonetheless, this case is handled for a smaller class of Lorentz functions ($L^{p, 1}_{\mu_1, \mu_2} \left( \mathbb{H}^n \right)$) in Section \ref{EndPointEstimateSection}.
	\section{End-Point Estimates for radial functions}
		\label{EndPointEstimateSection}
		The purpose of this section is twofold. First, we recall that in answering Question \ref{LPLInfinityQuestion} we encounter a ``gap" in the proof, and the case $\nu_1 = 0$ with $p = \frac{\mu_1 + n}{\alpha}$ is yet to be resolved. We now ask a weaker question for these values of $\nu_1$ and $p$.
		\begin{question}
			\label{EndPointQuestion1}
			Given $\alpha \geq 1$, $\mu_1 + \mu_2 > \alpha - n$ and $p = \frac{\mu_1 + n}{\alpha}$, what are the admissible values of $\nu \in \mathbb{R}$, so that the following inequality holds for all radial functions $f \in L^{p, 1}_{\mu_1, \mu_2} \left( \mathbb{H}^n \right)$?
			\begin{equation}
				\label{EndPointDesiredInequality1}
				\| \cosh^{\nu} d \left( 0, \cdot \right) R_{\alpha}f \left( \cdot \right) \|_{L^{\infty} \left( \Xi_k \left( \mathbb{H}^n \right) \right)} \leq C \| f \|_{L^{p, 1} \left( \mathbb{H}^n \right)}.
			\end{equation}
			Here, the constant $C > 0$ depends only on $\nu, \mu_1, \mu_2, \alpha$ and $p$, and is independent of the choice of $f$.
		\end{question}
		Next, we observe that the value $p = \frac{\mu_1 + \mu_2 + n - 1}{\alpha - 1}$ is an ``end-point" of the existence of the transform $R_{\alpha}$ on the weighted Lebesgue spaces. Hence, it is natural to ask the behaviour of the said transform at this end-point. However, it would not be wise to ask the question about mapping properties on the Lebesgue space but rather, we must restrict ourselves to a subspace of the Lebesgue space. We, therefore, ask the following question.
		\begin{question}
			\label{EndPointQuestion2}
			Given $\alpha > 1$, $\mu_1 + \mu_2 > \alpha - n$ and $p = \frac{\mu_1 + \mu_2 + n - 1}{\alpha - 1}$, what are the admissible values of $\nu_1, \nu_2 \in \mathbb{R}$ such that the following inequality holds for all radial functions in $f \in L^{p, 1}_{\mu_1, \mu_2} \left( \mathbb{H}^n \right)$?
			\begin{equation}
				\label{EndPointDesiredInequality2}
				\| \sinh^{\nu_1} d \left( 0, \cdot \right) \cosh^{\nu_2} d \left( 0, \cdot \right) R_{\alpha}f \left( \cdot \right) \|_{L^{\infty} \left( \Xi_k \left( \mathbb{H}^n \right) \right)} \leq C \| f \|_{L^{p, 1}_{\mu_1, \mu_2} \left( \mathbb{H}^n \right)}.
			\end{equation}
			Here, the constant $C > 0$ depends only on $\nu_1, \nu_2, \mu_1, \mu_2, \alpha$ and $p$, and is independent of the choice of $f$.
		\end{question}
		\begin{remark}
			\normalfont
			We remark that the case $\alpha = 1$ is excluded from the study since for $q < +\infty$, we have $L^{\infty, q} \left( X \right) = \left\lbrace 0 \right\rbrace$, for any measure space $X$.
		\end{remark}
		A question about end-point behaviour (similar to Question \ref{EndPointQuestion2}) about the Radon transform in the Euclidean space was answered by Oberlin and Stein in \cite{OberlinStein} (see p. 42). For the $k$-plane transform on the Euclidean space, this question was answered by Duoandikoetxea et al. in Lemma 5 of \cite{Duoandikoetxea} and was then studied in the hyperbolic case by Kumar and Ray in \cite{KumarRayEndPoint} (Theorem 3.8). Later, such results were obtained for the $k$-plane transform on weighted Lorentz spaces by the authors of \cite{SelfLPImproving} (see Section 5.1 of their article). Throughout this section, we shall assume that the functions under considerations are characteristic functions of measurable subsets of $\mathbb{H}^n$, since proving the desired inequalities for such functions give the inequalities for generic Lorentz functions (see Theorem 3.13 in \cite{SteinWeissFA}). Also, we shall use the fact that if $f = \chi_A$, for a measurable set $A$, we have,
		$$\| f \|_{L^{p, 1}_{\mu_1, \mu_2} \left( \mathbb{H}^n \right)} = \| f \|_{L^p_{\mu_1, \mu_2} \left( \mathbb{H}^n \right)} = \left( \int\limits_{A} \sinh^{\mu_1} d \left( 0, x \right) \cosh^{\mu_2} d \left( 0, x \right) \mathrm{d}x \right)^{\frac{1}{p}}.$$
		
		The following lemmata from \cite{KumarRayEndPoint} and \cite{SelfLPImproving} are crucial to the proofs of this section.
		\begin{lemma}[\cite{KumarRayEndPoint}]
			\label{MainLemmaEuclidean}
			Let $f$ be a characteristic function of a measurable subset of $\left( 0, \infty \right)$ with finite measure. Then, for $\gamma > 0$ and $p \geq 1$, we have,
			\begin{equation}
				\label{MainLemmaEuclideanEquation}
				\int\limits_{0}^{\infty} f \left( t \right) t^{\gamma - 1} \mathrm{d}t \leq C \left( \int\limits_{0}^{\infty} f \left( t \right) t^{p \gamma - 1} \mathrm{d}t \right)^{\frac{1}{p}},
			\end{equation}
			where, the constant $C$ only depends on $p$ and $\gamma$.
			
			Similarly, for characteristic functions of measurable subsets of $\mathbb{R}$ with finite measure, we have
			\begin{equation}
				\label{MainLemmaEuclideanEquation2}
				\int\limits_{\mathbb{R}} f \left( t \right) e^{\gamma t} \mathrm{d}t \leq C \left( \int\limits_{\mathbb{R}} f \left( t \right) e^{p\gamma t} \mathrm{d}t \right)^{\frac{1}{p}},
			\end{equation}
			where, $\gamma > 0$ and $C$ is a constant that depends only on $\gamma$ and $p$.
		\end{lemma}
		\begin{lemma}[\cite{SelfLPImproving}]
			\label{MainLemma}
			Let $f$ be a characteristic function of a measurable subset of $\left( 0, \infty \right)$ with finite measure. Then, for $\gamma > 0$, $\alpha \in \mathbb{R}$ and $p \geq \max \left\lbrace 1, 1 - \alpha \right\rbrace$, we have,
			\begin{equation}
				\label{MainLemmaEquation}
				\int\limits_{0}^{\infty} f \left( t \right) \sinh^{\gamma} t \ \mathrm{d}t \leq C \left( \int\limits_{0}^{\infty} f \left( t \right) \sinh^{p \gamma - \alpha} t \cosh^{\alpha}t \ \mathrm{d}t \right)^{\frac{1}{p}},
			\end{equation}
			where, the constant $C$ depends only on $p$, $\gamma$ and $\alpha$.
		\end{lemma}
		Apart from the two lemmata mentioned above, we also require the following estimates for characteristic functions of measurable subsets of $\left( 1, \infty \right)$. Again, in proving these estimates, we find that Lemma \ref{MainLemmaEuclidean} is crucial.
		\begin{lemma}
			\label{MainLemma2}
			Let $\gamma_1 > -1$ and $\gamma_2 \in \mathbb{R}$ be such that $\gamma_1 + \gamma_2 > 0$. Then, for $p \geq 1$ and for characteristic functions $f$ of measurable subsets of $\left( 1, \infty \right)$ with finite measure, we have,
			\begin{equation}
				\label{MainLemma2Equation1}
				\int\limits_{1}^{\infty} f \left( t \right) \left( t - 1 \right)^{\frac{\gamma_1 - 1}{2}} t^{\frac{\gamma_2 - 1}{2}} \mathrm{d}t \leq C \left( \int\limits_{1}^{\infty} f \left( t \right) \left( t - 1 \right)^{\frac{p \left( \gamma_1 + 1 \right)}{2} - 1} t^{\frac{p \gamma_2 - 1}{2}} \mathrm{d}t \right)^{\frac{1}{p}},
			\end{equation}
			where, the constant $C$ depends only on $p$, $\gamma_1$ and $\gamma_2$. Further, we also have,
			\begin{equation}
				\label{MainLemma2Equation2}
				\int\limits_{1}^{\infty} f \left( t \right) \left( t - 1 \right)^{\frac{\gamma_1 - 1}{2}} t^{\frac{\gamma_2 - 1}{2}} \mathrm{d}t \leq C \left( \int\limits_{1}^{\infty} f \left( t \right) \left( t - 1 \right)^{\frac{p \gamma_1 - 1}{2}} t^{\frac{p \gamma_2 - 1}{2}} \mathrm{d}t \right)^{\frac{1}{p}},
			\end{equation}
			where, the constant $C$ depends only on the choice of $p$, $\gamma_1$ and $\gamma_2$.
		\end{lemma}
		\begin{proof}
			With the given conditions on $\gamma_1$ and $\gamma_2$, let us consider the following, for characteristic functions of measurable subsets of $\left( 0, \infty \right)$.
			\begin{align*}
				\int\limits_{0}^{\infty} f \left( t \right) \sinh^{\gamma_1} t \cosh^{\gamma_2} t \ \mathrm{d}t &= \int\limits_{0}^{1} f \left( t \right) \sinh^{\gamma_1} t \cosh^{\gamma_2} t \ \mathrm{d}t \\
                &+ \int\limits_{1}^{\infty} f \left( t \right) \sinh^{\gamma_1} t \cosh^{\gamma_2} t \ \mathrm{d}t.
			\end{align*}
			Since for $0 \leq t \leq 1$, $\sinh t$ behaves as $t$ and $\cosh t$ behaves as a constant, and for $t > 1$, both $\sinh t$ and $\cosh t$ behave as $e^t$, we get
			\begin{align*}
				\int\limits_{0}^{\infty} f \left( t \right) \sinh^{\gamma_1} t \cosh^{\gamma_2} t \ \mathrm{d}t \leq C \left[ \int\limits_{0}^{1} f \left( t \right) t^{\gamma_1} \mathrm{d}t + \int\limits_{0}^{\infty} f \left( t \right) e^{\left( \gamma_1 + \gamma_2 \right) t} \mathrm{d}t \right].
			\end{align*}
			Now, using Equation \eqref{MainLemmaEuclideanEquation} of Lemma \ref{MainLemmaEuclidean} in the first integral and Equation \eqref{MainLemmaEuclideanEquation2} of Lemma \ref{MainLemmaEuclidean} in the second integral, we have,
			\begin{equation}
				\label{MainLemma2Intermediate}
                \begin{aligned}
                    \int\limits_{0}^{\infty} f \left( t \right) \sinh^{\gamma_1} t \cosh^{\gamma_2} t \ \mathrm{d}t \leq C &\left[ \left( \int\limits_{0}^{1} f \left( t \right) t^{p \gamma_1 + p - 1} \mathrm{d}t \right)^{\frac{1}{p}} \right. \\
                    &\left. + \left( \int\limits_{1}^{\infty} f \left( t \right) e^{p \left( \gamma_1 + \gamma_2 \right)t} \mathrm{d}t \right)^{\frac{1}{p}} \right].
                \end{aligned}
			\end{equation}
			Again using the fact that for $t \in \left( 0, 1 \right)$, we may replace $t$ by $\sinh t$ and a constant by $\cosh t$, and for $t > 1$, we may replace exponentials by hyperbolic sine and cosine, and observing that $1 \leq e^{\left( p - 1 \right) t}$, for any $t \geq 0$ and $p \geq 1$, we get
			\begin{align*}
				\int\limits_{0}^{\infty} f \left( t \right) \sinh^{\gamma_1} t \cosh^{\gamma_2} t \ \mathrm{d}t \leq C \left( \int\limits_{0}^{\infty} f \left( t \right) \sinh^{p\gamma_1 + p - 1} t \cosh^{\gamma_2} t \ \mathrm{d}t \right)^{\frac{1}{p}}.
			\end{align*}
			Equation \eqref{MainLemma2Equation1} is readily obtained by replacing $\cosh^2 t$ by $t$. To obtain Equation \eqref{MainLemma2Equation2}, we dominate $t^{p - 1}$ in the first integral of Equation \eqref{MainLemma2Intermediate} by $1$ and proceed and above.
		\end{proof}
		We now begin answering Question \ref{EndPointQuestion1}. We start by getting the necessary conditions for Inequality \eqref{EndPointDesiredInequality1}.
		\begin{theorem}
			\label{EndPoint1Necessary}
			Let $\alpha \geq 1$, $\mu_1 + \mu_2 > \alpha - n$ and $p = \frac{\mu_1 + n}{\alpha}$. Then, Inequality \eqref{EndPointDesiredInequality1} holds only if
			\begin{equation}
				\label{EndPoint1NecessaryEquation}
				\nu \leq \frac{\mu_1 + \mu_2 + n - 1}{p}.
			\end{equation}
		\end{theorem}
		\begin{proof}
			For a fixed $a > 1$, let us consider the function $f_a: \mathbb{H}^n \rightarrow \mathbb{C}$ given by $f_a \left( x \right) = \chi_{\left( a, 2a \right)} \left( \cosh d \left( 0, x \right) \right)$. Then, it is easy to observe (for instance, see Equation (5.22) of \cite{SelfLPImproving}) that
			$$\| f_a \|_{L^{p, 1}_{\mu_1, \mu_2} \left( \mathbb{H}^n \right)} \leq C a^{\frac{\mu_2 - 1}{p} + \alpha}.$$
			On the other hand, using Equation \eqref{FormualRadialFunctionsEquation}, we get
			$$R_{\alpha}f_a \left( \xi \right) = \frac{C}{\cosh^{\alpha - 1} d \left( 0, \xi \right)} \int\limits_{\cosh d \left( 0, \xi \right)}^{\infty} \chi_{\left( a, 2a \right)} \left( t \right) \left( t^2 - \cosh^2 d \left( 0, \xi \right) \right)^{\frac{\alpha}{2} - 1} \mathrm{d}t.$$
			By substituting $t = x \cosh d \left( 0, \xi \right)$ and simplifying, we get
			$$R_{\alpha}f_a \left( \xi \right) = C \int\limits_{1}^{\infty} \chi_{\left( a, 2a \right)} \left( x \cosh d \left( 0, \xi \right) \right) \left( x^2 - 1 \right)^{\frac{\alpha}{2} - 1} \mathrm{d}x.$$
			We now see, by considering $\xi \in \Xi_k \left( \mathbb{H}^n \right)$ with $\cosh d \left( 0, \xi \right) = a$, that
			$$\| \cosh^{\nu} d \left( 0, \cdot \right) R_{\alpha}f \left( \cdot \right) \|_{L^{\infty} \left( \Xi_k \left( \mathbb{H}^n \right) \right)} \geq C a^{\nu}.$$
			Now, we test Inequality \eqref{EndPointDesiredInequality1} against $f_a$ for varying $a > 1$, and obtain
			$$a^{\nu} \leq C a^{\frac{\mu_2 - 1}{p} + \alpha}.$$
			As $a \rightarrow \infty$, the above inequality forces
			$$\nu \leq \frac{\mu_2 - 1}{p} + \alpha = \frac{\mu_1 + \mu_2 + n - 1}{p},$$
			since $p = \frac{\mu_1 + n}{\alpha}$.
		\end{proof}
		We now give a proof of the sufficiency of the condition mentioned in Equation \eqref{EndPoint1NecessaryEquation}, thereby completely answering Question \ref{EndPointQuestion1}.
		\begin{theorem}
			\label{EndPoint1Sufficient}
			Let $\alpha \geq 1$, $\mu_1 + \mu_2 > \alpha - n$ and $p = \frac{\mu_1 + n}{\alpha}$. Then, Inequality \eqref{EndPointDesiredInequality1} holds for all radial functions $f \in L^{p, 1}_{\mu_1, \mu_2} \left( \mathbb{H}^n \right)$ if and only if $\nu$ satisfies the conditions of Equation \eqref{EndPoint1NecessaryEquation} and $p \left( \alpha - 1 \right) \geq 1$.
		\end{theorem}
		\begin{proof}
			First we observe that by replacing $t$ by $\cosh t$ in Equation \eqref{FormualRadialFunctionsEquation}, we get for radial functions on $\mathbb{H}^n$,
			\begin{align}
				\label{AlternateRepresentationsRAlpha1}
				R_{\alpha}f \left( \xi \right) &= \frac{C}{\cosh^{\alpha - 1} d \left( 0, \xi \right)} \int\limits_{d \left( 0, \xi \right)} \tilde{f} \left( \cosh t \right) \left( \cosh^2 t - \cosh^2 d \left( 0, \xi \right) \right)^{\frac{\alpha}{2} - 1} \sinh t \ \mathrm{d}t \\
				\label{AlternateRepresentationsRAlpha2}
				&= \frac{C}{\cosh^{\alpha - 1} d \left( 0, \xi \right)} \int\limits_{d \left( 0, \xi \right)}^{\infty} \tilde{f} \left( \cosh t \right) \left( 1 - \frac{\sinh^2 d \left( 0, \xi \right)}{\sinh^2 t} \right)^{\frac{\alpha}{2} - 1} \sinh^{\alpha - 1} t \ \mathrm{d}t.
			\end{align}
			We prove the result in two cases.
			
			\textbf{Case I:} First, we consider $\alpha \geq 2$, which is an easier case to deal. We notice that in this case, we have
			$$\left( 1 - \frac{\sinh^{2} d \left( 0, \xi \right)}{\sinh^2 t} \right)^{\frac{\alpha}{2} - 1} \leq 1,$$
			for any $t \geq d \left( 0, \xi \right)$. Therefore, using Equation \eqref{AlternateRepresentationsRAlpha2}, we get
			\begin{align}
				\label{EndPoint1Case1Intermediate}
				\left| \cosh^{\nu} d \left( 0, \xi \right) R_{\alpha}f \left( \xi \right) \right| &\leq C \cosh^{\nu - \alpha + 1} d \left( 0, \xi \right) \int\limits_{d \left( 0, \xi \right)}^{\infty} \left| \tilde{f} \left( \cosh t \right) \right| \sinh^{\alpha - 1} t \ \mathrm{d}t.
			\end{align}
			Let us further divide our analysis into two cases depending on the values of $d \left( 0, \xi \right)$.
			
			\textbf{Case I(A):} Let us consider $d \left( 0, \xi \right)$ to be ``large", say $d \left( 0, \xi \right) \geq 1$. We have that $\frac{\mu_1 + n}{\alpha} = p \geq 1$ and $\mu_1 + \mu_2 > \alpha - n$. These estimates together with the fact that $\alpha \geq 1$, give $p > 1 - \mu_2$. Therefore, by an application of Equation \eqref{MainLemmaEquation} of Lemma \ref{MainLemma}, we get
			\begin{align*}
				\left| \cosh^{\nu} d \left( 0, \xi \right) R_{\alpha}f \left( \xi \right) \right| &\leq C \cosh^{\nu - \alpha + 1} d \left( 0, \xi \right) \times \\
                &\left( \int\limits_{d \left( 0, \xi \right)}^{\infty} \tilde{f} \left( \cosh t \right) \sinh^{\mu_1 + n - 1} t \cosh^{\mu_2} t \sinh^{1 - p - \mu_2} t \ \mathrm{d}t \right)^{\frac{1}{p}}.
			\end{align*}
			Since we have $1 - p - \mu_2 < 0$, we observe that $\cosh^{1 - p - \mu_2} t \leq \cosh^{1 - p - \mu_2} d \left( 0, \xi \right)$, for every $t > d \left( 0, \xi \right)$. Hence, we get
			\begin{align*}
				\left| \cosh^{\nu} d \left( 0, \xi \right) R_{\alpha}f \left( \xi \right) \right| &\leq C \cosh^{\nu - \alpha - \frac{\mu_2 - 1}{p}} d \left( 0, \xi \right) \times \\
                &\left( \int\limits_{d \left( 0, \xi \right)}^{\infty} \tilde{f} \left( \cosh t \right) \sinh^{\mu_1 + n - 1} t \cosh^{\mu_2} t \tanh^{1 - p - \mu_2} t \ \mathrm{d}t \right)^{\frac{1}{p}}.
			\end{align*}
			Further, we notice from Equation \eqref{EndPoint1NecessaryEquation} that $\cosh^{\nu - \alpha - \frac{\mu_2 - 1}{p}} d \left( 0, \xi \right) \leq 1$, for every $\xi \in \Xi_k \left( \mathbb{H}^n \right)$. Also, since $d \left( 0, \xi \right) \geq 1$, we have $\tanh 1 \leq \tanh t \leq 1$, for any $t > d \left( 0, \xi \right)$. Therefore, we get,
			\begin{align*}
				\left| \cosh^{\nu} d \left( 0, \xi \right) R_{\alpha}f \left( \xi \right) \right| \leq C \| f \|_{L^{p, 1}_{\mu_1, \mu_2} \left( \mathbb{H}^n \right)}.
			\end{align*}
			
			\textbf{Case I(B):} Let us now consider $0 < d \left( 0, \xi \right) < 1$. Here, we observe that any term involving powers of $\cosh d \left( 0, \xi \right)$ behaves as a constant, and therefore, we may introduce or remove such terms without affecting the inequalities (upto constants). Consequently, to estimate the left-hand-side of Equation \eqref{EndPointDesiredInequality1}, it is sufficient to look at the integral on the right-hand-side of Equation \eqref{EndPoint1Case1Intermediate}. We split the integral into two parts as follows.
			\begin{align*}
				\int\limits_{d \left( 0, \xi \right)}^{\infty} \tilde{f} \left( \cosh t \right) \sinh^{\alpha - 1} t \ \mathrm{d}t &= \int\limits_{d \left( 0, \xi \right)}^{1} \tilde{f} \left( \cosh t \right) \sinh^{\alpha - 1} t \ \mathrm{d}t + \int\limits_{1}^{\infty} \tilde{f} \left( \cosh t \right) \sinh^{\alpha - 1} t \ \mathrm{d}t \\
				&\leq \int\limits_{d \left( 0, \xi \right)}^{1} \tilde{f} \left( \cosh t \right) \sinh^{\alpha - 1} t \ \mathrm{d}t + C \| f \|_{L^{p, 1}_{\mu_1, \mu_2} \left( \mathbb{H}^n \right)},
			\end{align*}
			due to Case I(A) and the fact that any power of $\cosh d \left( 0, \xi \right)$ can be replaced by a constant. Therefore, we only need to study the first integral appearing above. To do so, let us start by observing that when $0 < d \left( 0, \xi \right) < t < 1$, $\sinh t$ behaves as $t$. Therefore, we have,
			$$\int\limits_{d \left( 0, \xi \right)}^{1} \tilde{f} \left( \cosh t \right) \sinh^{\alpha - 1} t \ \mathrm{d}t \leq C \int\limits_{d \left( 0, \xi \right)}^{1} \tilde{f} \left( \cosh t \right) t^{\alpha - 1} \mathrm{d}t.$$
			Now, using Equation \eqref{MainLemmaEuclideanEquation} of Lemma \ref{MainLemmaEuclidean}, we get
			\begin{align*}
				\int\limits_{d \left( 0, \xi \right)}^{1} \tilde{f} \left( \cosh t \right) \sinh^{\alpha - 1} t \ \mathrm{d}t \leq C \left( \int\limits_{d \left( 0, \xi \right)}^{1} \tilde{f} \left( \cosh t \right) t^{\mu_1 + n - 1} \ \mathrm{d}t \right)^{\frac{1}{p}}.
			\end{align*}
			Again replacing $t$ by $\sinh t$ and using the fact that any power of $\cosh t$ behaves as a constant when $d \left( 0, \xi \right) < t < 1$, we get
			\begin{align*}
				\int\limits_{d \left( 0, \xi \right)}^{1} \tilde{f} \left( \cosh t \right) \sinh^{\alpha - 1} t \mathrm{d}t \leq C \| f \|_{L^{p, 1}_{\mu_1, \mu_2} \left( \mathbb{H}^n \right)}.
			\end{align*}
			This completes the proof of Case I.
			
			\textbf{Case II:} We now consider the case when $1 \leq \alpha < 2$. Here, we substitute $\cosh t = \cosh x \cosh d \left( 0, \xi \right)$ in Equation \eqref{AlternateRepresentationsRAlpha1} and simplify to get
			\begin{align*}
				\left| \cosh^{\nu} d \left( 0, \xi \right) R_{\alpha}f \left( \xi \right) \right| &\leq C \cosh^{\nu} d \left( 0, \xi \right) \int\limits_{0}^{\infty} \left| \tilde{f} \left( \cosh x \cosh d \left( 0, \xi \right) \right) \right| \sinh^{\alpha - 1} x \ \mathrm{d}x \\
				&= C \cosh^{\nu} d \left( 0, \xi \right) \left[ \int\limits_{0}^{1} \left| \tilde{f} \left( \cosh x \cosh d \left( 0, \xi \right) \right) \right| \sinh^{\alpha - 1} x \ \mathrm{d}x \right. \\
                &\left. + \int\limits_{1}^{\infty} \left| \tilde{f} \left( \cosh x \cosh d \left( 0, \xi \right) \right) \right| \sinh^{\alpha - 1} x \ \mathrm{d}x \right].
			\end{align*}
			Now, we use the fact that when $x \in \left( 0, 1 \right)$, $\sinh x$ behaves as $x$ and $\cosh x$ behaves as a constant, and therefore, we may apply Equation \eqref{MainLemmaEuclideanEquation} of Lemma \ref{MainLemmaEuclidean} on the first integral above. On the other hand, we have observed in Case I that $p > 1 - \mu_2$, so that we may apply Equation \eqref{MainLemmaEquation} of Lemma \ref{MainLemma} on the second integral above. Hence, we get after a simplification,
			\begin{align*}
				\left| \cosh^{\nu} d \left( 0, \xi \right) R_{\alpha}f \left( \xi \right) \right| &\leq C \cosh^{\nu} d \left( 0, \xi \right) \times \\
                &\left( \int\limits_{0}^{\infty} \left| \tilde{f} \left( \cosh x \cosh d \left( 0, \xi \right) \right) \right| \sinh^{\mu_1 + n - 1} t \cosh^{\mu_2} t \ \mathrm{d}x \right)^{\frac{1}{p}}.
			\end{align*}
			Resubstituting $\cosh x = \frac{\cosh t}{\cosh d \left( 0, \xi \right)}$ and simplifying, we get
			\begin{align*}
				&\left| \cosh^{\nu} d \left( 0, \xi \right) R_{\alpha}f \left( \xi \right) \right| \leq C \cosh^{\nu - \frac{\mu_1 + \mu_2 + n - 1}{p}} d \left( 0, \xi \right) \times \\
                &\left( \int\limits_{d \left( 0, \xi \right)}^{\infty} \left| \tilde{f} \left( \cosh t \right) \right| \left( \cosh^2 t - \cosh^2 d \left( 0, \xi \right) \right)^{\frac{\mu_1 + n}{2} - 1} \cosh^{\mu_2} t \sinh t \ \mathrm{d}t \right)^{\frac{1}{p}}.
			\end{align*}
			Finally, using Equation \eqref{EndPoint1NecessaryEquation} and the fact that $p \left( \alpha - 1 \right) \geq 1$ and $\alpha < 2$, we get $p\alpha = \mu_1 + n \geq 2$, and hence,
			\begin{align*}
				\left| \cosh^{\nu} d \left( 0, \xi \right) R_{\alpha}f \left( \xi \right) \right| &\leq C \| f \|_{L^{p, 1}_{\mu_1, \mu_2} \left( \mathbb{H}^n \right)}.
			\end{align*}
			This shows that the conditions mentioned in the statement of the Theorem are sufficient. The necessity of Equation \eqref{EndPoint1NecessaryEquation} was proved in Theorem \ref{EndPoint1Necessary}. On the other hand, we notice that if $\alpha \geq 2$, then the condition $p \left( \alpha - 1 \right) \geq 1$ is redundant. The necessity of this condition when $1 \leq \alpha < 2$ is made clear in Example \ref{CounterExampleEndPoint1}.
		\end{proof}
		We now see that for a Lorentz space estimate of the operator $R_{\alpha}$, we require $p \geq \frac{1}{\alpha - 1}$. As mentioned in the proof of Theorem \ref{EndPoint1Sufficient}, we shall focus on $1 \leq \alpha < 2$.
		\begin{example}
			\label{CounterExampleEndPoint1}
			\normalfont
			Let us fix $1 \leq \alpha < 2$, $\mu_1 + \mu_2 > \alpha - n$ and $p \geq 1$. For a fixed $a > 1$, and $b > a$, let us consider the function $f_b: \mathbb{H}^n \rightarrow \mathbb{C}$, given by $f_b \left( x \right) = \chi_{\left( a, b \right)} \left( \cosh d \left( 0, x \right) \right)$. Clearly, $f_b$ is the characteristic function of an annulus in $\mathbb{H}^n$ with inner radius $\cosh^{-1} \left( a \right)$ and outer radius $\cosh^{-1} \left( b \right)$. Therefore, we have,
			$$\| f_b \|_{L^{p, 1}_{\mu_1, \mu_2} \left( \mathbb{H}^n \right)} = C \left( \int\limits_{a}^{b} \left( t^2 - 1 \right)^{\frac{\mu_1 + n}{2} - 1} t^{\mu_2} \mathrm{d}t \right)^{\frac{1}{p}}.$$
			On the other hand, using Equation \eqref{FormualRadialFunctionsEquation}, we get
			\begin{align*}
				R_{\alpha}f \left( \xi \right) &= C \cosh^{1 - \alpha} d \left( 0, \xi \right) \int\limits_{\max \left\lbrace \cosh d \left( 0, \xi \right), a \right\rbrace}^{b} \left( t^2 - \cosh^2 d \left( 0, \xi \right) \right)^{\frac{\alpha}{2} - 1} \mathrm{d}t.
			\end{align*}
			Since $\alpha < 2$, we have $\left( t^2 - \cosh^2 d \left( 0, \xi \right) \right)^{\frac{\alpha}{2} - 1} \geq t^{\alpha - 2}$. Thus, choosing $\xi \in \Xi_k \left( \mathbb{H}^n \right)$ such that $\cosh d \left( 0, \xi \right) = a$, we notice that
			\begin{align*}
				&\| \sinh^{\nu_1} d \left( 0, \cdot \right) \cosh^{\nu_2} d \left( 0, \cdot \right) R_{\alpha}f \left( \cdot \right) \|_{L^{\infty} \left( \Xi_k \left( \mathbb{H}^n \right) \right)} \\
                &\geq C \left( a^2 - 1 \right)^{\nu_1} a^{\nu_2 - \alpha + 1} \left( b^{\alpha - 1} - a^{\alpha - 1} \right).
			\end{align*}
			Now, we choose $b < 2a$ so that $\left( b - a \right)^{\alpha - 1} \leq b^{\alpha - 1} - a^{\alpha - 1}$. That is,
			\begin{align*}
				\| \sinh^{\nu_1} d \left( 0, \cdot \right) \cosh^{\nu_2} d \left( 0, \cdot \right) R_{\alpha}f \left( \cdot \right) \|_{L^{\infty} \left( \Xi_k \left( \mathbb{H}^n \right) \right)} &\geq C \left( a^2 - 1 \right)^{\nu_1} a^{\nu_2 - \alpha + 1} \left( b - a \right)^{\alpha - 1}.
			\end{align*}
			We now test whether we have weighted $L^{p, 1} \left( \mathbb{H}^n \right)$-$L^{\infty} \left( \Xi_k \left( \mathbb{H}^n \right) \right)$ boundedness of the operator $R_{\alpha}$ for the functions $f_b$, where $a < b < 2a$. If such an operator boundedness holds, we would have, for every $a < b < 2a$,
			$$\left( b - a \right)^{p \left( \alpha - 1 \right)} \leq C_a \int\limits_{a}^{b} \left( t^2 - 1 \right)^{\frac{\mu_1 + n}{2} - 1} t^{\mu_2} \mathrm{d}t.$$
			It is easily seen by a simple application of L'H\^{o}spital Rule that this is not possible unless $p \left( \alpha - 1 \right) \geq 1$.
		\end{example}
		Theorems \ref{EndPoint1Necessary} and \ref{EndPoint1Sufficient} together with Example \ref{CounterExampleEndPoint1} completely answer Question \ref{EndPointQuestion1}. We now move on to answer Question \ref{EndPointQuestion2}. As before, we start by obtaining necessary conditions for Inequality \eqref{EndPointDesiredInequality2} to hold.
		\begin{theorem}
			\label{EndPoint2Necessary}
			Let $\alpha > 1$, $\mu_1 + \mu_2 > \alpha - n$ and $p = \frac{\mu_1 + \mu_2 + n - 1}{\alpha - 1}$. Then, Inequality \eqref{EndPointDesiredInequality2} holds only if
			\begin{equation}
				\label{EndPoint2NecessaryEquation}
				\nu_1 \geq \max \left\lbrace 0, \frac{\mu_1 + n}{p} - \alpha \right\rbrace \text{ and } \nu_1 + \nu_2 \leq \alpha - 1.
			\end{equation}
		\end{theorem}
		\begin{proof}
			For $\lambda > 0$, let us consider the function $f_{\lambda}: \mathbb{H}^n \rightarrow \mathbb{C}$ defined as $f_{\lambda} \left( x \right) = \chi_{B \left( 0, \lambda \right)} \left( x \right)$. Then, it is easy to see that
			\begin{equation}
				\label{LorentzNormBallEstimate}
                \begin{aligned}
                    \| f_{\lambda} \|_{L^{p, 1}_{\mu_1, \mu_2} \left( \mathbb{H}^n \right)} &= C \sinh^{\frac{\mu_1 + n}{p}} \lambda \cosh^{\frac{\mu_2 - 1}{p}} \lambda \left[ {}_2F_1 \left( \frac{1 - \mu_2}{2}, 1; 1 + \frac{\mu_1 + n}{2}; \tanh^2 \lambda \right) \right]^{\frac{1}{p}} \\
                    &\leq C \sinh^{\frac{\mu_1 + n}{p}} \lambda \cosh^{\frac{\mu_2 - 1}{p}} \lambda.
                \end{aligned}
			\end{equation}
			On the other hand, by substituting $t^2 = \cosh^2 d \left( 0, \xi \right) + \left( \cosh^2 \lambda - \cosh^2 d \left( 0, \xi \right) \right) u$, using Equation 15.6.1 of \cite{Olver} and simplifying, we get
			\begin{equation}
				\label{RAlphaBallComputationLorentz}
                \begin{aligned}
                    R_{\alpha}f_{\lambda} \left( \xi \right) = &C \tanh^{\alpha} \lambda \left( 1 - \frac{\sinh^2 d \left( 0, \xi \right)}{\sinh^2 \lambda} \right)^{\frac{\alpha}{2}} \times \\
                    &{}_2F_1 \left( \frac{\alpha + 1}{2}, \frac{\alpha}{2}; 1 + \frac{\alpha}{2}; \tanh^2 \lambda \left( 1 - \frac{\sinh^2 d \left( 0, \xi \right)}{\sinh^2 \lambda} \right) \right).
                \end{aligned}
			\end{equation}
			Therefore, by choosing $\xi \in \Xi_k \left( \mathbb{H}^n \right)$ such that $\sinh^2 d \left( 0, \xi \right) = \frac{1}{2} \sinh^2 \lambda$ in Equation \eqref{RAlphaBallComputationLorentz}, we get
			\begin{equation}
				\label{LInfinityNormRAlphaBallEstimate}
                \begin{aligned}
                    &\| \sinh^{\nu_1} d \left( 0, \cdot \right) \cosh^{\nu_2} d \left( 0, \cdot \right) R_{\alpha}f \left( \cdot \right) \|_{L^{\infty} \left( \Xi_k \left( \mathbb{H}^n \right) \right)} \\
                    &\geq C \sinh^{\nu_1} \lambda \left( 1 + \cosh^2 \lambda \right)^{\frac{\nu_2}{2}} \tanh^{\alpha} \lambda \ {}_2F_1 \left( \frac{\alpha + 1}{2}, \frac{\alpha}{2}; 1 + \frac{\alpha}{2}; \frac{\tanh^2 \lambda}{2} \right).
                \end{aligned}
			\end{equation}
			Also, from Equation \eqref{RAlphaBallComputationLorentz}, it is clear that $\nu_1 \geq 0$ is necessary for otherwise $\| \sinh^{\nu_1} d \left( 0, \cdot \right) \cosh^{\nu_2} d \left( 0, \cdot \right) R_{\alpha}f_{\lambda} \left( \cdot \right) \|_{L^{\infty} \left( \Xi_k \left( \mathbb{H}^n \right) \right)} = + \infty$.
			
			Now we use Inequalities \eqref{LInfinityNormRAlphaBallEstimate} and \eqref{LorentzNormBallEstimate} in Inequality \eqref{EndPointDesiredInequality2} and take $\lambda \rightarrow 0$ and $\lambda \rightarrow \infty$ to get the conditions mentioned in Equation \eqref{EndPoint2NecessaryEquation}.
		\end{proof}
		We now show the sufficiency of the conditions mentioned in Equation \eqref{EndPoint2NecessaryEquation}.
		\begin{theorem}
			\label{EndPoint2Sufficient}
			Let $\alpha > 1$, $\mu_1 + \mu_2 > \alpha - n$ and $p = \frac{\mu_1 + \mu_2 + n - 1}{\alpha - 1}$. Then, Inequality \eqref{EndPointDesiredInequality2} holds for all radial functions $f \in L^{p, 1}_{\mu_1, \mu_2} \left( \mathbb{H}^n \right)$ if and only if $\nu_1$ and $\nu_2$ satisfy the conditions of Equation \eqref{EndPoint2NecessaryEquation} and $p \left( \alpha - 1 \right) \geq 1$.
		\end{theorem}
		\begin{proof}
			The necessary part of the theorem was proved in Theorem \ref{EndPoint2Necessary} and Example \ref{CounterExampleEndPoint1}. We now show the sufficient part. To do so, we split our analysis into three cases.
			
			\textbf{Case I:} We first consider $\alpha \geq 2$ and $p \geq 1 - \mu_2$. We notice that since $\alpha \geq 2$, we have $\left( 1 - \frac{\sinh^2 d \left( 0, \xi \right)}{\sinh^2 t} \right)^{\frac{\alpha}{2} - 1} \leq 1$, for $t > d \left( 0, \xi \right)$. Therefore, using Equation \eqref{AlternateRepresentationsRAlpha2}, we get
			\begin{equation}
				\label{EndPoint2Case1Intermediate}
                \begin{aligned}
                    &\left| \sinh^{\nu_1} d \left( 0, \xi \right) \cosh^{\nu_2} d \left( 0, \xi \right) R_{\alpha}f \left( \xi \right) \right| \\
                    &\leq C \sinh^{\nu_1} d \left( 0, \xi \right) \cosh^{\nu_2 - \alpha + 1} d \left( 0, \xi \right) \int\limits_{d \left( 0, \xi \right)} \tilde{f} \left( \cosh t \right) \sinh^{\alpha - 1} t \ \mathrm{d}t.
                \end{aligned}
			\end{equation}
			Now, using Equation \eqref{MainLemmaEquation} of Lemma \ref{MainLemma}, Equation \eqref{EndPoint2NecessaryEquation} and $p \left( \alpha - 1 \right) = \mu_1 + \mu_2 + n - 1$, we get
			\begin{align*}
				\left| \sinh^{\nu_1} d \left( 0, \xi \right) \cosh^{\nu_2} d \left( 0, \xi \right) R_{\alpha}f \left( \xi \right) \right| &\leq C \| f \|_{L^{p, 1}_{\mu_1, \mu_2} \left( \mathbb{H}^n \right)}.
			\end{align*}
			
			\textbf{Case II:} We now consider $\alpha \geq 2$ and $p < 1 - \mu_2$. Here, we cannot use Lemma \ref{MainLemma}. To get around this complication, we split our proof into two parts.
			
			\textbf{Case II(A):} First, let us consider $d \left( 0, \xi \right) \geq 1$. In this case, we have that both hyperbolic sine and cosine (of $d \left( 0, \xi \right)$ and $t > d \left( 0, \xi \right)$) behave as exponentials. Therefore, we may write, using Equation \eqref{EndPoint2Case1Intermediate},
			\begin{align*}
				&\left| \sinh^{\nu_1} d \left( 0, \xi \right) \cosh^{\nu_2} d \left( 0, \xi \right) R_{\alpha}f \left( \xi \right) \right| \\
                &\leq C e^{\left( \nu_1 + \nu_2 - \alpha + 1 \right) d \left( 0, \xi \right)} \int\limits_{d \left( 0, \xi \right)}^{\infty} \tilde{f} \left( \cosh t \right) e^{\left( \alpha - 1 \right) t} \mathrm{d}t.
			\end{align*}
			Now, using Equation \eqref{MainLemmaEuclideanEquation2} of Lemma \ref{MainLemmaEuclidean}, Equation \eqref{EndPoint2NecessaryEquation} and the fact that both hyperbolic sine and cosine behave as exponentials, we get
			\begin{align*}
				&\left| \sinh^{\nu_1} d \left( 0, \xi \right) \cosh^{\nu_2} d \left( 0, \xi \right) R_{\alpha}f \left( \xi \right) \right| \\
                &\leq C \left( \int\limits_{\mathbb{R}} \tilde{f} \left( \cosh t \right) \sinh^{\mu_1 + n - 1} t \cosh^{\mu_2} t \tanh^{\mu_2} t \ \mathrm{d}t \right)^{\frac{1}{p}}.
			\end{align*}
			Now, we see that since $t > d \left( 0, \xi \right) > 1$ and $\mu_2 < 1 - p < 0$, $\tanh^{\mu_2} t \leq \tanh^{\mu_2} 1$. Hence, we get
			\begin{align*}
				\left| \sinh^{\nu_1} d \left( 0, \xi \right) \cosh^{\nu_2} d \left( 0, \xi \right) R_{\alpha}f \left( \xi \right) \right| \leq C \| f \|_{L^{p, 1}_{\mu_1, \mu_2} \left( \mathbb{H}^n \right)}.
			\end{align*}
			
			\textbf{Case II(B):} Next, let us consider $0 < d \left( 0, \xi \right) < 1$. In this case, we may ignore the terms involving powers of hyperbolic cosine of $d \left( 0, \xi \right)$ since they behave as constants. As in the proof of Theorem \ref{EndPoint1Sufficient}, we split the integral over $\left( d  \left( 0, \xi \right), \infty \right)$ into two integrals: one over $\left( d \left( 0, \xi \right), 1 \right)$ and another over $\left( 1, \infty \right)$. Once this is achieved, it is enough to estimate the integral over $\left( d \left( 0, \xi \right), 1 \right)$ since the other term can be bounded (upto a constant) by $\| f \|_{L^{p, 1}_{\mu_1, \mu_2} \left( \mathbb{H}^n \right)}$, using the techniques of Case II(A) and the fact that $\nu_1 \geq 0$.
			
			We, therefore, consider the following.
			\begin{align*}
				\int\limits_{d \left( 0, \xi \right)}^{1} \tilde{f} \left( \cosh t \right) \sinh^{\alpha - 1} t \ \mathrm{d}t \leq C \int\limits_{d \left( 0, \xi \right)}^{1} \tilde{f} \left( \cosh t \right) t^{\alpha - 1} \mathrm{d}t,
			\end{align*}
			since for small values of $t$, $\sinh t$ behaves as $t$. Then, using Equation \eqref{MainLemmaEuclideanEquation} of Lemma \ref{MainLemmaEuclidean} and the fact that $\sinh t$ behaves as $t$, we get
			\begin{align*}
				\int\limits_{d \left( 0, \xi \right)}^{1} \tilde{f} \left( \cosh t \right) \sinh^{\alpha - 1} t \ \mathrm{d}t \leq C \left( \int\limits_{d \left( 0, \xi \right)}^{1} \tilde{f} \left( \cosh t \right) \sinh^{\mu_1 + n - 1} t \sinh^{\mu_2 + p - 1} t \ \mathrm{d}t \right)^{\frac{1}{p}}.
			\end{align*}
			Since $p < 1 - \mu_2$, we get
			\begin{align*}
				\int\limits_{d \left( 0, \xi \right)}^{1} \tilde{f} \left( \cosh t \right) \sinh^{\alpha - 1} t \ \mathrm{d}t \leq C \sinh^{\frac{\mu_2 - 1}{p} + 1} d \left( 0, \xi \right) \| f \|_{L^{p, 1}_{\mu_1, \mu_2} \left( \mathbb{H}^n \right)}.
			\end{align*}
			That is, we have,
			\begin{align*}
				&\left| \sinh^{\nu_1} d \left( 0, \xi \right) \cosh^{\nu_2} d \left( 0, \xi \right) R_{\alpha}f \left( \xi \right) \right| \\
                &\leq C \left( \sinh^{\nu_1 + \frac{\mu_2 - 1}{p} - 1} d \left( 0, \xi \right) + 1 \right) \| f \|_{L^{p, 1}_{\mu_1, \mu_2} \left( \mathbb{H}^n \right)}.
			\end{align*}
			Due to Equation \eqref{EndPoint2NecessaryEquation}, we have $\nu_1 \geq \frac{\mu_1 + n}{p} - \alpha = \frac{1 - \mu_2}{p} - 1$. This gives us
			\begin{align*}
				\left| \sinh^{\nu_1} d \left( 0, \xi \right) \cosh^{\nu_2} d \left( 0, \xi \right) R_{\alpha}f \left( \xi \right) \right| &\leq C \| f \|_{L^{p, 1}_{\mu_1, \mu_2} \left( \mathbb{H}^n \right)}.
			\end{align*}
			
			\textbf{Case III:} Lastly, we take up the case when $1 < \alpha < 2$. This case is a bit tricky since we cannot eliminate the term $\left( 1 - \frac{\sinh^2 d \left( 0, \xi \right)}{\sinh^2 t} \right)^{\frac{\alpha}{2} - 1}$ as in the other two cases. To overcome this obstacle, we shall take another approach to studying the formula for the transform $R_{\alpha}$ (for radial functions). By substituting $t^2 = x \cosh^2 d \left( 0, \xi \right)$ in Equation \eqref{FormualRadialFunctionsEquation}, we get
			\begin{equation}
				\label{EndPoint2Case3Intermediate}
                \begin{aligned}
                    \left| \sinh^{\nu_1} d \left( 0, \xi \right) \cosh^{\nu_2} d \left( 0, \xi \right) R_{\alpha}f \left( \xi \right) \right| \leq &C \sinh^{\nu_1} d \left( 0, \xi \right) \cosh^{\nu_2} d \left( 0, \xi \right) \times \\
                    &\int\limits_{1}^{\infty} \tilde{f} \left( \sqrt{x} \cosh d \left( 0, \xi \right) \right) \left( x - 1 \right)^{\frac{\alpha}{2} - 1} x^{- \frac{1}{2}} \mathrm{d}x.
                \end{aligned}
			\end{equation}
			We now split our analysis in the following cases.
			
			\textbf{Case III(A):} First, let us consider $\mu_2 \geq 0$. Here, using Equation \eqref{MainLemma2Equation2} of Lemma \ref{MainLemma2} in Equation \eqref{EndPoint2Case3Intermediate}, we get
			\begin{align*}
				&\left| \sinh^{\nu_1} d \left( 0, \xi \right) \cosh^{\nu_2} d \left( 0, \xi \right) R_{\alpha}f \left( \xi \right) \right| \\
                &\leq C \sinh^{\nu_1} d \left( 0, \xi \right) \cosh^{\nu_2} d \left( 0, \xi \right) \times \\
                &\left( \int\limits_{1}^{\infty} \tilde{f} \left( \sqrt{x} \cosh d \left( 0, \right) \right) \left( x - 1 \right)^{\frac{p \left( \alpha - 1 \right) - 1}{2}} x^{- \frac{1}{2}} \mathrm{d}x \right)^{\frac{1}{p}}.
			\end{align*}
			We now substitute $x \cosh^2 d \left( 0, \xi \right) = \cosh^2 t$ and simplify to obtain
			\begin{align*}
				&\left| \sinh^{\nu_1} d \left( 0, \xi \right) \cosh^{\nu_2} d \left( 0, \xi \right) R_{\alpha}f \left( \xi \right) \right| \\
				&\leq C \sinh^{\nu_1} d \left( 0, \xi \right) \cosh^{\nu_2 - \alpha + 1} d \left( 0, \xi \right) \times \\
                &\left( \int\limits_{d \left( 0, \xi \right)}^{\infty} \tilde{f} \left( \cosh t \right) \left( \cosh^2 t - \cosh^2 d \left( 0, \xi \right) \right)^{\frac{p \left( \alpha - 1 \right) - 1}{2}} \sinh t \ \mathrm{d}t \right)^{\frac{1}{p}}.
			\end{align*}
			Since we have $p \left( \alpha - 1 \right) \geq 1$ and $\cosh^2 d \left( 0, \xi \right) \geq 1$, we get
			$$\left( \cosh^2 t - \cosh^2  \left( 0, \xi \right) \right)^{\frac{p \left( \alpha - 1 \right) - 1}{2}} \sinh t \leq \sinh^{p \left( \alpha - 1 \right)} t = \sinh^{\mu_1 + \mu_2 + n - 1} t$$.
			That is,
			\begin{align*}
				&\left| \sinh^{\nu_1} d \left( 0, \xi \right) \cosh^{\nu_2} d \left( 0, \xi \right) R_{\alpha}f \left( \xi \right) \right| \\
				&\leq C \sinh^{\nu_1} d \left( 0, \xi \right) \cosh^{\nu_2 - \alpha + 1} d \left( 0, \xi \right) \times \\
                &\left( \int\limits_{d \left( 0, \xi \right)}^{\infty} \tilde{f} \left( \cosh t \right) \sinh^{\mu_1 + n - 1} t \cosh^{\mu_2} t \tanh^{\mu_2} t \ \mathrm{d}t \right)^{\frac{1}{p}}.
			\end{align*}
			Since we have assumed that $\mu_2 \geq 0$, we get
			\begin{align*}
				&\left| \sinh^{\nu_1} d \left( 0, \xi \right) \cosh^{\nu_2} d \left( 0, \xi \right) R_{\alpha}f \left( \xi \right) \right| \\
                &\leq C \sinh^{\nu_1} d \left( 0, \xi \right) \cosh^{\nu_2 - \alpha + 1} d \left( 0, \xi \right) \| f \|_{L^{p, 1}_{\mu_1, \mu_2} \left( \mathbb{H}^n \right)}.
			\end{align*}
			Now, for large values of $d \left( 0, \xi \right)$, we have owing to Equation \eqref{EndPoint2NecessaryEquation} and the fact that both hyperbolic sine and cosine behave as exponentials,
			$$\sinh^{\nu_1} d \left( 0, \xi \right) \cosh^{\nu_2 - \alpha + 1} d \left( 0, \xi \right) \leq C e^{\left( \nu_1 + \nu_2 - \alpha + 1 \right) t} \leq C.$$
			On the other hand, for small values of $d \left( 0, \xi \right)$, $\cosh d \left( 0, \xi \right)$ behaves as a constant, and $\sinh^{\nu_1} d \left( 0, \xi \right) \leq C$, since $\nu_1 \geq 0$. This gives Inequality \eqref{EndPointDesiredInequality2} when $\mu_2 \geq 0$.
			
			\textbf{Case III(B):} We now consider $\mu_2 < 0$. We first take $\xi \in \Xi_k \left( \mathbb{H}^n \right)$, for which $d \left( 0, \xi \right) \geq 1$. Then, proceeding verbatim to Case III(A) above, we get
			\begin{align*}
				&\left| \sinh^{\nu_1} d \left( 0, \xi \right) \cosh^{\nu_2} d \left( 0, \xi \right) R_{\alpha}f \left( \xi \right) \right| \\
                &\leq C e^{\left( \nu_1 + \nu_2 - \alpha + 1 \right) d \left( 0, \xi \right)} \times \\
                &\left( \int\limits_{d \left( 0, \xi \right)}^{\infty} \tilde{f} \left( \cosh t \right) \sinh^{\mu_1 + n - 1} t \ \cosh^{\mu_2} t \tanh^{\mu_2} t \  \mathrm{d}t \right)^{\frac{1}{p}}.
			\end{align*}
			Again, due to Equation \eqref{EndPoint2NecessaryEquation}, we have $e^{\left( \nu_1 + \nu_2 - \alpha + 1 \right) d \left( 0, \xi \right)} \leq 1$. On the other hand, since $1 \leq d \left( 0, \xi \right) < t < \infty$, we get $\tanh 1 \leq \tanh t \leq 1$. Therefore, $\tanh^{\mu_2} t$ can also be dominated by a constant and we get,
			\begin{align*}
				\left| \sinh^{\nu_1} d \left( 0, \xi \right) \cosh^{\nu_2} d \left( 0, \xi \right) R_{\alpha}f \left( \xi \right) \right| \leq C \| f \|_{L^{p, 1}_{\mu_1, \mu_2} \left( \mathbb{H}^n \right)}.
			\end{align*}
			Now, let us consider $\xi \in \Xi_k \left( \mathbb{H}^n \right)$ so that $0 < d \left( 0, \xi \right) < 1$. Here, we comment that terms involving $\cosh d \left( 0, \xi \right)$ may be replaced by a constant. Now, we split the integral into two parts as follows.
			\begin{align*}
				&\left| \sinh^{\nu_1} d \left( 0, \xi \right) \cosh^{\nu_2} d \left( 0, \xi \right) R_{\alpha}f \left( \xi \right) \right| \\
				&\leq C \sinh^{\nu_1} d \left( 0, \xi \right) \left[ \int\limits_{\cosh d \left( 0, \xi \right)}^{\cosh 1} \tilde{f} \left( t \right) \left( t^2 - \cosh^2 d \left( 0, \xi \right) \right)^{\frac{\alpha}{2} - 1} \mathrm{d}t \right. \\
                &\left. + \int\limits_{\cosh 1}^{\infty} \tilde{f} \left( t \right) \left( t^2 - \cosh^2 t \right)^{\frac{\alpha}{2} - 1} \mathrm{d}t \right].
			\end{align*}
			From the analysis done for $d \left( 0, \xi \right) \geq 1$, we have that
			$$\int\limits_{1}^{\infty} \tilde{f} \left( t \right) \left( t^2 - \cosh^2 t \right)^{\frac{\alpha}{2} - 1} \mathrm{d}t \leq C \| f \|_{L^{p, 1}_{\mu_1, \mu_2} \left( \mathbb{H}^n \right)}.$$
			We now consider the first term above. As in Case III(A), let us substitute $t^2 = x \cosh^2 d \left( 0, \xi \right)$ and simplify (by replacing powers of $\cosh d \left( 0, \xi \right)$ with constants wherever required) to obtain
			\begin{align*}
				&\sinh^{\nu_1} d \left( 0, \xi \right) \int\limits_{\cosh d \left( 0, \xi \right)}^{\cosh 1} \tilde{f} \left( t \right) \left( t^2 - \cosh^2 d \left( 0, \xi \right) \right)^{\frac{\alpha}{2} - 1} \mathrm{d}t \\
                &\leq C \sinh^{\nu_1} d \left( 0, \xi \right) \int\limits_{1}^{\frac{\cosh^2 1}{\cosh^2 d \left( 0, \xi \right)}} \tilde{f} \left( \sqrt{x} \cosh d \left( 0, \xi \right) \right) \left( x - 1 \right)^{\frac{\alpha}{2} - 1} x^{- \frac{1}{2}} \mathrm{d}x.
			\end{align*}
			Now, using Equation \eqref{MainLemma2Equation1} of Lemma \ref{MainLemma2}, we get
			\begin{align*}
				&\sinh^{\nu_1} d \left( 0, \xi \right) \int\limits_{\cosh d \left( 0, \xi \right)}^{\cosh 1} \tilde{f} \left( t \right) \left( t^2 - \cosh^2 d \left( 0, \xi \right) \right)^{\frac{\alpha}{2} - 1} \mathrm{d}t \\
				&\leq C \sinh^{\nu_1} d \left( 0, \xi \right) \left( \int\limits_{1}^{\frac{\cosh^2 1}{\cosh^2 d \left( 0, \xi \right)}} \tilde{f} \left( \sqrt{x} \cosh d \left( 0, \xi \right) \right) \left( x - 1 \right)^{\frac{p \left( \alpha - 1 \right) + p}{2} - 1} x^{-\frac{1}{2}} \mathrm{d}x \right)^{\frac{1}{p}}.
			\end{align*}
			Now, we substitute $x \cosh^2 d \left( 0, \xi \right) = \cosh^2 t$ and simplify to obtain
			\begin{align*}
				&\sinh^{\nu_1} d \left( 0, \xi \right) \int\limits_{\cosh d \left( 0, \xi \right)}^{\cosh 1} \tilde{f} \left( t^2 - \cosh^2 d \left( 0, \xi \right) \right)^{\frac{\alpha}{2} - 1} \mathrm{d}t \\
				&\leq C \sinh^{\nu_1} d \left( 0, \xi \right) \cosh^{\frac{1}{p'} - \alpha + 1} d \left( 0, \xi \right) \times \\
                &\left( \int\limits_{d \left( 0, \xi \right)}^{1} \tilde{f} \left( t \right) \left( \cosh^2 t - \cosh^2 d \left( 0, \xi \right) \right)^{\frac{p \left( \alpha - 1 \right) - 1 + p - 1}{2}} \sinh t \ \mathrm{d}t \right)^{\frac{1}{p}}.
			\end{align*}
			Observing that we have $p \left( \alpha - 1 \right) \geq 1$ and $p \geq 1$, and that $\cosh d \left( 0, \xi \right) \geq 1$ behaves as a constant, we get
			\begin{align*}
				&\sinh^{\nu_1} d \left( 0, \xi \right) \int\limits_{\cosh d \left( 0, \xi \right)}^{\cosh 1} \tilde{f} \left( t \right) \left( t^2 - \cosh^2 d \left( 0, \xi \right) \right)^{\frac{\alpha}{2} - 1} \mathrm{d}t \\
                &\leq C \sinh^{\nu_1} d \left( 0, \xi \right) \left( \int\limits_{d \left( 0, \xi \right)}^{1} \tilde{f} \left( t \right) \sinh^{\mu_1 + n - 1} t \sinh^{\mu_2 + p - 1} t \ \mathrm{d}t \right)^{\frac{1}{p}}.
			\end{align*}
			We notice that since $\nu_1 \geq 0$, we have $\sinh^{\nu_1} d \left( 0, \xi \right) \leq \sinh^{\nu_1} t$, whenever $t \geq d \left( 0, \xi \right)$. Also, due to the necessary conditions mentioned in Equation \eqref{EndPoint2NecessaryEquation} and the fact that $p = \frac{\mu_1 + \mu_2 + n - 1}{\alpha - 1}$, we have $\nu_1 \geq \frac{1 - \mu_2}{p} - 1$. Combining these observations together with the fact that for small values of $t$, $\cosh t$ behaves as a constant, we get,
			\begin{align*}
				&\sinh^{\nu_1} d \left( 0, \xi \right) \int\limits_{\cosh d \left( 0, \xi \right)}^{\cosh 1} \tilde{f} \left( t \right) \left( t^2 - \cosh^2 d \left( 0, \xi \right) \right)^{\frac{\alpha}{2} - 1} \mathrm{d}t \\
                &\leq C \left( \int\limits_{d \left( 0, \xi \right)}^{1} \tilde{f} \left( t \right) \sinh^{\mu_1 + n - 1} t \cosh^{\mu_2} t \sinh^{p\nu_1 + \mu_2 - 1 + p} t \ \mathrm{d}t \right)^{\frac{1}{p}} \\
				&\leq C \| f \|_{L^{p, 1}_{\mu_1, \mu_2} \left( \mathbb{H}^n \right)}.
			\end{align*}
			This completes the proof!
		\end{proof}
	\section{Weighted \texorpdfstring{$L^p$}{Lp}-\texorpdfstring{$L^p$}{Lp} boundedness}
		\label{LPLPBoundednessSection}
		In this section, we study the $L^p$-mapping properties of the transform $R_{\alpha}$ for generic functions in the weighted Lebesgue space. Particularly, we ask the following question.
		\begin{question}
			\label{LPLPBoundednessQuestion}
			What are the admissible values of $\alpha \geq 1$, $\mu_1, \mu_2, \nu_1, \nu_2 \in \mathbb{R}$ and $p \geq 1$ so that the operator $R_{\alpha}$ is bounded from $L^p_{\mu_1, \mu_2} \left( \mathbb{H}^n \right)$ to $L^p_{\nu_1, \nu_2} \left( \Xi_k \left( \mathbb{H}^n \right) \right)$?
		\end{question}
		Indeed, to talk about the boundedness of the transform $R_{\alpha}$, we first require its existence. Hence, we shall consider the basic assumption that $\alpha \geq 1$, $\mu_1 + \mu_2 > \alpha - n$ and $1 \leq p < \frac{\mu_1 + \mu_2 + n -  1}{\alpha - 1}$, without which Question \ref{LPLPBoundednessQuestion} does not make sense. Further, we notice that for the $L^p$-$L^p$ boundedness to hold for generic functions, we require such boundedness for radial functions as well. That is, the necessary conditions mentioned in Theorem \ref{LPImprovingNecessary} with $r = p$ are also necessary in this case. Particularly, by using $r = p$ in Equation \eqref{LPImprovingNecessaryEquation}, we get
		\begin{equation}
			\label{LPLPNecesssaryEquation}
			\nu_1 > k - n, \nu_1 \geq \mu_1 + k - p \alpha \text{ and } \nu_1 + \nu_2 \leq \mu_1 + \mu_2.
		\end{equation}
		With the necessary conditions made clear, we now proceed to see their sufficiency. We first derive a ``formula" for the transform $R_{\alpha}$ acting on generic functions. For the same, we require the following representation of X-ray transform in $\mathbb{H}^2$ due to Kurusa (see Lemma 1.1 in \cite{KurusaHn}).
		\begin{theorem}[\cite{KurusaHn}]
			\label{XRayH2Formula}
			Let $f: \mathbb{H}^2 \rightarrow \mathbb{C}$ be a function that is integrable on almost every geodesic in $\mathbb{H}^2$, where $\mathbb{H}^2$ is the real hyperbolic space of dimension $2$ described by the Poincar\'{e} ball model. We parametrize a geodesic $\xi$ by $\left( \bar{\omega}, h \right) \in \mathbb{S}^{1} \times \left( 0, \infty \right)$, where $\bar{\omega}$ {is} the direction of the foot of the perpendicular to $\xi$, and $h$ is the hyperbolic distance of $\xi$ from $0$, the hyperbolic origin. Then, the X-ray transform of $f$ is given by
			\begin{equation}
				\label{XRayH2}
				\begin{aligned}
					&Rf \left( \bar{\omega}, h \right) = \int\limits_{{\mathbb{S}^{1}_{\left\langle \omega, \bar{\omega} \right\rangle > \tanh h}}} f \left( \omega, \frac{1}{2} \tanh^{-1} \left( \frac{\tanh h}{\left\langle \omega, \bar{\omega} \right\rangle} \right) \right) \frac{\left( \coth^2 h \left\langle \omega, \bar{\omega} \right\rangle^2 - 1 \right)^{-1}}{\sinh h} \ \mathrm{d}\omega,
				\end{aligned}
			\end{equation}
			where $\mathbb{S}^{1}_{\left\langle \omega, \bar{\omega} \right\rangle > \tanh h} = \left\lbrace \omega \in \mathbb{S}^1 | \left\langle \omega, \bar{\omega} \right\rangle > \tanh h \right\rbrace$.
		\end{theorem}
		Using Theorem \ref{XRayH2Formula}, we get a representation of $R_{\alpha}f$ for functions $f$ that may not necessarily be radial.
		\begin{theorem}
			\label{FormulaGeneralFunctions}
			Let $f: \mathbb{H}^n \rightarrow \mathbb{C}$ be a ``nice" function. Then, for any $\left( \xi, u \right) \in \Xi_k \left( \mathbb{H}^n \right)$ with $u \neq 0$, we have,
			\begin{equation}
				\label{FormulaGeneralFunctionsEquation1}
                \begin{aligned}
                    R_{\alpha}f \left( \xi, u \right) = &\int\limits_{\mathbb{S}^k_{\xi \oplus u}, \left\langle \omega, \frac{u}{\| u \|} \right\rangle > \tanh \| u \|} f \left( \omega, \tanh^{-1} \left( \frac{\tanh \| u \|}{\left\langle \omega, \frac{u}{\| u \|} \right\rangle} \right) \right) \times \\
                    &\frac{\left( \coth^2 \| u \| \left\langle \omega, \frac{u}{\| u \|} \right\rangle^2 - 1 \right)^{- \left( \frac{\alpha + 1}{2} \right)} \left( 1 - \left\langle \omega, \frac{u}{\| u \|} \right\rangle^2 \right)^{\frac{\alpha - k}{2}}}{\sinh \| u \|} \mathrm{d}\omega,
                \end{aligned}
			\end{equation}
			where, $\xi \oplus u$ denotes the $\left( k + 1 \right)$-dimensional subspace of $\mathbb{R}^n$ spanned by $\xi$ and $u$ and
			$$\mathbb{S}^{k}_{\xi \oplus u, \left\langle \omega, \frac{u}{\| u \|} \right\rangle > \tanh \| u \|} := \left\lbrace \omega \in \mathbb{S}^{n - 1} \cap \xi \oplus u \Bigg| \left\langle \omega, \frac{u}{\| u \|} \right\rangle > \tanh \| u \| \right\rbrace$$
			Further, if $g \in SO \left( n \right)$ and $h > 0$ are such that $\left( \xi, u \right) = \left( g \cdot \mathbb{R}^k, hg \cdot e_{k + 1} \right)$, then, we have
			\begin{equation}
				\label{FormulaGeneralFunctionsEquation2}
                \begin{aligned}
                    R_{\alpha}f \left( \xi, u \right) = &\int\limits_{\mathbb{S}^k_{\left\langle \omega, e_{k + 1} \right\rangle > \tanh h}} f_g \left( \omega, \tanh^{-1} \left( \frac{\tanh h}{\left\langle \omega, e_{k + 1} \right\rangle} \right) \right) \times \\
                    &\frac{\left( \coth^2 h \left\langle \omega, e_{k + 1} \right\rangle^2 - 1 \right)^{- \left( \frac{\alpha + 1}{2} \right)} \left( 1 - \left\langle \omega, e_{k + 1} \right\rangle^2 \right)^{\frac{\alpha - k}{2}}}{\sinh h} \mathrm{d}h,    
                \end{aligned}
			\end{equation}
			where, $f_g \left( x \right) = f \left( g \cdot x \right)$.
		\end{theorem}
		\begin{proof}
			We begin by using the polar decomposition on a $k$-dimensional totally geodesic submanifold $\left( \xi, u \right)$ by keeping $u$ as the center of the decomposition. Then, we get,
			\begin{align*}
				R_{\alpha}f \left( \xi, u \right) &= \int\limits_{\mathbb{S}^{n - 1} \cap \xi} \int\limits_{0}^{\infty} f \left( \exp_u \left( t\sigma \right) \right) \sinh^{\alpha - 1} t \ \mathrm{d}t \\
				&= \left| \mathbb{S}^{n - 1} \cap \xi \right| \int\limits_{SO_k \left( \xi \right)} \int\limits_{0}^{\infty} f_g \left( \exp_u \left( t\sigma_0 \right) \right) \sinh^{\alpha - 1} t \ \mathrm{d}t,
			\end{align*}
			where, $SO_k \left( \xi \right)$ is the rotation group of the subspace $\xi$, and $\sigma_0 \in \mathbb{S}^{n - 1} \cap \xi$ is a fixed direction. Let $R_{\sigma_0} \in SO_k \left( \xi \right)$ be a rotation that takes $\sigma_0$ to $-\sigma_0$. Then, it is easy to observe that
			\begin{align*}
				R_{\alpha}f \left( \xi, u \right) &= \left| \mathbb{S}^{n - 1} \cap \xi \right| \int\limits_{\faktor{SO_k \left( \xi \right)}{\left\lbrace I, R_{\sigma_0} \right\rbrace}} \int\limits_{- \infty}^{\infty} f_g \left( \exp_u \left( t \sigma_0 \right) \right) \sinh^{\alpha - 1} t \ \mathrm{d}t.
			\end{align*}
			We now observe that the inner integral is the X-ray transform of the function $f_g \left( \cdot \right) \sinh^{\alpha - 1} d \left( u, \cdot \right)$ along the geodesic through $u$ in the direction $\sigma_0$. Then, using Equation \eqref{XRayH2} of Theorem \ref{XRayH2Formula}, and simplifying, we get
			\begin{align*}
				&R_{\alpha}f \left( \xi, u \right) \\
                &= \left| \mathbb{S}^{n - 1} \cap \xi \right| \int\limits_{\faktor{SO_k \left( \xi \right)}{\left\lbrace I, R_{\sigma_0} \right\rbrace}} \int\limits_{\mathbb{S}^1_{\left\langle \omega, \frac{u}{\| u \|} \right\rangle > \tanh \| u \|}} f_g \left( \omega, \tanh^{-1} \left( \frac{\tanh \| u \|}{\left\langle \omega, \frac{u}{\| u \|} \right\rangle} \right) \right)  \times \\
				&\left( \frac{\tanh \| u \|}{\left\langle \omega, \frac{u}{\| u \|} \right\rangle} \right)^{\alpha - 1} \left( 1 - \left( \frac{\tanh^2 \| u \|}{\left\langle \omega, \frac{u}{\| u \|} \right\rangle} \right)^2 \right)^{- \frac{\alpha - 1}{2}} \times \\
                &\left( \cosh^2 \| u \| \left\langle \omega, \frac{u}{\| u \|} \right\rangle^2 - 1 \right)^{-1} \frac{\left( 1 - \left\langle \omega, \frac{u}{\| u \|} \right\rangle^2 \right)^{\frac{\alpha - 1}{2}}}{\sinh \| u \|} \mathrm{d}\omega \mathrm{d}g.
			\end{align*}
			Equation \eqref{FormulaGeneralFunctionsEquation1} is obtained at once by using the polar decomposition on the sphere $\mathbb{S}^k_{\xi \oplus u}$ with respect to $\frac{u}{\| u \|}$. Also, Equation \eqref{FormulaGeneralFunctionsEquation2} is only a restatement of Equation \eqref{FormulaGeneralFunctionsEquation1} by using isometries of $\mathbb{H}^n$.
		\end{proof}
		We are now in a position to state and prove the weighted $L^p$-$L^p$ boundedness of the operator $R_{\alpha}$.
		\begin{theorem}
			\label{LPLPBoundednessTheorem}
			Let $\alpha \geq 1$, $\mu_1 + \mu_2 > \alpha - n$ and $1 \leq p \leq \frac{\mu_1 + \mu_2 + n + 1}{\alpha + 1}$. Then, we have the following.
			\begin{enumerate}
				\item[\mylabel{LPLPCase1}{(A)}] If $\nu_1, \nu_2 \in \mathbb{R}$ are such that $\nu_1 + \nu_2 > p - n - 1$, then the operator $R_{\alpha}$ is bounded from $L^p_{\mu_1, \mu_2} \left( \mathbb{H}^n \right)$ to $L^p_{\nu_1, \nu_2} \left( \Xi_k \left( \mathbb{H}^n \right) \right)$ if $\nu_1$ and $\nu_2$ satisfy Equation \eqref{LPLPNecesssaryEquation} and we have $\nu_1 + \nu_2 < p \left( \alpha + 1 \right) - n - 1$. Moreover, in this case, we have,
				\begin{equation}
					\label{RAlphaOperatorNormLPLPCase1}
                    \begin{aligned}
                        &\| R_{\alpha} \|_{L^p_{\mu_1, \mu_2} \left( \mathbb{H}^n \right) \rightarrow L^p_{\nu_1, \nu_2} \left( \Xi_k \left( \mathbb{H}^n \right) \right)} \\
                        &\leq \left( \frac{\left| \mathbb{S}^{n - k - 1} \right|}{\left| \mathbb{S}^{n - 1} \right|} \right)^{\frac{1}{p}} \left| \mathbb{S}^{k - 1} \right| \frac{\Gamma \left( \frac{\nu_1 + n - k}{2p} \right) \Gamma \left( \frac{\alpha + 1}{2} - \frac{\nu_1 + \nu_2 + n + 1}{2p} \right)}{\Gamma \left( \frac{\alpha + 1}{2} + \frac{\nu_2 + k + 1}{2p} \right)}.
                    \end{aligned}
				\end{equation}
				\item[\mylabel{LPLPCase2}{(B)}] If $\nu_1, \nu_2 \in \mathbb{R}$ are such that $\nu_1 + \nu_2 \leq p - n - 1$, then the operator $R_{\alpha}$ is bounded from $L^p_{\mu_1, \mu_2} \left( \mathbb{H}^n \right)$ to $L^p_{\nu_1, \nu_2} \left( \Xi_k \left( \mathbb{H}^n \right) \right)$ if and only if $\nu_1$ and $\nu_2$ satisfy the conditions mentioned in Equation \eqref{LPLPNecesssaryEquation}. Moreover, in this case, we have,
				\begin{equation}
					\label{RAlphaOperatorNormLPLPCase2}
					\| R_{\alpha} \|_{L^p_{\mu_1, \mu_2} \left( \mathbb{H}^n \right) \rightarrow L^p_{\nu_1, \nu_2} \left( \Xi_k \left( \mathbb{H}^n \right) \right)} \leq \left( \frac{\left| \mathbb{S}^{n - k - 1} \right|}{\left| \mathbb{S}^{n - 1} \right|} \right)^{\frac{1}{p}} \left| \mathbb{S}^{k - 1} \right| \frac{\Gamma \left( \frac{\nu_1 + n - k}{2p} \right) \Gamma \left( \frac{\alpha}{2} \right)}{\Gamma \left( \frac{\alpha}{2} + \frac{\nu_1 + n - k}{2p} \right)}.
				\end{equation}
			\end{enumerate}
		\end{theorem}
		\begin{proof}
			Using Equations \eqref{FormulaGeneralFunctionsEquation2} and \eqref{XiKMeasure}, we get
			\begin{align*}
				&\| R_{\alpha}f \|_{L^{p}_{\nu_1, \nu_2} \left( \Xi_k \left( \mathbb{H}^n \right) \right)} = \left| \mathbb{S}^{n - k - 1} \right|^{\frac{1}{p}} \left[ \int\limits_{SO(n)} \int\limits_{0}^{\infty} \left| \int\limits_{\mathbb{S}^{k}_+} f_g \left( \omega, \tanh^{-1} \left( \frac{\tanh h}{\left\langle \omega, e_{k + 1} \right\rangle} \right) \right) \times \right. \right. \\
				&\left. \left. \frac{\left( \coth^2 h \left\langle \omega, e_{k + 1} \right\rangle^2 - 1 \right)^{- \left( \frac{\alpha - 1}{2} \right)} \left( 1 - \left\langle \omega, e_{k + 1} \right\rangle^2 \right)^{\frac{\alpha - k}{2}}}{\sinh h} \chi_{\left( \tanh h, 1 \right)} \left( \left\langle \omega, e_{k + 1} \right\rangle \right) \mathrm{d}\omega \right|^p \right. \\
                &\left. \sinh^{\nu_1 + n - k - 1} h \cosh^{\nu_2 + k}  h \ \mathrm{d}h \mathrm{d}g \right]^{\frac{1}{p}}.
			\end{align*}
			By an application of Minkowski's inequality, we get
			\begin{align*}
				&\| R_{\alpha}f \|_{L^{p}_{\nu_1, \nu_2} \left( \Xi_k \left( \mathbb{H}^n \right) \right)} \\
                &\leq \left| \mathbb{S}^{n - k - 1} \right|^{\frac{1}{p}} \int\limits_{\mathbb{S}^k_+} \left[ \int\limits_{SO \left( n \right)} \int\limits_{0}^{\tanh^{-1} \left( \left\langle \omega, e_{k + 1} \right\rangle \right)} \left| f_g \left( \omega, \tanh^{-1} \left( \frac{\tanh h}{\left\langle \omega, e_{k + 1} \right\rangle} \right) \right) \right|^p  \times \right. \\
				&\left. \left( \coth^2 h \left\langle \omega, e_{k + 1} \right\rangle^2 - 1 \right)^{-p \left( \frac{\alpha + 1}{2} \right)} \sinh^{\nu_1 + n - k - 1 - p} h \cosh^{\nu_2 + k} h \ \mathrm{d}h \mathrm{d}g \right]^{\frac{1}{p}} \times \\
                &\left( 1 - \left\langle \omega, e_{k + 1} \right\rangle^2 \right)^{\frac{\alpha - k}{2}} \mathrm{d}\omega.
			\end{align*}
			Next, we substitute $\tanh h = \left\langle \omega, e_{k + 1} \right\rangle \tanh t$ and simplify to obtain,
			\begin{align*}
				\| R_{\alpha}f \|_{L^{p}_{\nu_1, \nu_2} \left( \Xi_k \left( \mathbb{H}^n \right) \right)} &\leq \left| \mathbb{S}^{n - k - 1} \right|^{\frac{1}{p}} \int\limits_{\mathbb{S}^k_+} \left\langle \omega, e_{k + 1} \right\rangle^{\frac{\nu_1 + n - k}{p} - 1} \left( 1 - \left\langle \omega, e_{k + 1} \right\rangle^2 \right)^{\frac{\alpha - k}{2}} \times \\
                &\left[ \int\limits_{SO(n)} \int\limits_{0}^{\infty} \left| f_g \left( \omega, t \right) \right|^p \sinh^{\mu_1 + n - 1} t \cosh^{\mu_2} t \times \right. \\
				&\left. \tanh^{p \alpha + \nu_1 - k - \mu_1} t \cosh^{p \left( \alpha + 1 \right) - \mu_1 - \mu_2 - n - 1} t \times \right. \\
                &\left.\left( 1 - \left\langle \omega, e_{k + 1} \right\rangle^2 \tanh^2 t \right)^{\frac{p - \nu_1 - \nu_2 - n - 1}{2}} \mathrm{d}t \mathrm{d}g \right]^{\frac{1}{p}} \mathrm{d}\omega.
			\end{align*}
			Due to Equation \eqref{LPLPNecesssaryEquation}, we have $\tanh^{p \alpha + \nu_1 - k - \mu_1} \leq 1$. Also, since $p \leq \frac{\mu_1 + \mu_2 + n + 1}{\alpha + 1}$, we have $\cosh^{p \left( \alpha + 1 \right) - \mu_1 - \mu_2 - n - 1 } t \leq 1$. That is, we get,
			\begin{equation}
				\label{LPLPIntermediate}
				\begin{aligned}
					\| R_{\alpha}f \|_{L^{p}_{\nu_1, \nu_2} \left( \Xi_k \left( \mathbb{H}^n \right) \right)} &\leq \left| \mathbb{S}^{n - k - 1} \right|^{\frac{1}{p}} \int\limits_{\mathbb{S}^k_+} \left\langle \omega, e_{k + 1} \right\rangle^{\frac{\nu_1 + n - k}{p} - 1} \left( 1 - \left\langle \omega, e_{k + 1} \right\rangle^2 \right)^{\frac{\alpha - k}{2}} \\
                    &\left[ \int\limits_{SO(n)} \int\limits_{0}^{\infty} \left| f_g \left( \omega, t \right) \right|^p \sinh^{\mu_1 + n - 1} t \cosh^{\mu_2} t \times \right. \\
					&\left. \left( 1 - \left\langle \omega, e_{k + 1} \right\rangle^2 \tanh^2 t \right)^{\frac{p - \nu_1 - \nu_2 - n - 1}{2}} \mathrm{d}t \mathrm{d}g \right]^{\frac{1}{p}} \mathrm{d}\omega.
				\end{aligned}
			\end{equation}
			\begin{enumerate}
				\item[\mylabel{LPLPCase1Proof}{(A)}] Let us now consider $\nu_1 + \nu_2 > p - n - 1$. Then, we have
				$$\left( 1 - \lambda \left\langle \omega, e_{k + 1} \right\rangle^2 \tanh^2 t \right)^{\frac{p - \nu_1 - \nu_2 - n - 1}{2}} \leq \left( 1 - \left\langle \omega, e_{k + 1} \right\rangle^2 \right)^{\frac{p - \nu_1 - \nu_2 - n - 1}{2}}.$$
				Hence, we get from Equation \eqref{LPLPIntermediate},
				\begin{align*}
					\| R_{\alpha}f \|_{L^p_{\nu_1, \nu_2} \left( \Xi_k \left( \mathbb{H}^n \right) \right)} &\leq \left( \frac{\left| \mathbb{S}^{n - k - 1} \right|}{\left| \mathbb{S}^{n - 1} \right|} \right)^{\frac{1}{p}} \| f \|_{L^p_{\mu_1, \mu_2} \left( \mathbb{H}^n \right)} \times \\
                    &\int\limits_{\mathbb{S}^k_+} \left\langle \omega, e_{k + 1} \right\rangle^{\frac{\mu_1 + n - k}{p} - 1} \left( 1 - \left\langle \omega, e_{k + 1} \right\rangle^2 \right)^{\frac{\alpha - k + 1}{2} - \frac{\mu_1 + \mu_2 + n + 1}{2p}} \mathrm{d}\omega.
				\end{align*}
				From the polar decomposition of $\mathbb{S}^{k}_+$, owing to the assumption that $\nu_1 + \nu_2 < p \left( \alpha + 1 \right) - n - 1$ we get
				\begin{align*}
					&\int\limits_{\mathbb{S}^k_+} \left\langle \omega, e_{k + 1} \right\rangle^{\frac{\mu_1 + n - k}{p} - 1} \left( 1 - \left\langle \omega, e_{k + 1} \right\rangle^2 \right)^{\frac{\alpha - k + 1}{2} - \frac{\nu_1 + \nu_2 + n + 1}{2p}} \mathrm{d}\omega \\
                    &= \left| \mathbb{S}^{k - 1} \right| \frac{\Gamma \left( \frac{\nu_1 + n - k}{2p} \right) \Gamma \left( \frac{\alpha + 1}{2} - \frac{\nu_1 + \nu_2 + n + 1}{2p} \right)}{\Gamma \left( \frac{\alpha + 1}{2} - \frac{\nu_2 + k + 1}{2p} \right)}.
				\end{align*}
				This completes the proof of \ref{LPLPCase1Proof}.
				\item[\mylabel{LPLPCase2Proof}{(B)}] Next, we take up the case when $\nu_1 + \nu_2 \leq p - n - 1$. Here, we use the estimate
				$$\left( 1 - \left\langle \omega, e_{k + 1} \right\rangle^2 \tanh^2 t \right)^{\frac{p - \nu_1 - \nu_2 - n - 1}{2}} \leq 1.$$
				Then, we get from Equation \eqref{LPLPIntermediate} and polar decomposition on $\mathbb{S}^k_+$,
				\begin{align*}
					\| R_{\alpha}f \|_{L^p_{\nu_1, \nu_2} \left( \Xi_k \left( \mathbb{H}^n \right) \right)} &\leq \left( \frac{\left| \mathbb{S}^{n - k - 1} \right|}{\left| \mathbb{S}^{n - 1} \right|} \right)^{\frac{1}{p}} \left| \mathbb{S}^{k - 1} \right| \frac{\Gamma \left( \frac{\nu_1 + n - k}{2p} \right) \Gamma \left( \frac{\alpha}{2} \right)}{\Gamma \left( \frac{\alpha}{2} + \frac{\nu_1 + n - k}{2p} \right)} \| f \|_{L^p_{\mu_1, \mu_2} \left( \mathbb{H}^n \right)}.
				\end{align*}
			\end{enumerate}
		\end{proof}
	\section{Conclusion}
		\label{ConclusionSection}
		In this article, we have studied the mapping properties of a weighted analogue of the $k$-dimensional totally geodesic Radon transform on $\mathbb{H}^n$. As mentioned within Sections \ref{LPImprovingSection} and \ref{LPLPBoundednessSection}, certain aspects of this study remain open. We formulate them as questions to be undertaken at a later stage.
		\begin{question}
			\label{OpenQuestion1}
			Does Theorem \ref{LPLInfinitySufficient} hold with $\nu_1 = 0$ and $p = \frac{\mu_1 + n}{\alpha}$?
		\end{question}
		\begin{question}
			\label{OpenQuestion2}
			Does Theorem \ref{LPLPBoundednessTheorem} hold for $\frac{\mu_1 + \mu_2 + n + 1}{\alpha + 1} < p < \frac{\mu_1 + \mu_2 + n - 1}{\alpha - 1}$?
		\end{question}
		As mentioned earlier, currently, we do not have any technique to tackle the problem. However, we believe that the questions above are answered affirmatively.
		\begin{question}
			\label{OpenQuestion3}
			What is the operator norm of $R_{\alpha}$ as an operator from $L^p_{\mu_1, \mu_2} \left( \mathbb{H}^n \right)$ to $L^p_{\nu_1, \nu_2} \left( \Xi_k \left( \mathbb{H}^n \right) \right)$?
		\end{question}
		Indeed, we have an estimate for the operator norm of $R_{\alpha}$. However, we notice that the two bounds do not match, and therefore we do not have a guess to begin with. Further, if one tries to follow the techniques of Rubin (see p. 3460 of \cite{RubinKPlane}), one ends up with integrals involving the Appell's hypergeometric functions. This suggests that there is a need to study the Appell's functions in the context of evaluating their integrals and obtaining various transformation of variable formulae.
		
		We also would like to comment that in the real hyperbolic space, the transforms $R_{\alpha}$ need not be necessarily $L^p$-improving. This was shown in \cite{SelfLPImproving} (p. 71--72). We consider only $L^p$-improving mapping properties for the sake of ease and convenience to use Young's inequality. The complete understanding of ``$L^p$-decreasing" properties of the transforms $R_{\alpha}$ is yet to be made, and will indeed require techniques different than those discussed here.
		
		Apart from the study conducted in this article, we now mention a few directions that are opened up. First, we would like to stress that the operator given in Equation \eqref{RieszPotentialRTDefinition} requires an independent study. As a first step, one would like to obtain a one-dimensional (integral) representation of the operator acting on radial functions. The authors of this article believe that such a representation should be obtained in terms of compositions of fractional integrals on the half line (for the Euclidean analogue, we refer the reader to \cite{RubinFIRepresentation} for the Riesz potentials). However, any attempt at getting this representation would end up in expressions involving Appell's hypergeometric functions.
		
		Secondly, one may also think of defining a ``singular integral" over totally geodesic submanifolds of $\mathbb{H}^n$. Particularly, given a ``nice" function $f: \mathbb{H}^n \rightarrow \mathbb{C}$, and a function $\Omega: \mathbb{S}^{n - 1} \rightarrow \mathbb{C}$, we define the singular operator
		$$T_{\Omega}f \left( \xi \right) := \int\limits_{\mathbb{R}^k} f \left( g \cdot \exp_0 \left( x \right) \right) \frac{\Omega \left( \frac{x}{\| x \|} \right)}{\sinh^k \left( d \left( 0, \exp_0 \left( x \right) \right) \right)} \mathrm{d}x.$$
		Then, we may proceed to analyse the $L^p$-boundedness of the operator $T_{\Omega}$. We remark here that such operators are not studied for the Euclidean space.

\end{document}